\documentclass[twoside,12pt,leqno]{article}
\advance\topmargin by -\headheight\advance\topmargin by -\headsep
\newcommand{\tmvolxx}{xx}
\newcommand{\tmyearyyyy}{yyyy}

\newcommand{\FirstPageHead}[3]{
{\footnotesize 
\vskip -8mm 
\centerline {Travaux math\'ematiques, \quad 
Volume #1 (#2), 
#3,\quad \copyright\  Universit\'e du Luxembourg}}\vspace{-3mm}}

\usepackage{amsmath}
\usepackage{amsthm}
\usepackage{amssymb}
\usepackage{amscd}
\usepackage{graphicx}
\usepackage{epsfig}
\usepackage[all,cmtip]{xy}
\usepackage[utf8]{inputenc}
\usepackage{enumerate}
\usepackage[toc,page]{appendix}
\usepackage{float}
\usepackage{tkz-euclide}
\usetikzlibrary{decorations.markings,arrows}
\tikzset{
	arrowMe/.style={
		postaction=decorate,
		decoration={
			markings,
			mark=at position .5 with {\arrow[thick]{#1}}
		}
	}
}
\pgfarrowsdeclaredouble{<<s}{>>s}{stealth}{stealth}
\pgfarrowsdeclaretriple{<<<s}{>>>s}{stealth}{stealth}

\DeclareMathOperator{\bq}{b}

\DeclareMathOperator{\Res}{Res}

\DeclareMathOperator{\vol}{vol}

\DeclareMathOperator{\Li}{{Li}}

\numberwithin{equation}{section}
\numberwithin{equation}{section}

\allowdisplaybreaks
\begin{document}
\thispagestyle{empty}
\FirstPageHead{\tmvolxx}{\tmyearyyyy}{\pageref{firstpage}--\pageref{lastpage}}
\label{firstpage}
\newcommand{\relmiddle}[1]{\mathrel{}\middle#1\mathrel{}}
\newcommand{\ul}{\underline}
\newcommand{\eps}{\epsilon}
\newcommand{\MOD}{\mathrm{mod}}
\newcommand{\pil}{\vec}
\newcommand{\word}[1]{#1\index{#1}}
\newcommand{\emphword}[1]{\emph{#1}\index{#1}}
\newcommand{\abs}[1]{\lvert #1 \rvert}
\newcommand{\mgd}[1]{\left\{ #1 \right\}}
\newcommand{\pa}[1]{\left( #1 \right)}
\newcommand{\tpa}[1]{\left\{ #1 \right\}}
\newcommand{\Abs}[1]{\lVert #1 \rVert}
\newcommand{\braket}[2]{\langle #1 | #2 \rangle}
\newcommand{\knot}[1]{\langle #1 \rangle_N}
\newcommand{\pbraket}[2]{( #1 | #2 )}
\newcommand{\bra}[1]{\langle #1 |}
\newcommand{\ket}[1]{| #1 \rangle}
\newcommand{\ps}[3]{{\psi}_{#1,#2}(#3)}
\newcommand{\Spinc}{\text{Spin}^c}
\newcommand{\frakg}{\mathfrak{g}}
\newcommand{\g}{\frakg}
\newcommand{\frakX}{\mathfrak{X}}
\newcommand{\frakp}{\mathfrak{p}}
\newcommand{\h}{\mathfrak{h}}
\newcommand{\fraksl}{\mathfrak{sl}}
\newcommand{\frako}{\mathfrak{o}}
\newcommand{\fraksp}{\mathfrak{sp}}
\newcommand{\gtilde}{\tilde{\frakg}}
\newcommand{\frakh}{\mathfrak{h}}
\newcommand{\fraks}{\mathfrak{s}}
\newcommand{\frakm}{\mathfrak{m}}
\newcommand{\fraku}{\mathfrak{u}}
\newcommand{\gl}{\mathfrak{gl}}
\newcommand{\bbA}{\mathbb{A}}
\newcommand{\bbC}{\mathbb{C}}
\newcommand{\bbE}{\mathbb{E}}
\newcommand{\bbF}{\mathbb{F}}
\newcommand{\bbO}{\mathbb{O}}
\newcommand{\bbZ}{\mathbb{Z}}
\newcommand{\bbN}{\mathbb{N}}
\newcommand{\bbQ}{\mathbb{Q}}
\newcommand{\bbR}{\mathbb{R}}
\newcommand{\bbP}{\mathbb{P}}
\newcommand{\bbH}{\mathbb{H}}

\newcommand{\setZ}{\mathbb{Z}}
\newcommand{\setC}{\mathbb{C}}
\newcommand{\Exp}{\mathrm{Exp}}
\newcommand{\inv}{\varpi}
\newcommand{\del}[1]{\frac{\partial}{\partial {#1}}}
\newcommand{\pd}[2]{\frac{\partial#1}{\partial#2}}
\newcommand{\grad}{\nabla}
\newcommand{\Nabla}{\nabla}
\newcommand{\SW}{\text{SW}}
\newcommand{\ind}{\text{ind}}
\newcommand{\even}{\text{even}}
\newcommand{\frakodd}{\text{odd}}
\newcommand{\Cl}{\text{Cl}}
\newcommand{\cl}{\text{cl}}
\newcommand{\Vol}{\mathrm{Vol}}
\newcommand{\Maps}{\text{Maps}}
\newcommand{\calA}{\mathcal{A}}
\newcommand{\calB}{\mathcal{B}}
\newcommand{\calC}{\mathcal{C}}
\newcommand{\calD}{\mathcal{D}}
\newcommand{\calE}{\mathcal{E}}
\newcommand{\calF}{\mathcal{F}}
\newcommand{\calG}{\mathcal{G}}
\newcommand{\calH}{\mathcal{H}}
\newcommand{\calI}{\mathcal{I}}
\newcommand{\calJ}{\mathcal{J}}
\newcommand{\calL}{\mathcal{L}}
\newcommand{\calM}{\mathcal{M}}
\newcommand{\calN}{\mathcal{N}}
\newcommand{\calP}{\mathcal{P}}
\newcommand{\calS}{\mathcal{S}}
\newcommand{\calT}{\mathcal{T}}
\newcommand{\calU}{\mathcal{U}}
\newcommand{\calV}{\mathcal{V}}
\newcommand{\calX}{\mathcal{X}}
\newcommand{\calZ}{\mathcal{Z}}
\newcommand{\calHom}{\mathcal{H}\mathrm{om}}
\newcommand{\frakt}{\mathfrak{t}}
\newcommand{\so}{\mathfrak{so}}
\newcommand{\su}{\mathfrak{su}}
\newcommand{\SU}{\mathrm{SU}}
\newcommand{\SO}{\mathrm{SO}}
\newcommand{\SL}{\mathrm{SL}}
\newcommand{\PSL}{\mathrm{PSL}}
\newcommand{\spin}{\mathfrak{spin}}
\newcommand{\spinc}{\mathfrak{spin}^c}
\newcommand{\calO}{\mathcal{O}}
\newcommand{\PD}{\text{PD}}
\newcommand{\End}{\text{End}}
\newcommand{\intmult}{\invneg}
\newcommand{\acts}{\circlearrowright}
\newcommand{\bnabla}{\boldsymbol{\nabla}}
\newcommand{\Spin}{\text{Spin}}
\newcommand{\dvol}{d\text{vol}}
\newcommand{\isom}{\cong}
\newcommand{\dbar}{\overline{\partial}}
\newcommand{\dd}{\partial}
\newcommand{\ddx}{\frac{\dd}{\dd x}}
\newcommand{\ddy}{\frac{\dd}{\dd y}}
\newcommand{\ddxi}{\frac{\dd}{\dd x_i}}
\newcommand{\ddxj}{\frac{\dd}{\dd x_j}}
\newcommand{\ddyi}{\frac{\dd}{\dd y_i}}
\newcommand{\unit}{\mathbf{1}}
\newcommand{\bbet}{\unit}
\newcommand{\xbar}{\bar{x}}
\newcommand{\transverse}{\pitchfork}
\newcommand{\tv}{\transverse}
\newcommand{\Map}{\mathrm{Map}}
\newcommand{\ch}{\mathrm{ch}}
\newcommand{\td}{\mathrm{td}}

\newcommand{\YM}{\text{YM}}

\newcommand{\ddz}{\frac{d}{dz}}

\def\Xint#1{\mathchoice
	{\XXint\displaystyle\textstyle{#1}}%
	{\XXint\textstyle\scriptstyle{#1}}%
	{\XXint\scriptstyle\scriptscriptstyle{#1}}%
	{\XXint\scriptscriptstyle\scriptscriptstyle{#1}}%
	\!\int}
\def\XXint#1#2#3{{\setbox0=\hbox{$#1{#2#3}{\int}$}
		\vcenter{\hbox{$#2#3$}}\kern-.5\wd0}}
\def\doubleint{\Xint=}
\def\dashint{\Xint-}
\newcommand{\circint}{\frakoint}
\newcommand{\pt}{\mathrm{pt}}
\newcommand{\Morse}{\text{Morse}}
\newcommand{\cupp}{\smallsmile}
\newcommand{\id}{\mathrm{id}}
\newcommand{\semidirect}{\rtimes}
\newcommand{\Der}{\mathrm{Der}}
\newcommand{\Vir}{\mathrm{Vir}}
\newcommand{\Aut}{\mathrm{Aut}}
\newcommand{\Hom}{\mathrm{Hom}}
\newcommand{\Ext}{\mathrm{Ext}}
\newcommand{\MCG}{\mathrm{MCG}}
\newcommand{\gr}{\mathrm{gr}}
\newcommand{\deter}{\mathrm{det}}
\newcommand{\Ker}{\text{Ker}}
\newcommand{\cs}{\text{cs}}
\newcommand{\Span}{\mathrm{span}}
\newcommand{\im}{\text{im}}
\newcommand{\Fr}{\text{Fr}}
\newcommand{\inter}{\text{int }}
\newcommand{\Vect}{\text{Vect}}
\newcommand{\Conn}{\text{Conn}}
\newcommand{\dAdt}{\frac{\dd A}{\dd t}}
\newcommand{\ddt}{\frac{\dd}{\dd t}}
\newcommand{\LHS}{\mathrm{LHS}}
\newcommand{\RHS}{\mathrm{RHS}}
\newcommand{\tpsi}[3]{\tilde{\psi}'_{#1,#2}(#3)}
\newtheorem{thm}{Theorem}[section]
\newtheorem{theorem}[thm]{Theorem}
\newtheorem{prop}[thm]{Proposition}
\newtheorem{cor}[thm]{Corollary}
\newtheorem{conj}[thm]{Conjecture}
\newtheorem{conjec}[thm]{Conjecture}
\newtheorem{lem}[thm]{Lemma}
\newtheorem{lemma}[thm]{Lemma}
\newtheorem{claim}[thm]{Claim}

\newtheorem{proposition}[thm]{Proposition}

\newtheorem{corollary}[thm]{Corollary}

\newtheorem{conjecture}[thm]{Conjecture}

\theoremstyle{definition}
\newtheorem{remark}[thm]{Remark}
\newtheorem{rem}[thm]{Remark}
\newtheorem{exer}[thm]{Exercise}
\newtheorem{definition}[thm]{Definition}
\newtheorem{defn}[thm]{Definition}
\newtheorem{propdef}[thm]{Proposition/Definition}
\newtheorem{example}[thm]{Example}





\newcommand{\eproof}{\begin{flushright} $\square$ \end{flushright}}

\let\eproof\endproof 

\newcommand{\QED}{\begin{flushright} QED \end{flushright}}
\newcommand{\ex}{\paragraph{{\bf Example.}}}

\newcommand\zkn[1][k]{Z_{#1}^{(n)}}
\newcommand\zk[1][k]{Z_{#1}^{(2)}}
\newcommand{\spane}{\mathop{\fam0 Span}\nolimits}
\newcommand{\re}{\mathop{\fam0 Re}\nolimits}
\newcommand{\Sign}{\mathop{\fam0 Sign}\nolimits}
\newcommand{\expo}{\exp}\newcommand{\cO}{{\mathcal O}}\newcommand{\E}{ {\mathcal E}}
\newcommand{\Sk}{{\mathcal S}}\newcommand{\cZ}{ {\mathcal Z}}
\newcommand{\cV}{\mathop{\fam0 {\tilde {\mathcal V}}}\nolimits}
\newcommand{\tOmega}{\mathop{\fam0 {{{\tilde \Omega}}}}\nolimits}
\newcommand{\bH}{\mathop{\fam0 {{\mathbb H}}}\nolimits}
\newcommand{\tPhi}{\mathop{\fam0 {{{\tilde \Phi}}}}\nolimits}
\newcommand{\tcH}{\mathop{\fam0 {\tilde {\mathcal H}}}\nolimits}
\newcommand{\cD}{\mathop{\fam0 {\mathcal D}}\nolimits}
\newcommand{\indext}{\mathop{\fam0 Index}\nolimits}
\newcommand{\reg}{\mathop{\fam0 reg}\nolimits}
\newcommand{\irr}{\mathop{\fam0 irr}\nolimits}
\newcommand{\ab}{\mathop{\fam0 ab}\nolimits}
\newcommand{\sing}{\mathop{\fam0 sing}\nolimits}
\newcommand{\sign}{\mathop{\fam0 sign}\nolimits}
\newcommand{\alg}{\mathop{\fam0 alg}\nolimits}
\newcommand{\topo}{\mathop{\fam0 top}\nolimits}
\newcommand{\eq}{\mathop{\fam0 eq}\nolimits}
\newcommand{\Td}{\mathop{\fam0 Td}\nolimits}
\newcommand{\Ch}{\mathop{\fam0 Ch}\nolimits}
\newcommand{\R}{{\mathbb R}}
\newcommand{\tR}{{\tilde \R}}
\newcommand{\bC}{{\mathbb C}}
\newcommand{\bR}{{\mathbb R}}
\newcommand{\bT}{{\bar T}}
\newcommand{\C}{\mathbb C}
\newcommand{\A}{{\mathcal A}}
\newcommand{\X}{{\mathcal X}}
\newcommand{\Z}{{\mathbb Z}}
\newcommand{\bZ}{\Z{}}
\newcommand{\f}{{\tilde f}}
\newcommand{\oc}{{\overline c}}
\newcommand{\D}{{\mathcal D}}
\newcommand{\ra}{\mathop{\fam0 \rightarrow}\nolimits}
\newcommand{\CP}{\mathop{\fam0 {\mathbb CP}}\nolimits}
\newcommand{\ipl}{\mathop{\fam0 \langle}\nolimits}
\newcommand{\ipr}{\mathop{\fam0 \rangle}\nolimits}
\newcommand{\sct}{\mathop{\fam0 \Gamma}\nolimits}
\newcommand{\G}{ {\mathcal G}}
\newcommand{\M}{{\mathcal M}}
\newcommand{\ct}{ \llcorner}
\newcommand{\cR}{ {\mathcal R}}
\newcommand{\cA}{ {\mathcal A}}
\newcommand{\cH}{ {\mathcal H}}
\newcommand{\cM}{ {\mathcal M}}
\newcommand{\T}{ {\mathcal T}}
\newcommand{\fg}{ {\mathfrak g}}
\newcommand{\V}{ {\mathcal V}}\newcommand{\N}{ {\mathbb N}}
\newcommand{\tL}{ {\tilde {\mathcal L}}}
\newcommand{\tY}{ {\tilde Y}}
\newcommand{\Top}{ \hat{\text{T}}}
\newcommand{\pop}{ \hat{\text{p}}}
\newcommand{\qop}{ \hat{\text{q}}}
\newcommand{\tD}{ {\tilde D}}
\newcommand{\tG}{ {\tilde G}}
\newcommand{\tW}{ {\tilde W}}
\newcommand{\s}{\sigma}
\newcommand{\tcV}{\tilde \cV}
\newcommand{\bN}{{\mathbf N}}
\newcommand{\trho}{{\tilde \rho}}
\newcommand{\Nablat}{{\mathbf {\hat \nabla}}^{t}}
\newcommand{\Nablae}{{\mathbf {\hat \nabla}}^e}
\newcommand{\Nablaet}{{\mathbf {\hat \nabla}}^{e,t}}
\newcommand{\BTstar}{\star^{\text{\tiny BT}}}
\newcommand{\tBTstar}{{\tilde \star}^{\text{\tiny BT}}}
\newcommand{\mom}{\mathbf p}
\newcommand{\pos}{\mathbf q }
\newcommand{\sfu}{\mathbf{u}}
\newcommand{\sfv}{\mathbf{v}}
\newcommand{\wv}{{\mathbf{w}}}
\newcommand{\pto}{\mathbf{T}}
\newcommand{\qdl}{\Phi_{\bq}}


\newcommand{\plim}[1]{\mathop{\fam0 \stackrel{\textstyle Lim}
		{\stackrel{\textstyle \longrightarrow}{#1}}}}
\newcommand{\lra}{\longrightarrow}

\newcommand{\varplim}{\mathop{\fam0 \varprojlim}\nolimits}
\newcommand{\e}[1]{\mathbf #1}
\newcommand{\Teim}{Teichm{\"u}ller }



\markboth{J.~E. Andersen J.~J.~K. Nissen}{Asymptotic aspects of the Teichmüller TQFT}
$ $
\bigskip

\bigskip

\centerline{{\Large   Asymptotic aspects of the Teichmüller TQFT}}

\bigskip
\bigskip
\centerline{{\large by J{\o}rgen Ellegaard Andersen and Jens-Jakob Kratmann Nissen\footnote{Work supported in part by the center of excellence grant ``Center for Quantum Geometry of Moduli Spaces'' from the Danish National Research Foundation (DNRF95).}}}

\vspace*{.7cm}

\begin{abstract}
	
We calculate the knot invariant coming from the Teichmüller TQFT \cite{AK1}.
Specifically we calculate the knot invariant for the complement of the knot $6_1$ both in the original \cite{AK1} and the new formulation of the Teichmüller TQFT \cite{AK2} for the one-vertex H-triangulation of $(S^3,6_1)$. We show that the two formulations give equivalent answers. Furthermore we apply a formal stationary phase analysis and arrive at the Andersen-Kashaev volume conjecture as stated in \cite[Conj. 1]{AK1}.

Furthermore we calculate the first examples of knot complements in the new formulation showing that the new formulation is equivalent to the original one in all the special cases considered.

Finally, we provide an explicit isomorphism between the Teichmüller TQFT  representation of the mapping class group of a once punctured torus and a representation of this mapping class group on the space of Schwartz class functions on the real line.
\end{abstract}

\pagestyle{myheadings}

\section{Introduction}
Since discovered and axiomatised by Atiyah \cite{MR1001453}, Segal \cite{MR981378} and Witten \cite{W}, Topological Quantum Field Theories (TQFT's) have been studied extensively. 
The first constructions of such theories in dimension $2+1$ was given by Reshetikhin and Turaev \cite{Turaev, RT1, RT2} who obtained TQFT's through surgery and the combinatorial framework of Kirby calculus, and by Turaev and Viro \cite{TV} using the framework of triangulations and Pachner moves. 
In both constructions the central algebraic ingredients comes from the category of finite dimensional representation of the quantum group $U_{q}(\mathfrak{sl}(2,\C))$ at roots of unity. 
Subsequently Blanchet, Habegger, Masbaum and Vogel gave a pure topological construction using Skein theory \cite{BHMV1,BHMV2}. Recently it has been established by the first author and Ueno that this TQFT is equivalent to the one coming from conformal field theory \cite{AU1,AU2,AU3,AU4} and further by the work of Laszlo \cite{La} in the higher genus case with no marked point and the first author and Egsgaard \cite{AE} in genus zero with marked points (for certain labels), that these TQFT's can be studied from the point of view of geometric quantization of the compact moduli space of flat SU$(2)$ connections. The first author has extensively studied the asymptotics of this TQFT using this quantization of moduli spaces approach to this theory \cite{A1, A2, AGr1, AH1, AMU, A3, A4, A5, AGa1, AB1, A6, AGL, AHi,  A7, AHJMMc}.

A new line of development was initiated by Kashaev in \cite{K6j} where a state sum invariant of links in $3$-manifolds was defined by using the combinatorics of charged triangulations. Here the charges are algebraic versions of dihedral angles of ideal hyperbolic tetrahedra in finite cyclic groups. The approach was subsequently developed further by Baseilhac, Benedetti and by
Geer, Kashaev and Turaev \cite{BB,GKT}. 

New challenges appear when one tries to construct combinatorial versions of Chern--Simons theory with non-compact gauge group such as $\text{PSL}(2,\R)$, which is the isometry group of $2$-dimensional hyperbolic space. When one consideres the corresponding classcial moduli space of flat $\text{PSL}(2,\R)$-connections on a two dimensional surface, a connected component is identified with Teichmüller space, hence this Chern--Simons theory deserves the name Teichm\"{u}ller TQFT.

Quantum Teichmüller theory corresponds to a specific classes of unitary mapping class representations on infinite dimensional Hilbert spaces \cite{K1,CF}. 
Based on quantum Teichmüller theory several formal state-integral partition functions have been studied by Hikami, Dimofte, Gukov, Lenells, Zagier, Dijkgraaf, Fuji, Manabe \cite{HI1,HI2,DGLZ,DFM} with the view to approach the Teichm\"{u}ller TQFT. 
The question about convergence of the studied integrals however remained open until a mathematical rigorous version of Teichmüller TQFT was suggested by the first author and Kashaev in \cite{AK1}. See also \cite{AK1a,AK1b}. The convergence property of the Teichmüller TQFT is a property of the underlying combinatorial setting. An extra structure on the triangulations called a shape structure is imposed where each tetrahedron carries dihedral angles of an ideal hyperbolic tetrahedron. The dihedral angles provide absolute convergence and moreover they implement the complete symmetry with respect to change of edge orientation. 
The positivity condition of dihedral angles seems to impose restrictions on the construction of topologically invariant partition functions. 
In \cite{arXiv:1210.8393} Kashaev, Luo and Vartanov suggests a TQFT of Turaev--Viro type based on the combinatorics of shaped triangulations. As the absolute convergence of the partition function in this model is also based on positivity of dihedral angles, it is similar to the Teichmüller TQFT. A consequence is that as in the case of the Teichmüller TQFT the $2-3$ Pachner move is not immidiately always applicable. However, in this model no other topological restrictions are needed.
A new formulation of the Teichmüller TQFT was suggested in \cite{AK2}. In the new formulation of Teichmüller TQFT both the $2-3$ and $3-2$ Pachner moves are applicable and as in the case of the TQFT of Turaev--Viro type \cite{arXiv:1210.8393} no other topological restrictions are needed. 

Recently in \cite{AK3} the first author of this paper and Kashaev have constructed quantum Chern--Simons theory for $\text{PSL}(2,\C)$ for all non-negative integer levels $k$ and furthermore understood how it relates to geometric quantization of $\text{PSL}(2,\C)$-moduli spaces. They have proposed a general scheme which just requires a Pontryagin self-dual locally compact group, which is expected to lead to the construction of the $\text{SL}(
n,\C)$ quantum Chern-Simons theory for all non-negative integer levels $k$. From the geometric quantization of moduli spaces viewpoint, the corresponding representations of the mapping class groups have been constructed in \cite{AGa2}. This work is closely related to the work of Dimofte \cite{Di} on the physics side.
The Teichmüller TQFT is the complex quantum Chern--Simons TQFT for $\text{PSL}(2,\C)$ at level $k=1.$ See also \cite{AM} in this volume.

\subsection*{Outline}
We will review the construction of the charged tetrahedral operators which originates from Kashaev's quantization of Teichmüller space. The main ingredients in this theory are Penner's cell decomposition of decorated Teichmüller space and the associated Ptolemy groupoid \cite{P1} and Faddeev's quantum dilogarithm \cite{F} which allows us to change polarization on Teichmüller space.  

We recall how the partition function from the Teichmüller TQFT is defined using tetrahedral operators in both the original version \cite{AK1} and in the new formulation \cite{AK2}. 

We will then prove the equivalence of the two versions by direct calculations in several cases and elaborate on the Andersen-Kashaev \textit{volume conjecture} arising in \cite[Conj. 1.]{AK1}.  

Following this, we will investigate the Teichmüller TQFT representation of the genus one, one parked point, mapping class group and prove that it is equivalent to an action of this same mapping class group acting on the space of Schwartz class functions on the real line.

\subsection*{Acknowledgements} We would like to thank Rinat Kashaev for many interesting discussions.

\section{Teichmüller Space}
As mentioned in the introduction quantum Chern--Simons theory with non-compact gauge group is of great interest. The gauge group in Teichmüller theory is $\text{PSL}(2,\R)$ which is the isometry group of 2 dimensional hyperbolic space. 

Let $M$ be a 3-manifold. Recall that the classical phase space of Chern--Simons theory with gauge group $\text{PSL}(2,\R)$ is given by the moduli space of flat connections 
\begin{equation*}
\calM = \Hom(\pi_1(M),\text{PSL}(2,\R))/\text{PSL}(2,\R).
\end{equation*}
It is natural to start from $M=\Sigma \times \R$, where $\pi_1(M)=\pi_1(\Sigma)$, so we can talk about the moduli space of flat connections on the surface
\begin{equation*}
\calM_{\Sigma} = \Hom(\pi_1(\Sigma),\text{PSL}(2,\R))/\text{PSL}(2,\R).
\end{equation*} 
We can write the moduli space as the disjoint union of connected components
\begin{equation*}
\calM_{\Sigma} = \bigsqcup_{-\chi(\Sigma)\leq k\leq \chi(\Sigma)}(\calM_{\Sigma})_k.
\end{equation*}
Teichmüller space is given by the connected component with the maximal index 
\begin{equation*}
\calT_{\Sigma}=(\calM_{\Sigma})_{-\chi(\Sigma)}.
\end{equation*} 
Let $\Sigma=\Sigma_{g,s}$ be a surface of finite type, i.e. $\Sigma$ is an oriented genus $g$ surface with $s$ boundary components or punctures.  
Then, topologically, Teichmüller space is an open ball of dimension $6g-6+2s$, i.e.
\begin{equation*}
\calT_{\Sigma}\cong \R^{6g-6+2s}.
\end{equation*}
Recall that $\calT_{\Sigma}$ is a symplectic space, where the symplectic structure is given by the Weil--Petersson symplectic form. 

\subsection{Penner coordinate system on $\tilde{\calT}_{\Sigma}$} 
Let $\Sigma$ be an oriented genus $g$ surface with $s>0$ punctures and Euler characteristic $2-2g-s<0.$ We denote the set of punctures
\begin{equation*}
V:=\mgd{P_1,\dots,P_s}.
\end{equation*}  
\begin{defn}
	A homotopy class of a path running between $P_i$ and $P_j$ is called an ideal arc. A set of ideal arcs obtained by taking a family $X$ of disjointly embedded ideal arcs in $\Sigma$ running between punctures and subject to the condition that each component of $\Sigma\backslash X$ is a triangle is called an ideal triangulation. Let $\Delta_\Sigma$ denote the set of all ideal triangulations.
\end{defn}

Now take an ideal triangulation $\tau \in \Delta_{\Sigma}$ and calculate all $\lambda$-lengths with respect to a fixed configuration of horocycles. Let $E(\tau)$ denote the set of edges in $\tau$. We impose the equivalence relation 
\begin{equation*}
\lambda \sim \lambda' : E(\tau) \to \R_{>0},
\end{equation*}
if there exists $f: V \to \R_{>0}$ such that 
\begin{equation*}
\lambda'(e)=f(v_1)f(v_2)\lambda(e) \quad e\in E(\tau),
\end{equation*}
where $v_1$ and $v_2$ are endpoints of $e$.
Counting the number of edges and vertices in an ideal triangulation establishes the following
\begin{equation*}
\R^{E(\tau)}_{>0}/\R^{V}_{>0} \cong \R^{6g-6+2s}.
\end{equation*}
The $\lambda$-lengths parametrizes the decorated Teichmüller space $\tilde{\calT}_{\Sigma}$, which is a principal $\R^s_{>0}$ foliated fibration $\phi: \tilde{\calT}_{\Sigma} \to {\calT}_{\Sigma}$, where the fiber over a point of $\calT_{\Sigma}$ is the space of all horocycles about the punctures of $\Sigma$.

\begin{thm}[Penner]
	\begin{enumerate}[(a)]
		\item As a topological space the decorated Teichmüller space is homeomorphic to the set of positive numbers on edges given by 
		$\lambda$-lengths
		\begin{equation*}
		\tilde{\calT}_{\Sigma} \cong \R^{E(\tau)}_{>0}.
		\end{equation*}
		\item Using the map $\phi$ which forgets the horocycles we can pull back the Weil--Petersson symplectic form to the decorated Teichmüller space. The pull-back satisfies the formula
		\begin{equation*} 
		\phi^*\omega_{WP}=\sum_{\begin{tikzpicture}[scale=.2,baseline]
			\draw (-1,0)--(0,2)--(1,0)--cycle;
			\draw (-1,0)--(1,0);
			\draw (0,0) node[anchor=north]{$c$};
			\draw (-.5,.5) node[anchor=south east]{$a$};
			\draw (.5,.5) node[anchor=south west]{$b$};
			\end{tikzpicture}
		} \frac{da \wedge db}{ab}+ \frac{db \wedge dc}{bc}+ \frac{dc \wedge da}{ca}.
		\end{equation*}
		\item The mapping class group is contained in the groupoid generated by Ptolemy transformations.
		Suppose $a,b,c,d,e \in \tau \in \Delta_{\Sigma}$ are such that $\mgd{a,b,e}$ and $\mgd{c,d,e}$ bound distinct triangles. The operation that changes the ideal triangulation $\tau$ into $\tau^{e}$, which consists of the ideal arcs of $\tau$ except $e$, which is replaced by $e'$ such that triangles $\mgd{a,b,e}$ and $\mgd{c,d,e}$ are replaced by $\mgd{b,c,e'}$ and $\mgd{a,d,e'}$, is called an \textit{elementary move} (see Figure \ref{elementary move}). The six $\lambda$-lengths are related by one single equation
		\begin{equation}\label{relationee'}
		ee'=ac+bd.
		\end{equation}
		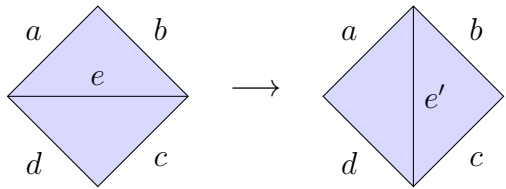
\begin{figure}[h]
		\begin{center}
			\begin{tikzpicture}[scale=.6,baseline]
			\draw[fill=blue!15] (-2,0)--(0,2)--(2,0)--(0,-2)--cycle;
			\draw (-2,0)--(2,0);
			\draw (0,0) node[anchor=south]{$e$};
			\draw (-1,1) node[anchor=south east]{$a$};
			\draw (1,1) node[anchor=south west]{$b$};
			\draw (-1,-1) node[anchor=north east]{$d$};
			\draw (1,-1) node[anchor=north west]{$c$};
			\end{tikzpicture}
			\quad $\longrightarrow$ \quad 
			\begin{tikzpicture}[scale=.6,baseline]
			\draw[fill=blue!15] (-2,0)--(0,2)--(2,0)--(0,-2)--cycle;
			\draw (0,-2)--(0,2);
			\draw (0,0) node[anchor=west]{$e'$};
			\draw (-1,1) node[anchor=south east]{$a$};
			\draw (1,1) node[anchor=south west]{$b$};
			\draw (-1,-1) node[anchor=north east]{$d$};
			\draw (1,-1) node[anchor=north west]{$c$};
			\end{tikzpicture}\quad
			\caption{Elementary move}
			\label{elementary move}
			\end{center}
		\end{figure}
	\end{enumerate}
\end{thm}
Due to positivity this is a global coordinate change between parametrizations associated to two ideal triangulations. Two ideal triangulations are related through a sequence of flips. Composing the relations on Ptolemy transformations one obtains the relations between two coordinate systems. 

\subsection{Ratio Coordinates}

\begin{defn}
	An ideal triangulation with a choice of distinguished corner for each triangle is called a \textit{decorated ideal triangulation} (d.i.t).
\end{defn}

For an ideal triangle with sides having $\lambda$-lengths $a,b,c$ we assign ratio coordinates according to Figure \ref{rc}, where $t = \left(\frac{a}{c},\frac{b}{c}\right) = (t_1,t_2)$.
\begin{figure}[h]
\begin{center}
	\begin{tikzpicture}[scale=.6,baseline]
	\draw[fill=blue!15] (-2,-1)--(0,1)--(2,-1)--cycle;
	\draw (0,-1) node[anchor=north]{$c$};
	\draw (-1,0) node[anchor=south east]{$a$};
	\draw (1,0) node[anchor=south west]{$b$};
	\draw (0,1) node[anchor=north]{$*$};
	\end{tikzpicture}
	\quad $\longrightarrow$ \quad 
	\begin{tikzpicture}[scale=.6,baseline]
	\draw [fill=blue!15](-2,-1)--(0,1)--(2,-1)--cycle;
	\draw (0,-.2)node{$t$};
	\draw (0,1) node[anchor=north]{$*$};
	\end{tikzpicture}
	\caption{Ratio coordinates}
	\label{rc}
	\end{center}
\end{figure}
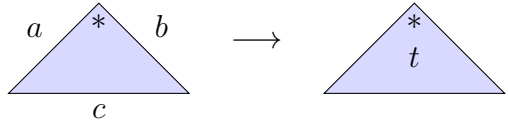
The pull back of the Weil--Petersson symplectic 2-form is then written in the very simple way
\begin{equation*} 
\phi^*\omega_{WP}=\sum_{\begin{tikzpicture}[scale=.3,baseline]
	\draw (-1,0)--(0,2)--(1,0)--cycle;
	\draw (-1,0)--(1,0);
	\draw (0,2.2) node[anchor=north]{$\cdot$};
	\draw (0,.6) node{$t$};
	\end{tikzpicture}} \frac{dt_1 \wedge dt_2}{t_1t_2}
=: \sum_{\begin{tikzpicture}[scale=.3,baseline]
	\draw (-1,0)--(0,2)--(1,0)--cycle;
	\draw (-1,0)--(1,0);
	\draw (0,2.2) node[anchor=north]{$\cdot$};
	\draw (0,.6) node{$t$};\end{tikzpicture}} \omega_t,
\end{equation*}
where the sum is over all triangles. 

The d.i.t. $\tau_t$ obtained from $\tau$ by a change of distinguished corner of triangle $t$ as indicated in Figure \ref{elementary change}
is said to be obtained from $\tau$ by the \textit{elementary change of decoration} in triangle $t$. 
The d.i.t. $\tau^{e}$ obtained from the d.i.t. $\tau$ by the elementary move along the i.a. $e$, where distinguished corners are as indicated in Figure \ref{dem}, is said to be obtained from $\tau$ by the \textit{decorated elementary move} along the i.a. $e$.  

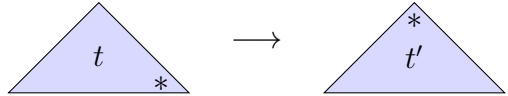
\begin{figure}[h]
\begin{center}
	\begin{tikzpicture}[scale=.6,baseline]
	\draw [fill=blue!15](-2,-1)--(0,1)--(2,-1)--cycle;
	\draw (1.8,-0.8) node[anchor=east]{$*$};
	\draw (0,-.2)node{$t$};
	\end{tikzpicture} \quad $\longrightarrow$ \quad 
	\begin{tikzpicture}[scale=.6,baseline]
	\draw [fill=blue!15](-2,-1)--(0,1)--(2,-1)--cycle;
	\draw (0,-.2)node{$t'$};
	\draw (0,1) node[anchor=north]{$*$};
	\end{tikzpicture}\caption{Elementary change of decoration.}
	\label{elementary change}
\end{center}
\end{figure}

\begin{figure}[h]
\begin{center}
	\begin{tikzpicture}[scale=.6,baseline]
	\draw[fill=blue!15] (-2,0)--(0,2)--(2,0)--(0,-2)--cycle;
	\draw (0,-2)--(0,2);
	\draw (-1.2,0) node[anchor=west]{$x$};
	\draw (1.2,0) node[anchor=east]{$y$};
	\draw (-.2,1.5) node{$*$};
	\draw (.2,1.5) node{$*$};
	\end{tikzpicture}
	\quad $\longrightarrow$ \quad 
	\begin{tikzpicture}[scale=.6,baseline]
	\draw[fill=blue!15] (-2,0)--(0,2)--(2,0)--(0,-2)--cycle;
	\draw (-2,0)--(2,0);
	\draw (0,1.2) node[anchor=north]{$u$};
	\draw (0,-1.2) node[anchor=south]{$v$};
	\draw (-1.5,-.2) node{$*$};
	\draw (0,2) node[anchor=north]{$*$};
	\end{tikzpicture}
	\caption{Decorated elementary move}
	\label{dem}
	\end{center}
\end{figure}
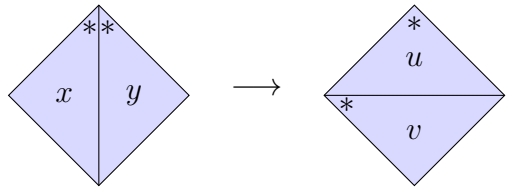

It is easily seen that the coordinates $u,v$ are related to the coordinates $x,y$. The relation is given by the two functions in $x$ and $y$.
\begin{align*}
u &=x \cdot y = (x_1y_1,x_1y_2+x_2), \\
v &=x * y = \left(\frac{x_2y_1}{x_1y_2+x_2},\frac{y_2}{x_1y_2+x_2} \right).
\end{align*}
We now observe that 
\begin{equation*}
\omega_x+\omega_y=\omega_u+\omega_v,
\end{equation*}
so that the change of coordinate with respect to this transformation
$$T:(x,y) \mapsto (u,v)$$
 is a symplectomorphism of $\R^4_{>0}$.

\section{Tetrahedral operator from quantum Teichmüller theory}
We recall the main algebraic ingredients of quantum Teichmüller theory, following the approach of \cite{K1,K2,K3}. Consider the canonical quantization of $T^*\R^n$ with the standard symplectic structure in the position representation. The Hilbert space we get is $L^2(\R^n)$. Position coordinates $q_i$ and momentum coordinates $p_i$ on $T^*\R^n$ upon quantization becomes selfadjoint unbounded operators  $\pos_i$ and $\mom_i$ acting on $L^2(\R^n)$ via the formulae 
\begin{equation*}
\pos_j (f)(t)=t_jf(t), \quad \mom_j(f)(t)=\frac{1}{2 \pi i}\frac{\partial}{\partial t_j}f(t), \quad \forall t\in \R^n, 
\end{equation*}
satisfying the Heisenberg commutation relations 
\begin{equation}
[\pos_j,\pos_k]=[\mom_j,\mom_k]=0, \quad [\mom_j,\pos_k]=\frac{1}{2\pi i}\delta_{j,k}.
\end{equation}
By the spectral theorem, one defines the operators 
\begin{equation*}
\sfu_i = e^{2\pi \bq  \pos_i}, \quad \sfv_i = e^{2 \pi \bq \mom_i}.
\end{equation*}
The commutation relations for $\sfu_i$ and $\sfv_j$ takes the form
\begin{equation*}
[\sfu_j,\sfu_k]=[\sfv_j,\sfv_k]=0, \quad \sfu_j\sfv_k = e^{2\pi \bq^2 \delta_{j,k}}\sfv_k\sfu_j.
\end{equation*}
Consider the operations for $\wv_j=(\sfu_j,\sfv_j), j\in \mgd{1,2},$
\begin{align}
\wv_{1}\cdot \wv_{2} &:=(\sfu_1\sfu_2,\sfu_1\sfv_2+\sfv_1), \\ 
\wv_{1}* \wv_{2} &:=(\sfv_1\sfu_2(\sfu_1\sfv_2+\sfv_1)^{-1},\sfv_2(\sfu_1\sfv_2+\sfv_1)^{-1}).
\end{align}

\begin{prop}[Kashaev]\label{prop:Kashaev}
	Let $\psi(z)$ be some solution to the functional equation 
	\begin{equation}\label{eq:functional}
	\psi\left(z+\frac{i\bq}{2}\right)=\psi\left(z-\frac{i\bq}{2}\right)(1+e^{2\pi \bq z}), \quad z\in \C.
	\end{equation}
	Then, the operator 
	\begin{equation}\label{eq:operator}
	\pto = \pto_{12}  := e^{2\pi i \mom_1 \pos_2} \psi\left(\pos_1-\pos_2+\mom_2\right), 
	\end{equation}
	defines a continuous linear map from $\calS(\R^4)$ to $\calS(\R^4)$, which satisfies the equations 
	\begin{align}\label{tetraop}
	\wv_1 \cdot \wv_2 \pto = \pto \wv_1, \quad \wv_1 * \wv_2 \pto = \pto \wv_2.
	\end{align}
\end{prop}
For a proof of this proposition see \cite{AK1} and Appendix \ref{TOP}.
One particular solution of \eqref{eq:functional} is given by Faddeev's quantum dilogarithm \cite{F}
\begin{equation}
\psi(z)= \frac{1}{\qdl(z)}.
\end{equation} The most important property of the operator \eqref{eq:operator} is the pentagon identity in $L^2(\R^3)$
\begin{equation}\label{eq:five}
\pto_{12}\pto_{13}\pto_{23}=\pto_{23}\pto_{12},
\end{equation}which follows from the five term identity \eqref{eq:fiveterm} satisfied by Faddeev's quantum dilogarithm. The indices in \eqref{eq:five} has the standard meaning. For example $\pto_{13}$ is obtained from $\pto_{12}$ by replacing $\pos_2$ and $\mom_2$ with $\pos_3$ and $\mom_3$ respectively and so forth.

\subsection{Oriented triangulated pseudo 3-manifolds}
Consider the disjoint union of finitely many copies of the standard 3-simplices in $\R^3$, each having totally ordered vertices. The order of the vertices induces an orientation on edges. Identify some codimension-1 faces of this union in pairs by vertex order preserving and orientation reversing affine homeomorphisms called \textit{gluing homeomorphisms}. The quotient space $X$ is a specific $CW$-complex with oriented edges which will be called an oriented \textit{triangulated pseudo 3-manifold}. For $i\in \mgd{0,1,2,3}$, we denote by $\Delta_i(X)$ the set of $i$-dimensional cells in $X$. 
For any $i>j$, we denote 
\begin{equation*}
\Delta_{i}^{j}(X)=\mgd{(a,b) \mid a\in \Delta_i(X), \: b\in \Delta_j(a)}
\end{equation*}
with natural projection maps 
\begin{align*}
\phi_{i,j}: \Delta_i^j(X) \to \Delta_i(X),  \quad \phi^{i,j}: \Delta_i^j(X) \to \Delta_j(X).
\end{align*}
We also have canonical boundary maps 
\begin{equation*}
\partial_i : \Delta_j(X) \to \Delta_{j-1}(X), \quad 0 \leq i \leq j,
\end{equation*}
which in the case of a $j$-dimensional simplex $S=[v_0,v_1,\dots,v_j]$ with ordered vertices $v_0,v_1,\dots,v_j$ in $\R^3$ takes the form
\begin{equation*}
\partial_i S = [v_0,\dots,v_{i-1},v_{i+1},\dots,v_j], \quad i\in \mgd{0,\dots,j}.
\end{equation*}

\subsection{Shaped 3-manifolds} Let $X$ be an oriented triangulated pseudo 3-manifold. 
\begin{defn}
	A \textit{shape structure} on $X$ is an assignment to each edge of each tetrahedron of $X$ a positive number called the dihedral angle,
	\begin{equation*}
	\alpha_X : \Delta_{3}^1(X) \to \R_+
	\end{equation*}
	so that the sum of the three angles at the edges from each vertex of each tetrahedron is $\pi$. An oriented triangulated pseudo 3-manifold with a shape structure will be called a shaped pseudo 3-manifold. 
\end{defn}
It is straightforward to see that the dihedral angles at opposite edges are equal. 

\begin{figure}[h]
\begin{center}
	\begin{tikzpicture}[scale=.8,>=stealth,baseline]
	\draw[fill=blue!15] (-2,0)--(.5,-1)--(2,1)--(-.5,3)--(-2,0);
	\draw[dashed,black!40] (0,0.5)--(2,1) (-2,0)--(0,0.5);
	\draw (0,1)--(.5,-1) (-.5,3)--(0,1);
	\draw (-.5,3) node[anchor=south]{$0$};
	\draw (-2,0) node[left]{$1$};
	\draw (.5,-1) node[anchor=north]{$2$};
	\draw (2,1) node[right]{$3$};
	\draw (-0.75,-0.5) node[anchor=north east]{$\gamma$};
	\draw (0,0.4) node[anchor=south east,black!40]{$\beta$};
	\draw (0,1) node[anchor= west]{$\beta$};
	\draw (-1.25,1.5) node[anchor=south east]{$\alpha$};
	\draw (1.25,0) node [anchor=north west]{$\alpha$};
	\draw (.75,2) node [anchor=south west]{$\gamma$};
	\end{tikzpicture} 
	\caption{Labeling of edges by dihedral angles.}
	\label{dihedral}
	\end{center}
\end{figure}
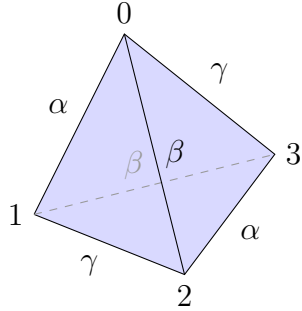

\begin{defn}
	To each shape structure on $X$, we associate a \textit{Weight function} 
	\begin{equation*}
	\omega_X : \Delta_1(X) \to \R_+,
	\end{equation*}
	which to each edge of $X$ associates the total sum of dihedral angles around it
	\begin{equation*}
	\omega_X(e)=\sum_{(T,e) \in \Delta_3^1(X)} \alpha_X (T,e).
	\end{equation*}
	An edge of a shaped pseudo 3-manifold $X$ will be called \textit{balanced} if it is internal and $\omega_X(e)=2\pi$. We call a shaped pseudo 3-manifold \textit{fully balanced} if all edges of $X$ are balanced.
\end{defn} 

\subsection{Shape gauge transformation}
In the space of shape structures on a pseudo 3-manifold there is a gauge group action. The gauge group is generated by the total dihedral angles around internal edges acting through the Neumann--Zagier Poisson bracket. See \cite{AK1} for further details.

\subsection{Geometric interpretation of the five term identity}
For an operator $\pto$ we denote the integral kernel of the operator as $\langle x_0,x_2\mid \pto \mid x_1,x_3 \rangle$.
Then the pentagon identity can be written in the following way
\begin{equation}\label{eq:pent.}
\langle x,y,z \mid \pto_{12}\pto_{13}\pto_{23} \mid u,v,w\rangle = \langle x,y,z \mid \pto_{23}\pto_{12} \mid u,v,w\rangle
\end{equation}
Decomposition of unity gives for the left hand side of \eqref{eq:pent.}
\begin{align*}
\langle x,y,z \mid \pto_{12}\pto_{13}\pto_{23} \mid u,v,w\rangle 
= \int &\langle x,y \mid \pto \mid \alpha_1,\alpha_2\rangle \langle \alpha_1, z \mid \pto \mid u,\beta_3\rangle \\ &\langle \alpha_2,\beta_3 \mid \pto \mid v,w\rangle  d\alpha_{1}d\alpha_{2}d\beta_{3}.
\end{align*}
Decomposing of unity for the right hand side gives
\begin{align*}
\langle x,y,z \mid \pto_{23}\pto_{12} \mid u,v,w\rangle & 
=\int \langle y,z \mid \pto \mid \gamma_2,w \rangle \langle x,\gamma_2 \mid \pto \mid u,v \rangle d\gamma_2.
\end{align*}
To make the correspondence between the pentagon identity and the 3-2 Pachner move precise, we label each vertex of a tetrahedron $T$ with a number $i\in \mgd{0,1,2,3}.$ The numbers on vertices induce an orientation on edges, i.e. we put arrows on the edges pointing in the direction from the smaller to the bigger label on vertices. The number at a vertex corresponds to the number of incoming edges, see Figure \ref{branching}.
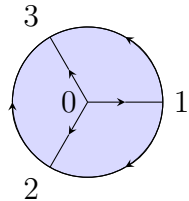
\begin{figure}[h]
\begin{center}
	\begin{tikzpicture}[scale=.5,>=stealth,baseline]
	\draw[fill=blue!15] (0,0) circle (2);
	\draw[->] (1,0)--(2,0) (0,0)--(1,0);
	\draw[->](-120:1)--(-120:2) (0,0)--(-120:1);
	\draw[->](120:1)--(120:2) (0,0)--(120:1);
	\draw[->](2,0) arc (0:60:2);
	\draw[->](2,0) arc (0:-60:2);
	\draw[->](-120:2) arc (-120:-180:2);
	\draw (60:2) arc (60:180:2);
	\draw (-60:2) arc (-60:-120:2);
	\draw (0,0) node[left]{$0$};
	\draw (2,0) node[right]{$1$};
	\draw (-120:2) node[anchor=north east]{$2$};
	\draw (120:2) node[anchor=south east]{$3$};
	\end{tikzpicture}
	\caption{Interpretation of a positively oriented tetrahedron.}
	\label{branching}
\end{center}
\end{figure}

\subsection{States}
A state of a tetrahedron $T$ with totally ordered vertices $\mgd{0,1,2,3}$ is a map $$x:\Delta_2(X)\to \R.$$ A tetrahedron in state $x$ is illustrated in Figure \ref{fig:state}, where $x_i:=x(\partial_i T).$ 
\begin{figure}[h]
\begin{center}
	\begin{tikzpicture}[scale=.7,>=stealth,baseline]
	\draw[fill=blue!15] (0,0) circle (2);
	\draw[->] (1,0)--(2,0) (0,0)--(1,0);
	\draw[->](-120:1)--(-120:2) (0,0)--(-120:1);
	\draw[->](120:1)--(120:2) (0,0)--(120:1);
	\draw[->](2,0) arc (0:60:2);
	\draw[->](2,0) arc (0:-60:2);
	\draw[->](-120:2) arc (-120:-180:2);
	\draw (60:2) arc (60:180:2);
	\draw (-60:2) arc (-60:-120:2);
	\draw (4,0) node[left]{$x_0$};
	\draw (-1,0) node{$x_1$};
	\draw (1,1) node[left]{$x_2$};
	\draw (1,-1) node[left]{$x_3$};
	\end{tikzpicture}\qquad\qquad
	\caption{A tetrahedron in state $x$.}
	\label{fig:state}
	\end{center}
\end{figure}
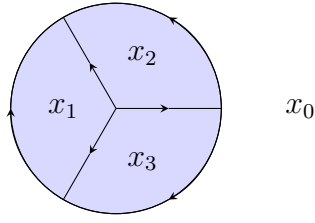

We identify a tetrahedron $T$ in state $x$ as in Figure \ref{fig:state} with the integral kernel $\langle x_0 , x_2 \mid \pto \mid x_1,x_3 \rangle$. This gives a geometric interpretation of the pentagon identity \eqref{eq:pent.} as the 2-3 Pachner move as illustrated in Figure \ref{fig:pachnerdecomp} and Figure \ref{fig:pachner}.
\begin{figure}[h]
\begin{center}
	\begin{tikzpicture}[scale=.7,>=stealth,baseline]
	\draw[fill=blue!15] (2,2)--(2,-2)--(4,0)--cycle;
	\draw[->] (3,1)--(2,2) (4,0)--(3,1);
	\draw[->] (3,-1)--(4,0) (2,-2)--(3,-1);
	\draw[->] (2,0)--(2,2) (2,-2)--(2,0);
	\draw[->] (2.3,-1.35)--(2.6,-.7) (2,-2)--(2.3,-1.35);
	\draw[->] (2.3,0.65)--(2.6,-.7) (2,2)--(2.3,0.65);
	\draw[->] (3.3,-.35)--(2.6,-.7) (4,0)--(3.3,-.35);
	\draw[fill=blue!15] (-1.5,1)--(0,3)--(1.5,1)--(0,-1)--cycle;
	\draw[->][dashed] (.3,1)--(1.5,1) (-1.5,1)--(.3,1);
	\draw[->] (0,2)--(0,3) (0,-1)--(0,2);
	\draw[->] (-0.75,0)--(0,-1) (-1.5,1)--(-0.75,0);
	\draw[->] (-0.75,2)--(0,3) (-1.5,1)--(-0.75,2);
	\draw[->] (0.75,2)--(0,3) (1.5,1)--(0.75,2);
	\draw[->] (0.75,0)--(1.5,1) (0,-1)--(0.75,0);
	\draw[fill=blue!15] (-2,2)--(-2,-2)--(-4,0)--cycle;
	\draw[->] (-3,1)--(-2,2) (-4,0)--(-3,1);
	\draw[->] (-3,-1)--(-2,-2) (-4,0)--(-3,-1);
	\draw[->] (-2,0)--(-2,2) (-2,-2)--(-2,0);
	\draw[->] (-2.3,-1.35)--(-2.6,-.7) (-2,-2)--(-2.3,-1.35);
	\draw[->] (-2.3,0.65)--(-2.6,-.7) (-2,2)--(-2.3,0.65);
	\draw[->] (-3.3,-.35)--(-2.6,-.7) (-4,0)--(-3.3,-.35);
	\end{tikzpicture}
	\quad=\quad 
	\begin{tikzpicture}[scale=.6,>=stealth,baseline]
	\draw[fill=blue!15] (2,.5)--(0,2.5)--(-2,0.5)--(.6,-.2)--cycle;
	\draw (0,2.5)--(.6,-.2);
	\draw[->] (-.7,0.15)--(.6,-.2) (-2,.5)--(-.7,0.15);
	\draw[->] (1.3,0.15)--(.6,-.2) (2,.5)--(1.3,0.15);
	\draw[->] (1,1.5)--(0,2.5) (2,.5)--(1,1.5);
	\draw[->] (-1,1.5)--(0,2.5) (-2,.5)--(-1,1.5);
	\draw[->][dashed] (0,0.5)--(2,0.5) (-2,0.5)--(0,0.5);
	\draw[fill=blue!15] (2,-.5)--(-2,-0.5)--(0,-2.5)--cycle;
	\draw[->] (0,-0.5)--(2,-0.5) (-2,-.5)--(0,-.5);
	\draw[->] (0.3,-1.85)--(0.6,-1.2) (0,-2.5)--(0.3,-1.85);
	\draw[->] (-0.7,-0.85)--(.6,-1.2) (-2,-.5)--(-0.7,-0.85);
	\draw[->] (1.3,-0.85)--(.6,-1.2) (2,-.5)--(1.3,-0.85);
	\draw[->] (1,-1.5)--(2,-.5) (0,-2.5)--(1,-1.5);
	\draw[->] (-1,-1.5)--(0,-2.5) (-2,-.5)--(-1,-1.5);
	\draw[->][dashed] (0,0.5)--(2,0.5) (-2,0.5)--(0,0.5);
	\end{tikzpicture}
	\caption{Decomposition of tetrahedra in the 2-3 Pachner move.}
	\label{fig:pachnerdecomp}
	\end{center}
\end{figure}
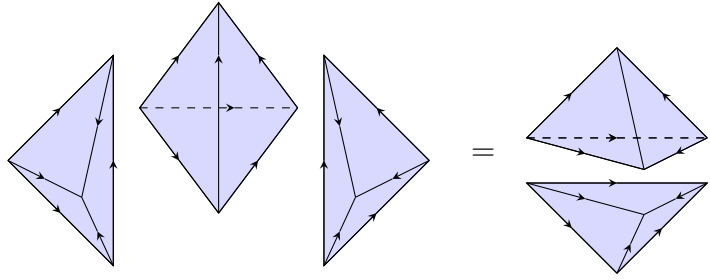
The integrations corresponds to gluing of faces as illustrated in Figure \ref{fig:pachner}. 
\begin{figure}[h]
\begin{center}
	\begin{tikzpicture}[scale=.6,baseline]
	\draw[fill=blue!15] (2,0)--(0,2)--(-2,0)--(0,-2)--cycle;
	\draw (0,2)--(.7,-.7)--(0,-2);
	\draw (-2,0)--(.7,-.7)--(2,0);
	\draw[dashed] (-2,0)--(2,0);
	\draw[dotted] (0,2)--(0,-2);
	\end{tikzpicture}\quad
	$=$ \quad
	\begin{tikzpicture}[scale=.6,baseline]
	\draw[fill=blue!15] (2,0)--(0,2)--(-2,0)--(0,-2)--cycle;
	\draw (0,2)--(.6,-.7)--(0,-2);
	\draw (-2,0)--(.6,-.7)--(2,0);
	\draw[dashed] (-2,0)--(2,0);
	\end{tikzpicture}
	\caption{3-2 Pachner move.}
	\label{fig:pachner}
	\end{center}
\end{figure}
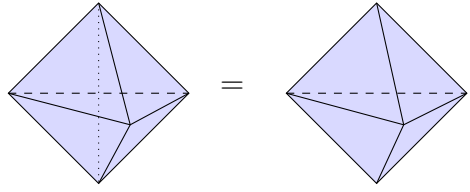

\subsection{Integral kernel}
Let us calculate the integral kernel of the operator $\pto$
\begin{align*}
\langle x_0 , x_2 \mid \pto \mid x_1, x_3 \rangle \equiv \pto f(x,y) = \int \langle x,y \mid \pto \mid u,v \rangle f(u,v) dudv,
\end{align*} where $\pto$ is the operator given by \eqref{eq:operator}.
\begin{align*}
\langle x_0 , x_2 \mid \pto \mid x_1, x_3 \rangle &= \langle x_0 , x_2 \mid e^{2\pi i \mom_1 \pos_2} \psi\left(\pos_1-\pos_2+\mom_2\right) \mid x_1, x_3 \rangle \\
&= e^{x_2 \frac{\partial}{\partial x_0}}\langle x_0 , x_2 \mid \psi\left(\pos_1-\pos_2+\mom_2\right) \mid x_1, x_3 \rangle \\
&= \langle x_0 +x_2 , x_2 \mid \psi\left(\pos_1-\pos_2+\mom_2\right) \mid x_1, x_3 \rangle \\
&=\int \langle x_0 +x_2 , x_2 \mid e^{2 \pi i (\pos_1-\pos_2+\mom_2)y} \mid x_1, x_3 \rangle \tilde{\psi}(y) \:dy \\
&=\int e^{2 \pi i yx_1}\delta(x_1-x_0-x_2)\langle x_2 \mid e^{2 \pi i (\mom_2-\pos_2)y } \mid x_3 \rangle \tilde{\psi}(y) \:dy \\
&=\int e^{2\pi i x_1 y}\delta(x_1-x_0-x_2)\tilde{\psi}(y)e^{-2\pi i x_3 y}\langle  x_2+y\mid x_3 \rangle \: dy \\
&=\int e^{2 \pi i x_1 y } \delta(x_1-x_0-x_2)\tilde{\psi}(y)e^{-2\pi i x_3 y}\delta(x_2+y-x_3)e^{\pi i y^2} \:dy \\
&=e^{2\pi i x_1(x_3-x_2)}\delta(x_1-x_0-x_2)\tilde{\psi}(x_3-x_2)e^{-2\pi i x_3(x_3-x_2)+\pi i (x_3-x_2)^2} \\
&= \delta(x_1-x_0-x_2)\tilde{\psi}'(x_3-x_2)e^{2\pi i x_0(x_3-x_2)},
\end{align*}
where 
\begin{equation*}
\tilde{\psi}'(x):=\tilde{\psi}(x)e^{-\pi i x^2}, \quad \text{and} \quad\tilde{\psi}(x):=\int_{\R}\psi(y)e^{2\pi i xy} \: dy.
\end{equation*}

\subsection{Positively and negatively oriented tetrahedra} In an oriented triangulated $3$-manifold there are two possibilities for the orientation of tetrahedra. The orientation follows from Figure \ref{fig:orient}. To a negatively oriented tetrahedron the integral kernel associated to it in the geometric interpretation is the complex conjugate of that of a positively oriented tetrahedron.

\begin{figure}[h]
\begin{center}
	\begin{tikzpicture}[scale=.7,>=stealth,baseline]
	\draw[fill=blue!15] (0,0) circle (2);
	\draw[->] (1,0)--(2,0) (0,0)--(1,0);
	\draw[->](-120:1)--(-120:2) (0,0)--(-120:1);
	\draw[->](120:1)--(120:2) (0,0)--(120:1);
	\draw[->](2,0) arc (0:60:2);
	\draw[->](2,0) arc (0:-60:2);
	\draw[->](-120:2) arc (-120:-180:2);
	\draw (60:2) arc (60:180:2);
	\draw (-60:2) arc (-60:-120:2);
	\draw (0,0) node[left]{$0$};
	\draw (2,0) node[right]{$1$};
	\draw (-120:2) node[anchor=north east]{$2$};
	\draw (120:2) node[anchor=south east]{$3$};
	\draw (0,-2.5) node[anchor=north]{Positively oriented tetrahedron.};
	\end{tikzpicture}\qquad\qquad
	\begin{tikzpicture}[scale=.7,>=stealth,baseline]
	\draw[fill=blue!15] (0,0) circle (2);
	\draw[->] (1,0)--(2,0) (0,0)--(1,0);
	\draw[->](-120:1)--(-120:2) (0,0)--(-120:1);
	\draw[->](120:1)--(120:2) (0,0)--(120:1);
	\draw[->](2,0) arc (0:60:2);
	\draw[->](2,0) arc (0:-60:2);
	\draw[->](120:2) arc (120:180:2);
	\draw (60:2) arc (60:120:2);
	\draw (-60:2) arc (-60:-180:2);
	\draw (0,0) node[left]{$0$};
	\draw (2,0) node[right]{$1$};
	\draw (-120:2) node[anchor=north east]{$3$};
	\draw (120:2) node[anchor=south east]{$2$};
	\draw (0,-2.5) node[anchor=north]{Negatively oriented tetrahedron.};
	\end{tikzpicture}
	\caption{Orientations on tetrahedra.}
	\label{fig:orient}
	\end{center}
\end{figure}

\subsection{Charged Tetrahedral Operators and Pentagon Identity}\label{CTOPI}
To ensure that the Fourier integral is absolutely convergent charges on the operator $\pto$ are introduced. For any positive real $a$ and $c$ such that $b:=\frac{1}{2}-a-c$ is also positive, define the charged $\pto$-operators 
\begin{equation}
\pto(a,c):=e^{-\pi i c_{\bq}^2(4(a-c)+1)/6}e^{4\pi i c_{\bq} (c\pos_2-a\pos_1)} \pto e^{-4\pi i c_{\bq}(a\mom_2+c\pos_2)}
\end{equation}
and
\begin{equation*}
\bar\pto(a,c):=e^{\pi i c_{\bq}^2(4(a-c)+1)/6}e^{-4\pi i c_{\bq} (a\mom_2+c\pos_2)} \bar\pto e^{4\pi i c_{\bq}(c\pos_2-a\pos_1)}
\end{equation*}
where $\bar\pto := \pto^{-1}$ and $c_{\bq}:=\frac{i}{2}(\bq+\bq^{-1})$.
We have the equality
\begin{equation*}
\pto(a,c) = e^{2\pi i \mom_1\pos_2}\psi_{a,c}(\pos_1-\pos_2+\mom_2)
\end{equation*} where
\begin{equation*}
\psi_{a,c}(x):=\psi(x-2c_{\bq}(a+c))e^{-4\pi i c_{\bq}a(x-c_{\bq}(a+c))}e^{-\pi i c_{\bq}^2(4(a-c)+1)/6}.
\end{equation*}
The Fourier transformation formula for Faddeev's quantum dilogarithm \eqref{eq:Fourier} leads to the identities
\begin{align*}
\tpsi{a}{c}{x} &=e^{-\frac{\pi i }{12}}\psi_{c,b}(x), \\
\overline{\psi_{{a},{c}}({x})}&=e^{-\frac{\pi i }{6}}e^{\pi i x^2}\psi_{c,a}(-x) = e^{-\frac{\pi i }{12}}\tilde{\psi}_{b,c}(-x), \\
\overline{\tpsi{a}{c}{x}}&=e^{\frac{\pi i }{12}}\overline{\psi_{c,b}(x)}=e^{-\frac{\pi i }{12}}e^{\pi i x^2}\psi_{b,c}(-x).
\end{align*}
From the formulas above we obtain that
\begin{align}
\langle x_0,x_2 \mid \pto(a,c)\mid x_1,x_3 \rangle &= \delta(x_1-x_0-x_2)\tpsi{a}{c}{x_3-x_2}e^{2\pi i x_0(x_3-x_2)}, \\
\langle x,y \mid \bar \pto(a,c)\mid u,v \rangle &= \overline{\langle u,v \mid \pto(a,c)\mid x,y \rangle}.
\end{align}

\begin{prop}[Andersen--Kashaev \cite{AK1}]
	The charged pentagon identity is satisfied
	\begin{equation}
	\pto_{12}(a_4,c_4)\pto_{13}(a_2,c_2)\pto_{23}(a_0,c_0) = e^{\pi i c_{\bq}^2 P_e/3} \pto_{23}(a_1,c_1)\pto_{12}(a_3,c_3),
	\end{equation}
	where $$ P_e=2(c_0+a_2+c_4)-\frac{1}{2}$$ and $a_0,a_1,a_2,a_3,a_4,c_0,c_1,c_2,c_3,c_4 \in \R$ are such that 
	$$a_1=a_0+a_2, \: a_3=a_2+a_4, \: c_1=c_0+a_4,  \:c_3=a_0+c_4, \: c_2=c_1+c_3.$$
\end{prop}

\subsection{Partition function}
For a tetrahedron $T=[v_0,v_1,v_2,v_3]$ with ordered vertices $v_0,v_1,v_2,v_3$, we define its $\sign$
\begin{equation*}
\sign(T)=\sign(\det(v_1-v_0,v_2-v_0,v_3-v_0)).
\end{equation*}
For $(T,\alpha,x)$ an oriented tetrahedron with shape structure $\alpha$ in state $x$, define the partition function taking values in the space of tempered distributions by the formula
\begin{equation}
Z_{\hbar}(T,\alpha,x) := \begin{cases}
\langle x_0,x_2 \mid \pto(a,c) \mid x_1,x_3 \rangle, & \text{if } \sign(T)=1 \\
\langle x_1,x_3 \mid \bar{\pto}(a,c) \mid x_0,x_2 \rangle, & \text{if } \sign(T)=-1.
\end{cases}
\end{equation}
where $$x_i := x(\partial_i T)$$ and $$a=\frac{1}{2\pi} \alpha_T(T,e_{01}), \quad c=\frac{1}{2\pi} \alpha_T(T,e_{03}).$$

For a closed oriented triangulated pseudo 3-manifold $X$ with shape structure $\alpha$, we associate the partition function 
\begin{align}\label{eq:partition}
Z_{\hbar}(X,\alpha) := \int_{x \in \R^{\Delta_2(X)}} \prod_{T\in \Delta_3(X)} Z_{\hbar}(T,\alpha,x) \:dx.
\end{align}

\begin{thm}[Andersen--Kashaev \cite{AK1}]
	If $H_2(X\backslash \Delta_0(X),\Z)=0$, then the quantity $\abs{Z_{\hbar}(X,\alpha)}$ is well defined in the sense that the integral is absolutely convergent, and 
	\begin{enumerate}
		\item it depends on only the gauge reduced class of $\alpha$;
		\item it is invariant under $2-3$ Pachner moves.
	\end{enumerate}
\end{thm}

The definition of the partition function \eqref{eq:partition} can be extended to manifolds having boundary eventually giving rise to a TQFT, see \cite{AK1}. 

\subsection{Invariants of knots in 3-manifolds}
By considering one-vertex ideal triangulations of complements of hyperbolic knots in compact oriented closed 3-manifolds, we obtain knot invariants. 

Another possibility is to consider a one-vertex Hamiltonian triangulation (H-triangulation) of pairs (a closed 3-manifold $M$, a knot $K$ in $M$), i.e., a one-vertex triangulation of $M$, where the knot is represented by one edge, with degenerate shape structures, where the weight on the knot approaches zero and where simultaneously the weights on all other edges approach the balanced value $2\pi$. This limit by itself is divergent as a simple pole (after analytic continuation to complex angles) in the weight of the knot, but the residue at this pole is a knot invariant which is a direct analogue of Kashaev's invariants \cite{K6j}, which were at the origin of the hyperbolic volume conjecture.

In \cite{AK1} the first author and Rinat Kashaev have set forth the following conjecture:
\begin{conj}[Andersen and Kashaev \cite{AK1}]\label{conje}
	Let $M$ be a closed oriented 3-manifold. For any hyperbolic knot $K\subset M$, there exists a smooth function $J_{M,K}(\hbar,x)$ on $\R_{>0}\times\R$ which has the following properties.
	\begin{enumerate}[(1)]
		\item
		For any fully balanced shaped ideal triangulation $X$ of the complement of $K$ in $M$, there exists a gauge invariant real linear combination of dihedral angles $\lambda$, a (gauge non-invariant) real quadratic polynomial of dihedral angles $\phi$ such that 
		\begin{equation*}
		Z_{\hbar}(X)=e^{i\frac{\phi}{\hbar}}\int_{\R}J_{M,K}(\hbar,x)e^{-\frac{x\lambda}{\sqrt{\hbar}}} dx.
		\end{equation*}
		\item\label{kanmandet}
		For any one vertex shaped $H$-triangulation $Y$ of the pair $(M,K)$ there exists a real quadratic polynomial of dihedral angles $\phi$ such that 
		\begin{equation*}
		\lim_{\omega_Y\to \tau}\Phi_{\bq}\left(\frac{\pi-\omega_Y(K)}{2\pi i\sqrt{\hbar}}\right)Z_{\hbar}(Y)=e^{i\frac{\phi}{\hbar}-i\pi/12}J_{M,K}(\hbar,0),
		\end{equation*}
		where $\tau:\Delta_1(Y)\to\R$ takes the value $0$ on the knot $K$ and the value $2\pi$ on all other edges.
		\item
		The hyperbolic volume of the complement of $K$ in $M$ is recovered as the limit 
		\begin{equation*}
		\lim_{\hbar\to 0}2\pi\hbar \log{\abs{J_{M,K}(\hbar,0)}}=-\vol(M\backslash K).
		\end{equation*}
	\end{enumerate}
\end{conj}

\begin{thm}
	Conjecture \ref{conje} is true for the pair $(S^3,6_1)$ with 
	\begin{equation}
	J_{S^3,6_1}(\hbar,x)=\chi_{6_1}(x).
	\end{equation}
	The function $\chi_{6_1}(x)$ is defined to be:
	\begin{equation*}
	\chi_{6_1}(x)=\int_{\R^2} \frac{e^{2\pi i (x^2+\frac{1}{2}y^2 +2xy)e^{4\pi i c_{\bq}z}}}{\Phi_{\bq}(x+y)\Phi_{\bq}(x+z+c_{\bq})\Phi_{\bq}(y)\Phi_{\bq}(z-x-y)} dydz.
	\end{equation*}
\end{thm}
See \cite{JJKN} for a calculation of the invariant of an ideal triangulation of the complement of the knot $6_1.$

\section{New formulation}
In this section we recall the new formulation of the Teichmüller TQFT introduced by Andersen and Kashaev in \cite{AK2}.  
\subsection{States and Bolzmann weights}
Let $T\subset \R^3$ be a tetrahedron with shape structure $\alpha_T$ and vertex ordering mapping
$$ v : \mgd{0,1,2,3}\to\Delta_0(T).$$
A \emph{state} of a tetrahedron $T$ is a map $x\colon \Delta_1(T)\to\R$. Pictorially, a positive tetrahedron $T$ in state $x$ looks as follows
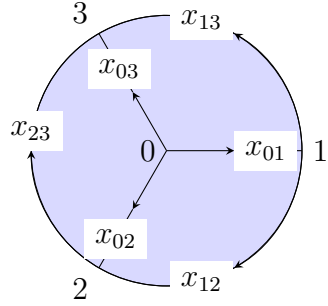
\begin{equation*}
\begin{tikzpicture}[scale=.9,>=stealth,baseline]
\draw[fill=blue!15] (0,0) circle (2);
\draw[->] (1,0)--(2,0) (0,0)--(1,0);
\draw[->](-120:1)--(-120:2) (0,0)--(-120:1);
\draw[->](120:1)--(120:2) (0,0)--(120:1);
\draw[->](2,0) arc (0:60:2);
\draw[->](2,0) arc (0:-60:2);
\draw[->](-120:2) arc (-120:-180:2);
\draw (60:2) arc (60:180:2);
\draw (-60:2) arc (-60:-120:2);
\draw (0,0) node[left]{$0$};
\draw (2,0) node[right]{$1$};
\draw (-120:2) node[anchor=north east]{$2$};
\draw (120:2) node[anchor=south east]{$3$};
\draw(-2,0) node[anchor=south,fill=white]{$x_{23}$};
\draw(1,0) node[anchor=west,fill=white]{$x_{01}$};
\draw(-120:1.5) node[fill=white]{$x_{02}$};
\draw(75:2) node[fill=white]{$x_{13}$};
\draw(-75:2) node[fill=white]{$x_{12}$};
\draw(120:1.4) node[fill=white]{$x_{03}$};
\end{tikzpicture}
\end{equation*}
More generally, a {state} of a triangulated pseudo 3-manifold $X$ is a map
\begin{equation*}
y:\Delta_{1}(X)\to \R.
\end{equation*}

For any state $y$ define the \emph{Boltzmann weight}
\begin{equation*}
B(T,x)=g_{\alpha_1,\alpha_3}(y_{02}+y_{13}-y_{03}-y_{12},y_{02}+y_{13}-y_{01}-y_{23})
\end{equation*}
if $T$ is positive and complex conjugate otherwise. Here $y_{ij}\equiv y(v_iv_j), \alpha_i\equiv \alpha_{T}(v_0v_i)/2\pi$,
\begin{equation}
g_{a,c}(s,t) = \sum_{m\in \Z} \tpsi{a}{c}{s+m}e^{\pi i t (s+2m)}. 
\end{equation}

\begin{thm}[Andersen--Kashaev \cite{AK1}]\label{thm:AK}
	Let X be a levelled shaped triangulated oriented pseudo 3-manifold. Then, the quantity
	\begin{equation}\label{eq:stateint}
	Z_{\hbar}^{\text{new}}(X) := e^{\pi i l_X/{4\hbar}} \int_{[0,1]^{\Delta_{1}(X)}} \left( \prod_{T\in \Delta_3(X)}B\left(T,y\big\vert_{\Delta_1(T)}\right) \right) \: dx
	\end{equation}
	admits an analytic continuation to a meromorphic function of the complex shapes, which is invariant under all shaped $2-3$ and $3-2$ Pachner moves (along balanced edges).
\end{thm}

\begin{conjecture}
	The proposed model in Theorem \ref{thm:AK} is equivalent to the Teichmüller TQFT from \cite{AK1}.  
\end{conjecture}

\begin{thm}
	The new formulation of the Teichmüller TQFT is equivalent to the original formulation for the pairs 
	$(S^3,3_1)$, $(S^3,4_1)$, $(S^3,5_2)$ and $(S^3,6_1)$.
\end{thm}

\section{Calculations of specific knot complements}
In the following calculations we encode an oriented triangulated pseudo $3$-manifold $X$ into a diagram where a tetrahedron $T$ is represented by an element 
\begin{equation}
T \qquad = \qquad \begin{tikzpicture}[scale=.5,baseline=5]
\draw[very thick] (0,0)--(3,0);
\draw (0,0)--(0,1);\draw (1,0)--(1,1);\draw (2,0)--(2,1);\draw (3,0)--(3,1);
\end{tikzpicture}
\end{equation}
where the vertical segments, ordered from left to right, correspond to the faces $\partial_0 T$,$\partial_1 T$, $\partial_2 T$,$\partial_3 T$ respectively. When we glue tetrahedron along faces, we illustrate this by joining the corresponding vertical segments. 

We will further use the notation $$ \nu_{a,c}:=e^{-\pi i c_{\bq^2}(4(a-c)+1)/6}.$$
\subsection{The complement of the figure-8-knot}
Let $X$ be  the following oriented triangulated pseudo 3-manifold, 
\begin{equation}
\begin{tikzpicture}[scale=.5,baseline=5]
\draw[very thick] (0,0)--(3,0);
\draw[very thick] (0,1)--(3,1);
\draw(0,0)--(1,1); 
\draw(0,1)--(1,0);
\draw(2,0)--(3,1);
\draw(3,0)--(2,1);
\end{tikzpicture}
\end{equation}
which represented the usual diagram for the complement of the figure eight knot.
Choosing an orientation, the diagram consists of one positive tetrahedron $T_{+}$ and one negative $T_{-}$. $\partial{X}=\emptyset$ and combinatorially we have 
$\Delta_{0}(X)=\{ * \}$, $\Delta_{1}(X)=\{e_0,e_1\}$. The gluing of the tetrahedra is vertex order preserving which means that edges are glued together in the following manner.
\begin{align*}
e_0=x_{01}^{+}=x_{03}^{+}=x_{23}^{+}=x_{02}^{-}=x_{12}^{-}=x_{13}^{-}=:x,\\
e_1=x_{02}^{+}=x_{13}^{+}=x_{12}^{+}=x_{01}^{-}=x_{03}^{-}=x_{23}^{-}=:y.
\end{align*}
That this diagram represents the complement of the figure eight know means that the topological space $X\backslash \{*\}$ is homeomorphic to the complement of the figure-eight knot. 
The set $\Delta_{3}^1(X)$ consists of the elements $(T_{\pm},e_{j,k})$ for $0\leq j< k \leq 3$. We fix a shape structure
\begin{equation*}
\alpha_{X}:\Delta_{3}^1(X) \to \R_{+}
\end{equation*}
by the formulae 
\begin{equation*}
\alpha_{X}(T_{\pm},e_{0,1})=2\pi a_{\pm}, \quad \alpha_{X}(T_{\pm},e_{0,2}) = 2\pi b_{\pm}, \quad \alpha_{X}(T_{\pm},e_{0,3})=2\pi c_{\pm},
\end{equation*}
where $a_{\pm}+b_{\pm}+c_{\pm}=\frac{1}{2}$.
This result in the following weight functions
\begin{equation*}
\omega_{X}(e_0)=2a_++c_++2b_-+c_-, \quad \omega_{X}(e_1)= 2b_++c_++2a_-+c_-.
\end{equation*}
In the completely balanced case these equations correspond to 
$$ a_+-b_+=a_--b_-. $$
The Boltzmann weights are given by the functions
\begin{align*}
B\left(T_{+},x_{\vert_{\Delta_{1}(T_+)}}\right) &=g_{a_+,c_+}(y-x,2(y-x)), \\
B\left(T_{-},x_{\vert_{\Delta_{1}(T_-)}}\right) &=\overline{g_{a_-,c_-}(x-y,2(x-y))}.
\end{align*}
We calculate the partition function for the Teichmüller TQFT using the new formulation
\begin{align*}
Z^{\text{new}}_{\hbar}(X)=&
\int_{[0,1]^2}\sum_{m,n\in \Z} \tpsi{a_+}{c_+}{y-x+m}\overline{\tpsi{a_-}{c_-}{x-y+n}}e^{4\pi i (y-x)(m+n)} \: dxdy \\
=&\,\int_{[0,1]}\sum_{m,n\in \Z} \tpsi{a_+}{c_+}{y+m}\overline{\tpsi{a_-}{c_-}{-y+n}}e^{4\pi i y(m+n)} dy \\
=&\,\sum_{m,n \in \Z}\int_{[m,m+1]} \tpsi{a_+}{c_+}{y}\overline{\tpsi{a_-}{c_-}{-y+m+n}}e^{4\pi i (y-m)(m+n)} dy \\
=&\,\sum_{p \in \Z}\int_{\R} \tpsi{a_+}{c_+}{y}\overline{\tpsi{a_-}{c_-}{-y+p}}e^{4\pi i yp} dy \\
=&\,e^{-\frac{\pi i}{6}}\sum_{p\in \Z}\int_{\R} \ps{c_+}{b_+}{y}\ps{b_-}{c_-}{y-p}e^{\pi i (y-p)^2} e^{4\pi i yp} dy \\
=&\,e^{-\frac{\pi i}{6}}\sum_{p\in \Z}\int_{\R} \psi(y-2c_{\bq}(c_++b_+)) \psi(y-p-2c_{\bq}(b_-+c_-))e^{\pi i y^2}e^{\pi i p^2} e^{2\pi i yp} \\
&\,\times e^{-4\pi i c_{\bq}c_+(y-c_{\bq}(c_++b_+))}e^{-4\pi ic_{\bq}b_-(y-p-c_{\bq}(b_-+c_-))}\\
&\, \times e^{-\pi i (4(c_+-b_+)+1)/6}e^{-\pi i (4(b_--c_-)+1)/6}dy.
\end{align*}
We set $Y=y-2c_{\bq}(c_++b_+)$. Assuming that we are in the completely balanced case we have that
$$-b_--c_-+c_++b_+ =  -b_++b_-. $$
Furthermore we have
$ y^2=Y^2+4c_{\bq}^2(c_++b_+)^2 +4c_{\bq}Y(c_++b_+).$
Implementing this we get the following expression
\begin{align*}
Z^{\text{new}}_{\hbar}(X)=&\,\nu_{c_+,b_+}\nu_{b_-,c_-}e^{-\frac{\pi i}{6}}\sum_{p\in \Z}\int_{\R} \psi(Y-p-2c_{\bq}(b_+-b_-))\psi(Y)\\
&\,\times e^{\pi i(Y^2+4c_{\bq}^2(c_++b_+)^2 +4c_{\bq}Y(c_++b_+))}e^{\pi i p^2}\\
&\;\times e^{2\pi i (Y+2c_{\bq}(c_++b_+))p}\\
&\,\times e^{-4\pi i c_{\bq}c_+(Y+c_{\bq}(c_++b_x))}e^{-4\pi ic_{\bq}b_-(Y-p-c_{\bq}(b_-+c_-)+2c_{\bq}(c_++b_+))} dY\\
=& \,\nu_{c_+,b_+}\nu_{b_-,c_-}e^{-\frac{\pi i}{6}}\sum_{p\in \Z}\int_{\R} \frac{1}{\Phi_{\bq}(Y-p-2c_{\bq}(b_+-b_-))}\frac{1}{\Phi_{\bq}(Y)}\\
&\,\times e^{\pi i Y^2}e^{\pi i p^2}
e^{-4\pi i c_{\bq} Y (-c_+-b_++c_++b_-)} e^{2\pi i Yp}\\
&\times \,e^{-4\pi i c_{\bq} p (-(c_++b_+)-b_-) }\\
&\times \,e^{4\pi i c_{\bq}^2((c_++b_+)^2-c_+(c_++b_+)-b_-(b_-+c_--2(c_++b_+))} dY.
\end{align*}
Now set 
\begin{equation}\label{u}
u=2c_{\bq}(b_+-b_-)
\end{equation}
and 
\begin{equation}\label{v}
v=2b_-+c_-=2b_++c_+,
\end{equation}
and use the formula
$$ \Phi_{\bq}(z)\Phi_{\bq}(-z)=\zeta_{inv}^{-1}e^{\pi i z^2},$$
together with the calculation
$$b_-+b_++c_+=b_-+b_+-2b_++2b_-+c_-=-(b_+-b_-)+(2b_-+c_-). $$
to get that
\begin{align*}
Z^{\text{new}}_{\hbar}(X)=& \,\nu_{c_+,b_+}\nu_{b_-,c_-}\zeta_{inv}e^{-\frac{\pi i}{6}}\sum_{p\in \Z}\int_{\R} \frac{\Phi_{\bq}(p+u-Y)}{\Phi_{\bq}(Y)}e^{-\pi i(Y^2+u^2+p^2-2Yu-2Yp+2up)}\\
&\,\times e^{\pi i Y^2} e^{\pi i p^2} 
e^{2\pi i  Y u}e^{2\pi i Yp} dY\\
&\times \,e^{-2\pi i pu}e^{2\pi i pv}\\
&\times \,e^{4\pi i c_{\bq}^2((c_++b_+)^2-c_+(c_++b_+)-b_-(b_-+c_--2(c_++b_+))}.
\end{align*}
Using the balance condition and formulas \eqref{u} and \eqref{v} we get the equality
\begin{align*}
&-4\pi i c_{\bq}^2\{(c_++b_+)^2+b_-(-b_--c_-+2c_++2b_+)+c_+(c_++b_+)\} = \\
&-4\pi i c_{\bq}^2 \{(-(b_+-b_1)^2) -c_+b_++c_+b_-+b_-c_+-b_-c_- \} = \\
&-2\pi i c_{\bq} \{ -(c_++2b_-)u \} +\pi i u^2 = \\
&-2\pi i c_{\bq} \{ -(2b_-+c_-)u+2(b_+-b_-))u \} +\pi i u^2= \pi i (uv-u^2).
\end{align*}
We get the following expression for the partition function
\begin{align*}
Z^{\text{new}}_{\hbar}(X)=& \,\nu_{c_+,b_+}\nu_{b_-,c_-}\zeta_{inv}e^{-\frac{\pi i}{6}}\sum_{p\in \Z}\int_{\R} \frac{\Phi_{\bq}(p+u-Y)}{\Phi_{\bq}(Y)}e^{-\pi iu^2}\\
&\times \,
e^{4\pi i  Y u}e^{4\pi i Yp}e^{-4\pi i pu}e^{2\pi i pv} e^{\pi i (uv-u^2)}dY\\
=&\,\nu_{c_+,b_+}\nu_{b_-,c_-}\zeta_{inv}e^{-\frac{\pi i}{6}}\sum_{p\in \Z}\int_{\R} \frac{\Phi_{\bq}(p+u-Y)}{\Phi_{\bq}(Y)}\\
&\times \,
e^{4\pi i  Y u}e^{4\pi i Yp}e^{-4\pi i pu}e^{2\pi i pv} e^{\pi i uv}e^{-2\pi iu^2}dY\\
=& \, \nu_{c_+,b_+}\nu_{b_-,c_-}\zeta_{inv}e^{-\frac{\pi i}{6}}\sum_{p\in \Z}\int_{\R} \frac{\Phi_{\bq}(p+u-Y)}{\Phi_{\bq}(Y)} e^{2\pi i (u+p)(2Y-u-p)} dY e^{\pi i v(2p+u)}. 
\end{align*}
Using the Weil-Gel'fand-Zak transform we see that the partition function has the form
\begin{align*}
Z^{\text{new}}_{\hbar}(X) = \nu_{c_+,b_+} \nu_{b_-,c_-}\zeta_{inv}e^{-\frac{\pi i}{6}}W(\chi_{4_1})(u,v).
\end{align*}
Where the function $\chi_{4_1}(x)=\int_{\R-i0}\frac{\Phi_{\bq}(x-y)}{\Phi_{\bq}(y)}e^{2\pi i x(2y-x)}dy.$ The function $\chi_{4_1}(x)$ is exactly the function $J_{S^3,4_1}(\hbar,x)$ from \cite[Thm. 5]{AK1}. It should be noted that this result is connected to Hikami's invariant. Andersen and Kashaev observes in \cite{AK1} that the expression 
\begin{equation*}
\frac{1}{2\pi \bq}\chi_{4_1}\left(-\frac{u}{\pi \bq}, \frac{1}{2}\right),
\end{equation*} 
where $\chi_{4_1}(x,\lambda)=\chi_{4_1}(x)e^{4\pi i c_{\bq}\lambda}$ 
is equal to the formal derived expression in \cite{HI2}.

\subsection{One vertex H-triangulation of the figure-8-knot}
Let $X$ be represented by the diagram 
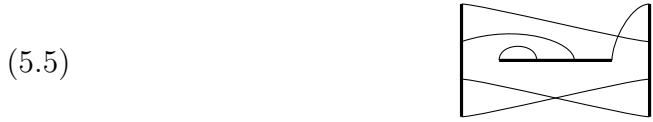
\begin{equation}
\begin{tikzpicture}[scale=0.5,baseline=17]
\draw[very thick] (0,0)--(0,3);
\draw[very thick] (1,3/2)--(4,3/2);
\draw[very thick] (5,0)--(5,3);
\draw(1,3/2)..controls (1,2) and (2,2)..(2,3/2);
\draw(3,3/2)..controls (3,2) and (3/2,5/2)..(0,2);
\draw(4,3/2)..controls (4,2) and (4.5,3)..(5,3);
\draw(0,3)..controls (1/2,3) and (4.5,2)..(5,2);
\draw(0,1)..controls (1/2,1) and (4.5,0)..(5,0);
\draw(0,0)..controls (1/2,0) and (4.5,1)..(5,1);
\end{tikzpicture}
\end{equation}
where the figure-eight knot is represented by the edge of the central tetrahedron connecting the maximal and next to maximal vertices. 
Choosing an orientation, the diagram consists of two positive tetrahedra $T_1,T_3$ and one negative $T_2$. $\partial{X}=\emptyset$ and combinatorially we have 
$\Delta_{0}(X)=\{ * \}$, $\Delta_{1}(X)=\{x,y,z,x'\}$. The gluing of the tetrahedra is vertex order preserving which means that edges are glued together in the following manner.
\begin{align*}
x&=x_{01}^{1}=x_{03}^{1}=x_{02}^{2}=x_{02}^{3}=x_{03}^{3}, \\
y&=x_{02}^{1}=x_{12}^{1}=x_{13}^{1}=x_{01}^{2}=x_{03}^{2}=x_{23}^{2}=x_{23}^{2}, \\
z&=x_{23}^{1}=x_{12}^{2}=x_{13}^{2}=x_{12}^{3}=x_{13}^{3},\\
x'&=x_{01}^{3}.
\end{align*}
This results in the following equations for the dihedral angles when we balance all but one edge
\begin{equation*}
b_1+a_3=b_2, \quad a_1=a_2+a_3.
\end{equation*}
In the limit where we let $a_3\to 0$ we get the equations
$$ b_1=b_2, \quad a_1=a_2.$$
The Boltzmann weights are given by the functions
\begin{align*}
&B\left(T_1,x_{\vert_{\Delta_{1}(T_1)}}\right)=g_{a_1,c_1}(y-x,2y-x-z), \\
&B\left(T_2,x_{\vert_{\Delta_{1}(T_2)}}\right)=\overline{g_{a_2,c_2}(x-y,x+z-2y)}, \\
&B\left(T_3,x_{\vert_{\Delta_{1}(T_3)}}\right)=g_{a_3,c_3}(0,x+z-x'-y).
\end{align*}
So we get that
\begin{align*}
Z^{\text{new}}_{\hbar}(X)=\int_{[0,1]^4}\sum_{m,n,l\in \Z} &\tpsi{a_1}{c_1}{y-x+m} e^{\pi i (2y-x-z)(y-x+2m)}\\
&\overline{\tpsi{a_2}{c_2}{x-y+n}}e^{-\pi i (x+z-2y)(x-y+2n)} \\
&\tpsi{a_3}{c_3}{l} e^{2\pi i (x+z-x'-y)l} \: dxdydzdx'.
\end{align*}
Integration over $x'$ removes one of the sums since $\int_{0}^{1}e^{-2\pi i x'l}dx'=\delta(l)$. Hence
\begin{align*}
Z^{\text{new}}_{\hbar}(X)=\tpsi{a_3}{c_3}{0}\int_{[0,1]^3}\sum_{m,n\in \Z} &\tpsi{a_1}{c_1}{y-x+m} e^{\pi i (2y-x-z)(y-x+2m)}\\
&\overline{\tpsi{a_2}{c_2}{x-y+n}}e^{-\pi i (x+z-2y)(x-y+2n)} \: dxdydz \\
=\tpsi{a_3}{c_3}{0}\int_{[0,1]^3}\sum_{m,n\in \Z} &\tpsi{a_1}{c_1}{y-x+m} \overline{\tpsi{a_2}{c_2}{x-y+n}}\\
&e^{2\pi i (2y-x)(m+n)}e^{-2\pi i z(m+n)} \: dxdydz.
\end{align*}
Now integration over $z$ gives $\int_0^1 e^{-2\pi i z(m+n)}dz=\delta(n+m).$ So the partition function takes the form

\begin{align*}
Z^{\text{new}}_{\hbar}(X)=\tpsi{a_3}{c_3}{0}\int_{[0,1]^2}\sum_{m\in \Z} &\tpsi{a_1}{c_1}{y-x+m} \overline{\tpsi{a_2}{c_2}{x-y-m}} \: dxdy,
\end{align*}
We make the shift $y \mapsto y+x$ to get the expression
\begin{align*}
Z^{\text{new}}_{\hbar}(X)&= \tpsi{a_3}{c_3}{0}\int_{[0,1]^2}\sum_{m\in \Z} \tpsi{a_1}{c_1}{y+m} \overline{\tpsi{a_2}{c_2}{-y-m}} \: dxdy \\
&= \tpsi{a_3}{c_3}{0}\int_{[0,1]}\sum_{m\in \Z} \tpsi{a_1}{c_1}{y+m} \overline{\tpsi{a_2}{c_2}{-y-m}} \: dy \\
&= \tpsi{a_3}{c_3}{0}\int_{\Z}\tpsi{a_1}{c_1}{y} \overline{\tpsi{a_2}{c_2}{-y}} \: dy \\
&= e^{-\frac{\pi i}{6}}\tpsi{a_3}{c_3}{0}\int_{\Z} \psi_{c_1,b_1}(y)\psi_{b_2,c_2}(y)e^{\pi i y^2}.
\end{align*}
We set $Y=y-2c_{\bq}(c_1+b_1)=y-c_{\bq}(1-2a_1)$. Assuming that we are in the case where all but one edge is balanced we have $a_1=a_2$  
$$ y^2=Y^2+c_{\bq}^2(1-2a_1)^2 +2c_{\bq}Y(1-2a_1).$$
Implementing this we get the following expression
\begin{align*}
Z^{\text{new}}_{\hbar}(X)= e^{-\frac{\pi i}{6}}\tpsi{a_3}{c_3}{0}\int_{\Z} &\psi(Y)\psi(Y)e^{\pi i (Y^2+c_{\bq}^2(1-2a_1)^2 +2c_{\bq} 
	Y(1-2a_1))} \\
&e^{-4\pi i c_{\bq} c_1(Y+c_{\bq}(1/2-a_1))}\nu_{c_1,b_1}\\
&e^{-4\pi i c_{\bq} b_2(Y+c_{\bq}(1/2-a_1))}\nu_{b_2,c_2} dy \\
=e^{-\frac{\pi i}{6}}\nu_{c_1,b_1}\nu_{b_2,c_2}\tpsi{a_3}{c_3}{0}\int_{\Z-0i} &\frac{1}{\Phi(Y)^2}e^{\pi i Y^2}\: dy \:e^{\frac{i\phi}{\hbar}}.
\end{align*}
This result corresponds exactly to the partition function in the original formulation, see \cite[Chap. 11]{AK1}.
I.e. in the limit where $a_3\to 0 $ we get the renormalised partition function
\begin{equation*}
\tilde{Z}^{\text{new}}_{\hbar}(X):=\lim_{a_3\to 0}\Phi_{\bq}(2c_{\bq}a_3-c_{\bq})Z^{\text{new}}_{\hbar}(X)=\frac{e^{-\pi i /12}}{\nu(c_3)}\chi_{4_1}(0).
\end{equation*}

\subsection{The complement of the knot $5_2$}
Let $X$ be represented by the diagram 
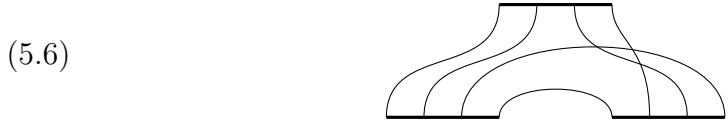
\begin{equation}
\begin{tikzpicture}[scale=.5,baseline=20]
\draw[very thick] (0,0)--(3,0);
\draw[very thick] (6,0)--(9,0);
\draw[very thick] (3,3)--(6,3);

\draw(3,0)..controls (3,1) and (6,1)..(6,0);
\draw(0,0)..controls (0,2) and (3,1)..(3,3);
\draw(1,0)..controls (1,2) and (4,1)..(4,3);
\draw(2,0)..controls (2,2.5) and (9,2.5)..(9,0);
\draw(5,3)..controls (5,1) and (8,2)..(8,0);
\draw(6,3)..controls (6,2) and (7,2)..(7,0);
\end{tikzpicture}
\end{equation}
Choosing an orientation the diagram consists of three positive tetrahedra. We denote $T_1,T_2,T_3$ the left, the right an top tetrahedra respectively. The combinatorial data in this case are $\Delta_{0}(X)=\mgd{*}$, $\Delta_1(X)=\mgd{e_0,e_1,e_2}$, $\Delta_2(X)=\mgd{f_0,f_1,f_2,f_3,f_4,f_5}$ and $\Delta_3(X)=\mgd{T_1,T_2,T_3}.$ 

The edges are glued in the following manner
\begin{align*}
e_0 & = x_{02}^{1}=x_{12}^{1}=x_{13}^{2}=x_{23}^{2}=x_{01}^{3}=x_{23}^{3} =: x, \\ 
e_1 & = x_{03}^{1}=x_{23}^{1}=x_{02}^{2}=x_{03}^{2}=x_{03}^{3}=x_{13}^{3}=x_{12}^{3}=:y \\
e_2 & = x_{01}^{1}=x_{13}^{1}=x_{01}^{2}=x_{12}^{2}=x_{02}^{3}=: z.\\
\end{align*}
We impose the condition that all edges are balanced which exactly corresponds to the two equations 
$$ 2a_3=a_1+c_2, \quad b_3=c_1+b_2.$$
The Bolzmann weights are given by the equations 
\begin{align*}
B\left(T_{1},x_{\vert_{\Delta_{1}(T_1)}}\right)&=g_{a_1,c_1}(z-y,x-y), \\
B\left(T_{2},x_{\vert_{\Delta_{1}(T_2)}}\right)&=g_{a_2,c_2}(x-z,y-z), \\
B\left(T_{2},x_{\vert_{\Delta_{1}(T_2)}}\right)&={g_{a_3,c_3}(z-y,z+y-2x)}. 
\end{align*}
We calculate the following function
\begin{align*}
Z^{\text{new}}_{\hbar}(X)=&\int_{[0,1]^3} \sum_{j,k,l \in \Z} g_{a_1,c_1}(z-y,x-y) g_{a_2,c_2}(x-z,y-z) \\
& \phantom{uuuuuuuuuu} \times g_{a_3,c_3}(z-y,z+y-2x) \; dxdydz \\
=& \int_{[0,1]^3}\sum_{j,k,l \in \Z} \tpsi{a_1}{c_1}{z-y+j}e^{\pi i (x-y)(z-y+2j)} \tpsi{a_2}{c_2}{x-z+k} \\
& e^{\pi i (y-z)(x-z+2k)}  \times \tpsi{a_3}{c_3}{z-y+l}e^{\pi i (z+y-2x)(z-y+2l)} \; dxdydz .
\end{align*}
Shift $x\mapsto x+z,$
\begin{align*}
Z^{\text{new}}_{\hbar}(X)=& \int_{[0,1]^3}\sum_{j,k,l \in \Z} \tpsi{a_1}{c_1}{z-y+j}e^{\pi i (x+z-y)(z-y+2j)} \tpsi{a_2}{c_2}{x+k}e^{\pi i (y-z)(x+2k)} \\ &\times \tpsi{a_3}{c_3}{z-y+l}e^{\pi i (y-2x-z)(z-y+2l)} \; dxdydz .
\end{align*}
Shift $z\mapsto z+y$
\begin{align*}
Z^{\text{new}}_{\hbar}(X)=& \int_{[0,1]^3}\sum_{j,k,l \in \Z} \tpsi{a_1}{c_1}{z+j}e^{\pi i (x+z)(z+2j)} \tpsi{a_2}{c_2}{x+k}e^{\pi i (-z)(x+2k)} \\ &\times \tpsi{a_3}{c_3}{z+l}e^{\pi i (-2x-z)(z+2l)} \; dxdydz \\
=& \int_{[0,1]^3}\sum_{j,k,l \in \Z} \tpsi{a_1}{c_1}{z+j} \tpsi{a_2}{c_2}{x+k} \tpsi{a_3}{c_3}{z+l}\\ &\times e^{\pi i (x+z)(z+2j)}e^{\pi i (-z)(x+2k)}e^{\pi i (-2x-z)(z+2l)} \; dxdydz \\
=& \int_{[0,1]^3}\sum_{j,k,l \in \Z} \tpsi{a_1}{c_1}{z+j} \tpsi{a_2}{c_2}{x+k} \tpsi{a_3}{c_3}{z+l}\\ &\times e^{2\pi i (x(j-2l-z)+z(j-k-l)} \; dxdydz.
\end{align*}
Integration over $y$ contributes nothing. We now shift $x\mapsto x-k$ and integrate over the interval $[-k,-k+1]$. 
\begin{align*}
Z^{\text{new}}_{\hbar}(X)=&\sum_{j,k,l\in \Z}\int_{[0,1]}\int_{[-k,-k+1]} \tpsi{a_1}{c_1}{z+j} \tpsi{a_2}{c_2}{x} \tpsi{a_3}{c_3}{z+l}\\ &\times e^{2\pi i ((x-k)(j-2l-z)+z(j-k-l)} \; dxdz \\
=& \sum_{j,k,l\in \Z} e^{2\pi i k(2l-j)}\int_{[0,1]}\int_{[-k,-k+1]}\tpsi{a_1}{c_1}{z+j} \tpsi{a_2}{c_2}{x} \tpsi{a_3}{c_3}{z+l}\\ &\times e^{2\pi i (x(j-2l-z)+z(j-l))} \; dxdz \\
=& \sum_{j,l\in \Z} \int_{[0,1]}\tpsi{a_1}{c_1}{z+j}  \tpsi{a_3}{c_3}{z+l}e^{2\pi i z(j-l)}\int_{\R}\tpsi{a_2}{c_2}{x} e^{-2\pi i x(z+2l-j)}dx\:dz \\
=& e^{-\frac{\pi i}{12}} \sum_{j,l \in \Z} \int_{[0,1]}\tpsi{a_1}{c_1}{z+j}  \tpsi{a_3}{c_3}{z+l} e^{2\pi i z(j-l)} \int_{\R} \ps{c_2}{b_2}{x} e^{-2\pi i x(z+2l-j)}dx\:dz \\
=& e^{-\frac{\pi i}{4}} \sum_{j,l\in \Z} \int_{[0,1]} \ps{c_1}{b_1}{z+j} \ps{c_3}{b_3}{z+l}\tilde{\psi}_{c_2,b_2}(z+2l-j) e^{2\pi i z(j-l)} dz.
\end{align*} 
We set $m=j-l$. 
\begin{align*}
Z^{\text{new}}_{\hbar}(X)=&e^{-\frac{\pi i}{4}} \sum_{l,m\in \Z} \int_{[0,1]} \ps{c_1}{b_1}{z+l+m} \ps{c_3}{b_3}{z+l}\tilde{\psi}_{c_2,b_2}(z+l-m) e^{2\pi i zm} \;dz \\
=& e^{-\frac{\pi i}{4}} \sum_{l,m\in \Z} \int_{[l,l+1]} \ps{c_1}{b_1}{z+m} \ps{c_3}{b_3}{z}\tilde{\psi}_{c_2,b_2}(z-m) e^{2\pi i zm} \;dz \\
=& e^{-\frac{\pi i}{3}}\sum_{m\in \Z} \int_{\R} \ps{c_1}{b_1}{z+m} \ps{c_3}{b_3}{z}\psi_{b_2,a_2}(z-m) e^{\pi i (z-m)^2}e^{2\pi i zm} \;dz \\
=& e^{-\frac{\pi i}{3}}\sum_{m\in \Z} \int_{\R} \ps{c_1}{b_1}{z+m} \ps{c_3}{b_3}{z}\psi_{b_2,a_2}(z-m) e^{\pi i (z^2+m^2)} \;dz. 
\end{align*}
\begin{align*}
Z^{\text{new}}_{\hbar}(X)=e^{-\frac{\pi i}{3}}\sum_{m\in \Z} \int_{\R} &\psi(z+m-c_{\bq}(1-2a_1))e^{-4\pi i c_{\bq}c_1\{(z+m)-c_{\bq}(1/2-a_1)\}} \\ &e^{-\pi i c_{\bq}^2(4(c_1-b_1)+1)/6}\\
&\psi(z-m-c_{\bq}(1-2c_2))e^{-4\pi i c_{\bq}c_1\{(z+m)-c_{\bq}(1/2-c_2)\}} \\ &e^{-\pi i c_{\bq}^2(4(c_1-b_1)+1)/6}\\
&\psi(z-c_{\bq}(1-2a_3))e^{-4\pi i c_{\bq}c_3\{(z+m)-c_{\bq}(1/2-a_3)\}}\\ &e^{-\pi i c_{\bq}^2(4(c_3-b_3)+1)/6}
e^{\pi i z^2}e^{\pi i p^2} \; dz.
\end{align*}
Set $w=z-c_{\bq}(1-2a_3)$
\begin{align*}
Z^{\text{new}}_{\hbar}(X)=e^{-\frac{\pi i }{3}} \sum_{m\in \Z} \int_{\R - i0} &\psi(w+m+2c_{\bq}(a_1-a_3)) \psi(w-m+2c_{\bq}(c_2-a_3)) \psi(w) \\ \times& e^{\pi i p^2} e^{\pi i w^2}e^{4\pi i c_{\bq}^2(1/2-a_3)^2}e^{4\pi i c_{\bq} w (1/2-a_3)} \\
& e^{-4\pi i c_{\bq}c_1\{ w+p+c_{\bq}(1-2a_3)-c_{\bq}(1/2-a_1) \}} \\
& e^{-4\pi i c_{\bq}c_1\{ w-p+c_{\bq}(1-2a_3)-c_{\bq}(1/2-c_2) \}} \\
& e^{-4\pi i c_{\bq}c_1\{ w+c_{\bq}(1/2-a_3) \}} \nu_{c_1,b_1}\nu_{b_2,a_2}\nu_{c_3,b_3} \; dw.
\end{align*}
Simplify by setting $u=2c_{\bq}(a_1-a_3).$ Using $c_1+b_2+c_3+a_3-1/2 = 0$ we are left with
\begin{align*}
Z^{\text{new}}_{\hbar}(X)= &e^{-\frac{\pi i }{3}} \sum_{m\in \Z} \int_{\R-i0} \psi(w+m+u)\psi(w-m-u)\psi(w) \\
& e^{\pi i w^2} e^{\pi i m^2} e^{4\pi i c_{\bq} (b_2-c_1)m} \\
& e^{-4\pi i c_{\bq}^2\{ -b_3^2-b_3c_3 +c_1(b_3+c_3)+b_2(b_3+c_3)+ (c_1-b_2)(a_1-a_3) \}}\; dw \\
&\nu_{c_1,b_1}\nu_{b_2,a_2}\nu_{c_3,b_3}.
\end{align*}
Let $v=2c_{\bq}(a_1-c_1+b_2-a_3),$ then
Note that 
$$ 4\pi i c_{\bq} (b_2-c_1) p = 4\pi i c_{\bq} (a_1-c_1+b_2-a_3) p -4\pi ic_{\bq}(a_3-a_1) =2\pi i (vp - up), $$
$$ -b_3^2-b_3c_3 +c_1(b_3+c_3)+b_2(b_3+c_3) = 0 ,$$
and
\begin{align*} -4 \pi i c_{\bq}^2 ((c_1-b_2)(a_1-a_3)) &= 4 \pi i c_{\bq}^2 ((a_1-c_1+b_2-a_3)(a_1-a_3) -(a_1-a_3)(a_1-a_3)) \\  &= \pi i (vu-u^2).
\end{align*}
\begin{align*}
Z^{\text{new}}_{\hbar}(X)=&e^{-\frac{\pi i}{3}} e^{\pi i uv} \sum_{m\in \Z} \int_{\R-i0} \frac{e^{\pi i w^2}e^{-\pi i m^2}e^{-\pi i u^2}}{\Phi_{\bq}(w+m+u)\Phi_{\bq}(w-m-u)\Phi_{\bq}(w)} dw e^{2 \pi i vm} \\
&\nu_{c_1,b_1}\nu_{b_2,a_2}\nu_{c_3,b_3}\\
=& e^{-\frac{\pi i}{3}} e^{\pi i uv} \sum_{m\in \Z} \int_{\R-i0}\frac{e^{\pi i (w+(u+m))(w-(u+m))}}{\Phi_{\bq}(w+m+u)\Phi_{\bq}(w-m-u)\Phi_{\bq}(w)} dw e^{2 \pi i vm} \\
&\nu_{c_1,b_1}\nu_{b_2,a_2}\nu_{c_3,b_3}\\
=&W\chi_{5_2}(u,v)\nu_{c_1,b_1}\nu_{b_2,a_2}\nu_{c_3,b_3}.
\end{align*}
Where $\chi_{5_2}(u)$ is given by the formula
$$ \chi_{5_2}(u)=e^{-\frac{\pi i}{3}}\int_{\R - i0} \frac{e^{\pi i (w-u) (w+u)}}{\Phi_{\bq}(w+m+u)\Phi_{\bq}(w-m-u)\Phi_{\bq}(w)} \;dw.$$
Again the function $\chi_{5_2}$ is that of \cite{AK1}, which again is related to Hikami's invariant,
in particular Hikami's formally derived expression  in \cite[(4.10)]{HI2} is equal to $e^{\pi i \frac{c_{\bq^2}}{3}}\frac{1}{2\pi \bq}\chi_{5_2}(\frac{-u}{\pi \bq},\frac{1}{2})$, where $\chi_{5_2}:= \chi_{5_2}(x)e^{4\pi i c_{\bq}x\lambda}.$

\subsection{One vertex H-triangulation of $(S^3,5_2)$}
Let $X$ be represented by the diagram 
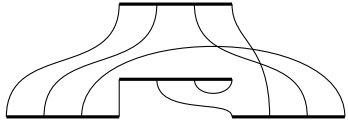
\begin{equation}
\begin{tikzpicture}[scale=.5,baseline=20]
\draw[very thick] (0,0)--(3,0);
\draw[very thick] (3,1)--(6,1);
\draw[very thick] (6,0)--(9,0);
\draw[very thick] (3,3)--(6,3);
\draw(5,1)..controls (5,1/2) and (6,1/2)..(6,1);
\draw (3,0)--(3,1);
\draw(4,1)..controls (4,0) and (6,.5)..(6,0);
\draw(0,0)..controls (0,2) and (3,1)..(3,3);
\draw(1,0)..controls (1,2) and (4,1)..(4,3);
\draw(2,0)..controls (2,2.5) and (9,2.5)..(9,0);
\draw(5,3)..controls (5,1) and (8,2)..(8,0);
\draw(6,3)..controls (6,2) and (7,2)..(7,0);
\end{tikzpicture}
\end{equation}
Choosing an orientation, the diagram consists of four positive tetrahedra $T_0,T_1,T_2,T_3$. $\partial{X}=\emptyset$ and combinatorially we have 
$\Delta_{0}(X)=\{ * \}$, $\Delta_{1}(X)=\{x,y,z,w,x'\}$. The gluing of the tetrahedra is vertex order preserving which means that edges are glued together in the following manner.
\begin{align*}
x&=x_{03}^{0}=x_{13}^{0}=x_{01}^{1}=x_{12}^{3}=x_{02}^{3}, \\
y&=x_{03}^{1}=x_{12}^{1}=x_{13}^{1}=x_{02}^{2}=x_{03}^{2}=x_{03}^{3}=x_{23}^{3}, \\
z&=x_{01}^{0}=x_{02}^{1}=x_{01}^{2}=x_{12}^{2}=x_{01}^{3}=x_{13}^{3} \\
v&=x_{02}^{0}=x_{12}^{0}=x_{23}^{1}=x_{13}^{2}=x_{23}^{2},\\
x'&=x_{23}^{0}.
\end{align*}
This results in the following equations for the dihedral angles, when we balance all edges but one edge.
\begin{equation*}
a_3=a_1-a_0=c_2,\quad a_0+b_1=b_2+c_3, \quad a_1+a_2+b_3=\frac{1}{2}+c_1.
\end{equation*}
The Boltzmann weights are given by the functions
\begin{align*}
&B\left(T_0,x_{\vert_{\Delta_{1}(T_0)}}\right)=g_{a_0,c_0}(0,v+x-z-x'), \\
&B\left(T_1,x_{\vert_{\Delta_{1}(T_1)}}\right)=g_{a_1,c_1}(z-y,z+y-x-v), \\
&B\left(T_2,x_{\vert_{\Delta_{1}(T_2)}}\right)={g_{a_2,c_2}(v-z,y-z)}, \\
&B\left(T_3,x_{\vert_{\Delta_{1}(T_3)}}\right)=g_{a_3,c_3}(z-y,x-y).
\end{align*}
The partition function is represented by the integral
\begin{align*}
Z^{\text{new}}_{\hbar}(X)=\int_{[0,1]^5}\sum_{m,n,k,p\in \Z} &\tpsi{a_0}{c_0}{m} e^{\pi i (v+x-z-x')(2m)}\\
&{\tpsi{a_1}{c_1}{z-y+n}}e^{\pi i (z+y-x-v)(z-y+2n)} \\
&{\tpsi{a_2}{c_2}{v-z+k}}e^{\pi i (y-z)(v-z+2k)} \\
&\tpsi{a_3}{c_3}{z-y+p} e^{\pi i (x-y)(z-y+2p)} \: dx'dxdydzdv. \\
\end{align*}
Integration over $x'$ removes one of the sums since $\int_{0}^{1}e^{-2\pi i x'm}dx'=\delta(m)$. Hence
\begin{align*}
Z^{\text{new}}_{\hbar}(X)=\tpsi{a_0}{c_0}{0}\int_{[0,1]^4} \sum_{n,k,p\in \Z} & {\tpsi{a_1}{c_1}{z-y+n}}e^{\pi i (z+y-x-v)(z-y+2n)} \\
&{\tpsi{a_2}{c_2}{v-z+k}}e^{\pi i (y-z)(v-z+2k)} \\
&\tpsi{a_3}{c_3}{z-y+p} e^{\pi i (x-y)(z-y+2p)} \: dx'dxdydzdv. \\
\end{align*}
Now integration over $x$ gives $\int_0^1 e^{-2\pi i x(n-p)}dx=\delta(n-p).$ Implementing this and shifting the variable $v\mapsto v+z$, the partition function takes the form
\begin{align*}
Z^{\text{new}}_{\hbar}(X)=\tpsi{a_0}{c_0}{0}\int_{[0,1]^3}\sum_{n,k\in \Z} &\tpsi{a_1}{c_1}{z-y+n} e^{\pi i(y-v)(z-y+2n)} \\
&\tpsi{a_2}{c_2}{v+k}e^{\pi i (y-z)(v+2k)} \\
&\tpsi{a_3}{c_3}{z-y+n}e^{-\pi i y(z-y+2n)} \; dydzdv.
\end{align*}
We make the shift $z \mapsto z+y$ to get the expression
\begin{align*}
Z^{\text{new}}_{\hbar}(X)=\tpsi{a_0}{c_0}{0}\int_{[0,1]^3}\sum_{n,k\in \Z} &\tpsi{a_1}{c_1}{z+n} e^{\pi i(y-v)(z+2n)} \\
&\tpsi{a_2}{c_2}{v+k}e^{-\pi i z(v+2k)} \\
&\tpsi{a_3}{c_3}{z+n}e^{-\pi i y(z+2n)} \; dydzdv,
\end{align*}
which is independent of $y$ so we can remove the integration over this variable.
We integrate over the variable $v$. 
\begin{align*}
\sum_{k\in \Z} \int_{[0,1]}\tpsi{a_2}{c_2}{v+k} e^{-2\pi i v(z+n)} dv e^{-2\pi i zk} =& \sum_{k\in \Z} \int_{k}^{k+1} \tpsi{a_2}{c_2}{v} e^{-2 \pi i v (z+n)} dv \\ &e^{-2\pi i zk}e^{2\pi i k(z+n)} \\
=& e^{-\frac{\pi i }{12}} \int_{\R} \psi_{c_2,b_2}(v) e^{-2\pi i z(z+n)} dv \\ =& e^{-\frac{\pi i}{12}} \tilde{\psi}_{c_2,b_2}(z+n) \\
=&e^{-\frac{\pi i }{6}}e^{\pi i (z+n)^2}\psi_{b_2,a_2}(z+n).
\end{align*}
We therefore get the expression
\begin{align*}
Z^{\text{new}}_{\hbar}(X)&= e^{-\frac{\pi i }{3}} \tpsi{a_0}{c_0}{0}\int_{[0,1]} \sum_{n\in \Z} \psi_{c_1,b_1}(z+n) \psi_{b_2,a_2}(z+n) \psi_{c_3,b_3}(z+n)e^{\pi i (z+n)^2} dz \\
&= e^{-\frac{\pi i }{3}} \tpsi{a_0}{c_0}{0}\int_{\R}  \psi_{c_1,b_1}(z) \psi_{b_2,a_2}(z) \psi_{c_3,b_3}(z)e^{\pi i (z)^2}
\end{align*}
We set $Z=z-2c_{\bq}(c_1+b_1)=y-c_{\bq}(1-2a_1)$. Assuming that we are in the case where all but the edge representing the knot is balanced, i.e. $a_0 \to 0$, we have $a_1=c_2=a_3$.  
$$ z^2=Z^2+c_{\bq}^2(1-2a_1)^2 +2c_{\bq}Z(1-2a_1).$$
Implementing this we get the expression.
\begin{align*}
Z^{\text{new}}_{\hbar}(X)= e^{-\frac{\pi i}{3}} \tpsi{a_0}{c_0}{0}\int_{\R} &\psi(Z)\psi(Z)\psi(Z)e^{\pi i (Z^2+c_{\bq}^2(1-2a_1)^2 +2c_{\bq} 
	Z(1-2a_1))} \\
&e^{-4\pi i c_{\bq} c_1(Z+c_{\bq}(1/2-a_1))}\nu_{c_1,b_1}\\
&e^{-4\pi i c_{\bq} b_2(Z+c_{\bq}(1/2-c_2))}\nu_{b_2,a_2}  \\
&e^{-4\pi i c_{\bq} c_2(Z+c_{\bq}(1/2-a_3))}\nu_{c_3,b_\Z3} dz. 
\end{align*}
\begin{align*}
Z^{\text{new}}_{\hbar}(X)&=\nu_{c_1,b_1}\nu_{b_2,a_2}\nu_{c_3,b_3}e^{-\frac{\pi i}{3}} e^{\frac{\phi i}{\hbar}}\tpsi{a_0}{c_0}{0}\int_{\R} \psi(Z)^3e^{\pi i Z^2} \;dz \\
&=\nu_{c_1,b_1}\nu_{b_2,a_2}\nu_{c_3,b_3}e^{-\frac{\pi i}{3}} e^{\frac{\phi i}{\hbar}}\tpsi{a_0}{c_0}{0}\int_{\R} \frac{e^{\pi i Z^2}}{\Phi_{\bq}(Z)^{3}} \;dz 
\end{align*}
Because the combination of dihedral angles in front of $Z$ sums to $0$.
$$ -4\pi i c_{\bq} Z (c_1+b_2+c_3-\frac{1}{2}+a_1)=-4\pi i c_{\bq} Z (a_1+b_1+c_1-\frac{1}{2})=0 $$
This corresponds to the partition function in the original formulation, see \cite{AK1}.

In this case the renormalised partition function takes the form
\begin{equation*}
\tilde{Z}^{\text{new}}_{\hbar}(X)=\lim_{a_0\to 0} \Phi_{\bq}{Z}^{\text{new}}_{\hbar}(X)=\frac{e^{i\pi /4}}{\nu_{c_0,0}}\chi_{5_2}(0).
\end{equation*}

\subsection{One-vertex H-triangulation of $(S^3,6_1)$}
Let $X$ be represented by the diagram 
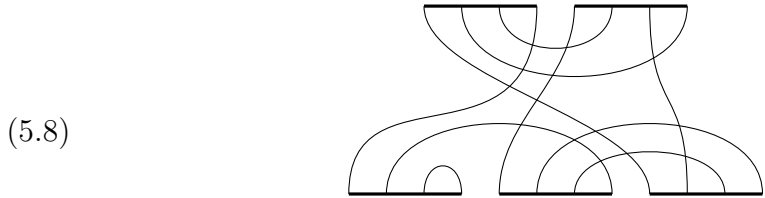
\begin{equation}\label{diagram:61}
\begin{tikzpicture}[scale=.5,baseline=20]
\draw[very thick] (0,0)--(3,0);
\draw[very thick] (4,0)--(7,0);
\draw[very thick] (8,0)--(11,0);
\draw[very thick] (2,5)--(5,5);
\draw[very thick] (6,5)--(9,5);
\draw(0,0)..controls (0,3.5) and (5,.5)..(5,5);
\draw(2,0)..controls (2,1) and (3,1)..(3,0);
\draw(1,0)..controls (1,2.5) and (7,2.5)..(7,0);
\draw(4,0)..controls (4,2) and (6,3)..(6,5);
\draw(5,0)..controls (5,2.5) and (11,2.5)..(11,0);
\draw(6,0)..controls (6,1.5) and (10,1.5)..(10,0);
\draw(8,0)..controls (8,2) and (2,3)..(2,5);
\draw(9,0)..controls (9,3) and (8,2)..(8,5);
\draw(3,5)..controls (3,2.5) and (9,2.5)..(9,5);
\draw(4,5)..controls (4,3.5) and (7,3.5)..(7,5);
\end{tikzpicture}
\end{equation}

This one vertex $H$-triangulation of $(S^3,6_1)$ consists of 5 tetrahedra $T_1$ and $T_3$ which are negatively oriented tetrahedra and $T_2,T_4,T_5$ which are positively oriented tetrahedra. We denoted the tetrahedra as follows. In the bottom we have $T_1,T_2,T_3$ from left to right and on top we have $T_4,T_5$ from right to left.

In the diagram the knot $6_1$ is represented by the edge connecting the maximal and next to maximal vertex of $T_1$. We impose a shape structure on the triangulation and balance all but the one edge representing the knot. We get the following equations on the shape parameters
\begin{align*}
& a_3 = a_1+c_2,\quad a_3+a_4=a_1+a_5, \quad a_1+ c_2 = c_4 + c_5, \\
&\frac{1}{2}+b_3+c_5 = a_2 + a_3 + a_4, \quad 1= a_2+c_3+c_4+a_5.
\end{align*}
We calculate the partition function for the Teichmüller TQFT using the original formulation of the theory. In this formulation the states are assigned to each face of each tetrahedron according to the diagram \eqref{diagram:61}.
\begin{align*}
Z_{\hbar}(X) = \int_{\R^{10}} &\overline{\left< w,t \:\vert \: T_{a_1,c_1}\: \vert \:u ,t \right>}\left<z ,q \:\vert \: T_{a_2,c_2}\: \vert \: v,u \right> \overline{\left< x, q\:\vert \:T_{a_3,c_3} \: \vert \: r, v\right>} \\ &\left< s, y\:\vert \: T_{a_4,c_4}\: \vert \:r ,z \right>\left< w,x \:\vert \:T_{a_5,c_5} \: \vert \:y ,s \right> d{\bar{x}} 
\end{align*}
\begin{align*}
Z_{\hbar}(X)= \int_{\R^{10}}& \delta(w+t-u)\delta(z+q-v)\delta(x+q-r)\delta(s+y-r)\delta(w+x-y) \\
& \overline{\tpsi{a_1}{c_1}{0}} e^{-2\pi i w(0)} \\
& \tpsi{a_2}{c_2}{u-q}e^{2\pi i z(u-q)} \\
& \overline{\tpsi{a_3}{c_3}{v-q}}e^{-2\pi i x(v-q)} \\
& \tpsi{a_4}{c_4}{x-y}e^{2\pi i s(z-y)} \\
& \tpsi{a_5}{c_5}{s-x}e^{2\pi i w(s-x)} \: dqdrdsdtdudvdxdwdzdy
\end{align*}
Integrating over five variables $t,v,r,y,w$ yields the expression
\begin{align*}
Z_{\hbar}(X)
=\overline{\tpsi{a_1}{c_1}{0}} \int_{\R^{5}} & \tpsi{a_2}{c_2}{u-q}e^{2\pi i z(u-q)} \\
\times& \overline{\tpsi{a_3}{c_3}{z}}e^{-2\pi i xz} \\
\times& \tpsi{a_4}{c_4}{z+s-x-q}e^{2\pi i s(z+s-x-q)} \\
\times &\tpsi{a_5}{c_5}{s-x}e^{2\pi i (q-s)(s-x)} \: dqdsdudxdz.
\end{align*}
We integrate over the variable $u$ using the Fourier transform.
\begin{equation*}
e^{-\frac{\pi i }{12}}\int_{\R} \psi_{{c_2},{b_2}}({u-q}) e^{2\pi i z (u-q)} du = e^{-\frac{\pi i }{12}}\tilde{\psi}_{{c_2},{b_2}}(-z)=e^{-\frac{\pi i }{6}}e^{\pi i z^2}\psi_{b_2,a_2}(-z).
\end{equation*}
Using formulas from Section \ref{CTOPI} we can write
\begin{align*}
Z_{\hbar}(X)
=e^{-\frac{3\pi i }{12}}\overline{\tpsi{a_1}{c_1}{0}} \int_{\R^{4}} & {\psi}_{{b_2},{a_2}}(-z)e^{\pi i z^2} \psi_{{b_3},{c_3}}(-z)e^{\pi i z^2} \\
& \tpsi{a_4}{c_4}{z+s-x-q} \tpsi{a_5}{c_5}{s-x}\\
&e^{2\pi i (sz-xz-qx)} \: dqdsdxdz.
\end{align*}
Integration over the variable $q$ gives
\begin{align*}
\int_{\R} \tpsi{a_4}{c_4}{z+s-x-q}e^{-2\pi i qx} dq &= e^{-\frac{\pi i }{12}} \tilde\psi_{c_4,b_4}(-x) e^{2\pi i (x^2-xz-sx)} \\
&=e^{-\frac{\pi i}{6}} \psi_{b_4,a_4}(-x)e^{2 \pi i( \frac{3}{2}x^2-xz-sx)}
\end{align*}
\begin{align*}
Z_{\hbar}(X) =e^{-\frac{5\pi i }{12}}\overline{\tpsi{a_1}{c_1}{0}} \int_{\R^{2}} & {\psi}_{{b_2},{a_2}}(-z) {\psi_{{b_3},{c_3}}(-z)} \psi_{{b_4},{a_4}}(-x) \tpsi{a_5}{c_5}{s-x} \\
&e^{2\pi i (sz-2xz+z^2+\frac{3}{2}x^2-sx)} \: dxdsdz \\
\end{align*}
Integration over $s$ now gives
\begin{align*}
e^{-\frac{\pi i }{12}}\int_{\R}\psi_{c_5,b_5}(s-x)e^{-2\pi i s(x-z)} ds &=e^{-\frac{\pi i }{12}}\tilde\psi_{c_5,b_5}(x-z)e^{-2\pi i (x^2 -xz)} \\
= &e^{-\frac{\pi i}{6}}\psi_{b_5,a_5}(x-z)e^{\pi i (x-z)^2}e^{-2\pi i (x^2 -xz)}.
\end{align*}
So the partition function takes the form
\begin{multline*}
Z_{\hbar}(X) =e^{-\frac{7\pi i }{12}}\overline{\tpsi{a_1}{c_1}{0}} \int_{\R^{2}}  {\psi}_{{b_2},{a_2}}(-z) {\psi_{{b_3},{c_3}}(-z)} \psi_{{b_4},{a_4}}(-x) \psi_{{b_5},{a_5}}(x-z) \\
e^{2\pi i (-2xz+\frac{3}{2}z^2+x^2)} \: dxdz.
\end{multline*}
which is equivalent to
\begin{multline}\label{61original}
Z_{\hbar}(X) =e^{-\frac{7\pi i }{12}}\overline{\tpsi{a_1}{c_1}{0}} \int_{\R^{2}}  {\psi}_{{b_2},{a_2}}(z) {\psi_{{b_3},{c_3}}(z)} \psi_{{b_4},{a_4}}(-x) \psi_{{b_5},{a_5}}(x+z) \\
e^{2\pi i (2xz+\frac{3}{2}z^2+x^2)} \: dxdz.
\end{multline}

Set $\tilde{z}=z-c_{\bq}(1-2c_2)$ and $-\tilde{x}=-x-c_{\bq}(1-2c_4)$. Then $$z-c_{\bq}(1-2a_3)=\tilde{z}+c_{\bq}(1-2c_2)-c_{\bq}(1-2a_3) = \tilde{z},$$ because $a_3\to c_2$ in the limit where $a_1\to 0.$
Furthermore we have 
$$ x+z-c_{\bq}(1-2c_5)=\tilde{x}-c_{\bq}(1-2c_4)+\tilde{z}+c_{\bq}(1-2c_2)-c_{\bq}(1-2c_5)=\tilde{x}+\tilde{z}-c_{\bq}$$
because $$c_4+c_5-c_2 \to 0$$ when $a_1 \to 0.$
We can now write the partition function in the following way
\begin{align*}
Z_{\hbar}(X)= e^{-\frac{7\pi i}{12}} \int_{\R^2} &\psi(\tilde z)\psi( \tilde z)\psi(-\tilde x)\psi(\tilde x + \tilde z -c_{\bq})\overline{\tpsi{a_1}{c_1}{0}}\\
&e^{-4\pi i  (-\tilde z-c_{\bq}(1-2c_2))(\tilde x-c_{\bq}(1-2c_4))+3\pi i (-\tilde z -c_{\bq}(1-2c_2))^2+2\pi i (\tilde x -c_{\bq}(1-2c_4))^2} \\
&e^{-4\pi i c_{\bq}b_2 (\tilde z + c_{\bq}(1/2-c_2))} \nu_{a_2,b_2}\\
&e^{-4\pi i c_{\bq}b_3 (\tilde z + c_{\bq}(1-2c_2)-c_{\bq}(1/2-a_3))} \nu_{b_3,c_3} \\
&e^{-4\pi i c_{\bq}b_4 (-\tilde x-c_{\bq}(1/2-c_4))}\nu_{b_4,a_4} \\
&e^{-4\pi i c_{\bq}b_5 (\tilde x +\tilde z -c_{\bq}(1-2c_4) +c_{\bq}(1-2c_2)-c_{\bq}(1/2-c_5))}\nu_{b_5,a_5} d\tilde x d\tilde z 
\end{align*}
In front of $\tilde z$ in the exponent we have the factor
\begin{align*} 
&4\pi i c_{\bq} (-1+2c_4 +3/2-3c_2-b_2-b_3-b_5) \\ 
&= 4\pi i c_{\bq}(1/2 + 2c_4 - 2c_2 -b_2-a_3-b_3-b_5) \\
&=4\pi i c_{\bq}(1/2 + 2c_4 - c_2 -1/2 + a_2 -1/2+c_3-1/2+a_5+c_5)  \\
&=4\pi i c_{\bq}(-1+1+c_4-c_2+c_5)=0.
\end{align*}
In front of $\tilde x$ in the exponent we also have the factor $0$ since 
\begin{align*}
-b_5+b_4+1-2c_2-1+2c_4&=-\frac{1}{2}+a_5+c_5+c_4+b_4+ c_4 -2c_2 \\&=-\frac{1}{2} +a_5 +\frac{1}{2}-a_4 - c_2 \\& = a_5-a_4-a_3=0.
\end{align*}
This gives us the partition function
\begin{align*}
Z_{\hbar}(X)= e^{i \frac{\phi}{\hbar}} e^{-\frac{7\pi i}{12}}\overline{\tpsi{a_1}{c_1}{0}} \int_{\R^2} &\psi(\tilde z)\psi( \tilde z)\psi(-\tilde x)\psi(\tilde x+\tilde z-c_{\bq})\\
& e^{2\pi i (\frac{3}{2}\tilde z^2 +\tilde x^2 + 2\tilde x \tilde z)} d\tilde x d\tilde z,
\end{align*}
where $\phi$ is a real quadratic polynomial of dihedral angles. 
Finally, we do the shift $\tilde x \mapsto \tilde x - \tilde z + c_{\bq}$ and get the expression
\begin{align*}
Z_{\hbar}(X) &= \zeta_{inv}^2e^{2\pi i c_{\bq}^2}e^{i \frac{\phi}{\hbar}} e^{-\frac{7\pi i}{12}}\overline{\tpsi{a_1}{c_1}{0}} \int_{\R^2} \frac{\Phi_{\bq}(\tilde z)\Phi_{\bq}(\tilde x)}{\Phi_{\bq}(- \tilde z)\Phi_{\bq}(\tilde x-\tilde z -c_{\bq})} e^{\pi i \tilde x^2+4\pi c_{\bq}x} d\tilde x d\tilde z.
\end{align*}
which exactly corresponds to the result for an H-triangulation of the $6_1$ knot in \cite{arXiv:1210.8393}.

\subsection{One vertex H-triangulation of $(S^3,6_1)$ -- New formulation}
We here calculate the partition function for the H-triangulation of the knot $6_1$ using the new formulation of the TQFT from quantum \Teim theory. 

The gluing pattern of faces and edges in diagram \eqref{diagram:61} results in the following states
\begin{align*}
x &:=x_{02}^{1}=x_{03}^{1}=x_{01}^{2}=x_{02}^{2}=x_{01}^{3}, \\
y &:=x_{03}^{2}=x_{13}^{2}=x_{02}^{3}=x_{03}^{3}=x_{13}^{3}=x_{02}^{4}=x_{03}^{4}=x_{03}^{5}, \\
z &:=x_{23}^{2}=x_{12}^{3}=x_{12}^{4}=x_{01}^{5}, \\
v &:=x_{12}^{1}=x_{13}^{1}=x_{23}^{3}=x_{23}^{4}=x_{12}^{5}=x_{13}^{5}, \\
w &:=x_{23}^{1}=x_{12}^{2}=x_{01}^{4}=x_{13}^{4}=x_{02}^{5}=x_{23}^{5}, \\
x' &:= x_{01}^1.
\end{align*}

The Bolzmann weights for the five tetrahedron are given by
\begin{align*}
&\overline{g_{a_1,c_1}(0,x+v-x'-w)}, \quad {g_{a_2,c_2}(x-w,y-z)}, \quad \overline{g_{a_3,c_3}(y-z,2y-x-v)}, \\
&g_{a_4,c_4}(w-z,y-v), \quad g_{a_5,c_5}(w-y,v-z+p).
\end{align*}
\begin{align*}
Z_{\hbar}^{\text{New}}(X) = \int_{[0,1]^6}\sum_{k,l,m,n,p \in \Z}3 
&\overline{\tpsi{a_1}{c_1}{k}} \\
&{\tpsi{a_2}{c_2}{x-w+l}} e^{\pi i (y-z)(x-w+2l)} \\ 
& \overline{\tpsi{a_3}{c_3}{y-z+m}}e^{-\pi i (2y-x-v)(y-z+2m)} \\
&\tpsi{a_4}{c_4}{w-z+n}e^{\pi i (y-v)(w-z+2n)} \\
&\tpsi{a_5}{c_5}{w-y+p} e^{\pi i (v-z)(w-y+p)}  dxdydzdvdwdx' 
\end{align*}
Integration over $x'$ gives $\delta(k)$, which removes one of the sums
\begin{align*}
Z_{\hbar}^{\text{New}}(X) =\overline{\tpsi{a_1}{c_1}{0}} \int_{[0,1]^5}\sum_{l,m,n,p \in \Z} 
&{\tpsi{a_2}{c_2}{x-w+l}} e^{\pi i (y-z)(x-w+2l)} \\ 
& \overline{\tpsi{a_3}{c_3}{y-z+m}}e^{-\pi i (2y-x-v)(y-z+2m)} \\
&\tpsi{a_4}{c_4}{w-z+n}e^{\pi i (y-v)(w-z+2n)} \\
&\tpsi{a_5}{c_5}{w-y+p} e^{\pi i (v-z)(w-y+p)}  dxdydzdvdw.
\end{align*}
We do a shift $x\to x+w$
\begin{align*}
Z_{\hbar}^{\text{New}}(X) =\overline{\tpsi{a_1}{c_1}{0}} \int_{[0,1]^5}\sum_{l,m,n,p \in \Z} 
&{\tpsi{a_2}{c_2}{x+l}} e^{\pi i (y-z)(x+2l)} \\ 
& \overline{\tpsi{a_3}{c_3}{y-z+m}}e^{-\pi i (2y-x-w-v)(y-z+2m)} \\
&\tpsi{a_4}{c_4}{w-z+n}e^{\pi i (y-v)(w-z+2n)} \\
&\tpsi{a_5}{c_5}{w-y+p} e^{\pi i (v-z)(w-y+p)}  dxdydzdvdw.
\end{align*}
Note that 
\begin{align*}
&\sum_{l \in \Z} e^{2\pi i l (y-z)} \int_{[0,1]}{\tpsi{a_2}{c_2}{x+l}}e^{-2\pi i x (z-y-m)}dx \\
&=e^{-\frac{\pi i}{12}}\sum_{l \in \Z}  \int_{l}^{l+1}{\psi_{{c_2},{b_2}}({x})}e^{-2\pi i x (z-y-m)}dx  \\
&=e^{-\frac{\pi i}{12}}  \int_{\R}{\psi_{{c_2},{b_2}}({x})}e^{-2\pi i x (z-y-m)}dx \\
&=e^{-\frac{\pi i}{12}}\tilde\psi_{c_2,b_2}(z-y-m) = e^{-\frac{\pi i}{6}}\psi_{b_2,a_2}(z-y-m)e^{\pi i (z-y-m)^2}.
\end{align*}
\begin{align*}
Z_{\hbar}^{\text{New}}(X) =e^{-\frac{\pi i 5}{12}}\overline{\tpsi{a_1}{c_1}{0}} &\int_{[0,1]^3} \sum_{m,n,p\in \Z}  \psi_{{b_2},{a_2}}({z-y-m}) \psi_{b_3,c_3}(z-y-m)\\
& \psi_{{c_4},{b_4}}({w-z+n})\psi_{{c_5},{b_5}}({w-y+p}) \\
&e^{2\pi i (z-y-m)^2}e^{-\pi i (2y-w-v)(y-z+2m)} \\ &e^{(y-v)(w-z+2n)}e^{\pi i (v-z)(w-y+2p)}
\:dzdw. 
\end{align*}
Integration over $v$ removes yet another sum. I.e. $n=p+m$.
We shift $z \mapsto z+y$ and $w\mapsto w+y$ and see that the function is independent of $y$ which yields the expression
\begin{align*}
Z_{\hbar}^{\text{New}}(X) =e^{-\frac{\pi i 5}{12}}\overline{\tpsi{a_1}{c_1}{0}} &\int_{[0,1]^3} \sum_{m,p\in \Z}  \psi_{{b_2},{a_2}}({z-m}) \psi_{b_3,c_3}(z-m)\\
& \psi_{{c_4},{b_4}}({w-z+p+m})\psi_{{c_5},{b_5}}({w+p}) \\
&e^{2\pi i (z-m)^2}e^{\pi i w(-z+2m)} e^{-\pi i z(w+2p)}
\:dzdw. 
\end{align*}
\begin{align*}
Z_{\hbar}^{\text{New}}(X) =e^{-\frac{\pi i 5}{12}}\overline{\tpsi{a_1}{c_1}{0}} &\int_{\R^2}  \psi_{{b_2},{a_2}}({z}) \psi_{b_3,c_3}(z)\\
& \psi_{{c_4},{b_4}}({w-z})\psi_{{c_5},{b_5}}({w}) \\
&e^{2\pi i z^2}e^{\pi i (w-m)(-z+m)} e^{-\pi i (z+m)(w+p)}
\:dzdw. 
\end{align*}
Now let $f(z):=\psi_{{b_2},{a_2}}({z}) \psi_{b_3,c_3}(z).$ We then calculate
\begin{align*}
Z_{\hbar}^{\text{New}}(X) &=e^{-\frac{\pi i 5}{12}}\overline{\tpsi{a_1}{c_1}{0}} \int_{\R^2}  f(z)
\psi_{{c_4},{b_4}}({w-z})\psi_{{c_5},{b_5}}({w}) 
e^{2\pi i z^2}e^{-2\pi i wz}\:dzdw \\
&=e^{-\frac{\pi i 5}{12}}\overline{\tpsi{a_1}{c_1}{0}} \int_{\R^2}  f(z) \psi_{b_3,c_3}(z)
\psi_{{c_4},{b_4}}({w})\psi_{{c_5},{b_5}}({w+z}) 
e^{-2\pi i wz} \:dzdw\\
&=e^{-\frac{\pi i 5}{12}}\overline{\tpsi{a_1}{c_1}{0}} \int_{\R^2}  f(z)  \psi_{c_4,b_4}(w)\tilde{\psi}_{{c_5},{b_5}}(x) e^{2\pi i ((x-z)(w+z)+z^2)}\: dxdzdw \\
&=e^{-\frac{\pi i 5}{12}}\overline{\tpsi{a_1}{c_1}{0}} \int_{\R^2}  f(z)  \psi_{c_4,b_4}(w)\tilde{\psi}_{{c_5},{b_5}}(x+z) e^{2\pi i (x(w+z)+z^2)}\: dxdzdw \\
&=e^{-\frac{\pi i 6}{12}}\overline{\tpsi{a_1}{c_1}{0}} \int_{\R^2}  f(z) \psi_{c_4,b_4}(w){\psi}_{{b_5},{a_5}}(x+z) e^{2\pi i (wx+2xz+\frac{3}{2}z^2+\frac{x^2}{2})}\: dxdzdw \\
&=e^{-\frac{\pi i 6}{12}}\overline{\tpsi{a_1}{c_1}{0}} \int_{\R^2}  f(z) \tilde{\psi}_{c_4,b_4}(-x){\psi}_{{b_5},{a_5}}(x+z) e^{2\pi i (2xz+\frac{3}{2}z^2+\frac{x^2}{2})}\: dzdw \\
&=e^{-\frac{\pi i 7}{12}}\overline{\tpsi{a_1}{c_1}{0}} \int_{\R^2}  f(z){\psi}_{b_4,a_4}(-x){\psi}_{{b_5},{a_5}}(x+z) e^{2\pi i (2xz+\frac{3}{2}z^2+x^2)}\: dzdw \\
\end{align*}
This is the exact same expression as in \eqref{61original} and the two formulations coinside. 

\subsection{Volume of $(S^3,6_1)$}

\begin{thm}
	The hyperbolic volume of the complement of $6_1$ in $S^3$ is recovered as the following limit
	\begin{equation}
	\lim_{\hbar \to 0} 2\pi \hbar \log{\abs{J_{s^3,6_1}(\hbar,0)}}=-\Vol(S^3\backslash 6_1).
	\end{equation}
\end{thm}
\begin{proof}
	We consider the expression
	\begin{align}
	J_{S^3,6_1}(\hbar,0)=\int_{\R^2}\frac{\Phi_{\bq}(x)\Phi_{\bq}(z)}{\Phi_{\bq}(-x)\Phi_{\bq}(z-x-c_{\bq})}e^{\pi i z^2-4\pi i c_{\bq}z}dxdz.
	\end{align}
	Using the quasi-classical asymptotic behaviour of Faddeev's quantum dilogarithm shown in Corollary \ref{eq:asym} we can approximate in the following manner.
	For $\bq$ close to zero the integral in \eqref{inv} is approximated by the double contour integral 
	\begin{align*}
	J_{S^3,6_1}(\hbar,0) &= \frac{1}{(2\pi \bq)^2}\int_{\R^2}\frac{\Phi_{\bq}\left(\frac{x}{2\pi \bq}\right)\Phi_{\bq}\left(\frac{z}{2\pi \bq}\right)}{\Phi_{\bq}\left(\frac{-x}{2\pi \bq}\right)\Phi_{\bq}\left(\frac{z-x-c_{\bq}}{2\pi \bq}\right)}e^{-\frac{z^2}{4\pi i \bq^2}+\frac{ y}{ \bq^2}}dxdz \\
	&\sim \frac{1}{(2\pi \bq)^2}\int_{\R^2}e^{\frac{1}{2\pi i \bq^2}(2\Li_2(-e^{x})+\Li_2(-e^z)-\Li_2(e^{z-x})-\frac{1}{2}z^2+2\pi i z +\frac{1}{2}x^2)} dxdz \\
	&=\frac{1}{(2\pi \bq)^2}\int_{\R^2}e^{\frac{1}{2\pi i \bq^2}V(x,z)} dxdz,
	\end{align*}
	where the potential $V$ is given by
	\begin{equation}
	V(x,z)=2\Li_2(-e^{x})+\Li_2(-e^z)-\Li_2(e^{z-x})-\frac{1}{2}z^2+2\pi i z +\frac{1}{2}x^2.
	\end{equation}
	It is easily seen that we can treat $\bq^2$ as $\hbar$. Therefore, we look for stationary points of the potential $V$
	\begin{equation}
	\frac{\partial V(x,z)}{\partial x}=-2\log(1+e^x)-\log(1-e^{z-x})+x=\log\frac{e^x}{(1+e^x)^2(1-e^{z-x})}.
	\end{equation}
	\begin{equation}
	\frac{\partial V(x,z)}{\partial z}=-\log(1+e^z)+\log(1-e^{z-x})-z+2\pi i= \log\frac{1-e^{z-x}}{(1+e^z)e^z}.
	\end{equation}
	Stationary points are given by solutions to the equations 
	\begin{align}
	e^x &= (1+e^x)^2(1-e^{z-x}),\label{equ} \\
	(1+e^z)e^z &= 1- e^{z-x}\label{eq}.
	\end{align}
	From \eqref{eq} we see that 
	\begin{align*}
	e^x=\frac{1}{e^{-z}-1-e^z}.
	\end{align*}
	Inserting in \eqref{equ} we get the equation
	\begin{align*}
	\frac{1}{e^{-z}-1-e^z} &= \left(\frac{e^{-z}-e^z}{e^{-z}-1-e^z}\right)^2(1+e^z)e^z \iff 1-t-t^2=(1-t^2)^2(1+t),
	\end{align*}
	where we set $e^z=t.$
	
	Numerical solutions for the last equation are given by
	\begin{align*}
	t_1&=-1,39923-0,32564 i, \quad 
	&&t_3=0,899232-0,400532 i, \\
	t_2&=-1,39923+0,32564 i, \quad 
	&&t_4=0,899232+0,400532 i.
	\end{align*}
	The maximal contribution to the integral comes from the point $t_2$. The saddle point method lets us obtain the following limit 
	\begin{equation*}
	\lim_{\hbar^2 \to 0}2\pi \hbar\abs{J_{S^3,6_1}(\hbar,0)}=-3.1632...\cdot I = -\Vol(S^3\backslash 6_1)
	\end{equation*}
\end{proof}


\section{The Teichmüller TQFT representation of the mapping class group $\Gamma_{1,1}$}\label{appmcg}
We will here give a representation for the mapping class group of the once punctured torus by the use of the new formulation of the Teichmüller  TQFT.

The framed mapping class group $\Gamma_{1,1}$ is generated by the standard elements $S$ and $T$. See e.g. Section 6 in \cite{AU3} for a description of these elements (they of course maps to the standard $S$ and $T$ matrix once mapped to the mapping class group of the torus). Hence we just need to understand how these to elements are represented by the Teichm\"{u}ller TQFT. To this end, we build a cobordism $(M,\mathbb T^2,{\mathbb T^2 }')$ from one triangulation of $\mathbb T^2$ to the image of this triangulation under the action of $S$ and likewise for the action of $T$. We triangulate the torus $\mathbb T^2 = S^1\times S^1$ according to Figure \ref{Torus}.
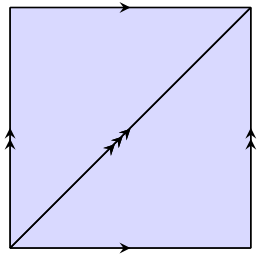
\begin{figure}[h]
\begin{center}
	\begin{tikzpicture}[scale=.8,baseline]
	\draw[fill=blue!15] (-2,2)--(2,2)--(2,-2)--(-2,-2)--cycle;
	\tkzDefPoint(-2,-2){A}
	\tkzDefPoint(2,-2){B}
	\tkzDefPoint(2,2){C}
	\tkzDefPoint(-2,2){D}
	\tkzDrawSegment[arrowMe=>>s](A,D)
	\tkzDrawSegment[arrowMe=>>s](B,C)
	\tkzDrawSegment[arrowMe=stealth](D,C)
	\tkzDrawSegment[arrowMe=stealth](A,B)
	\tkzDrawSegment[arrowMe=>>>s](A,C)
	\tkzDrawSegment[arrowMe=>>>s](A,C)
	\end{tikzpicture}
	\caption{Triangulated torus.}
	\label{Torus}
	\end{center}
\end{figure}
In this triangulation opposite arrows are identified and this gives us a triangulation with two triangles and three edges. We build the cobordism for the action of $S$ according to Figure \ref{Cobordism $X_S$} and the cobordism for $T$ according to Figure \ref{Cobordism $X_T$}. We see that on each boundary component we have three edges. 
The cobordisms that we build are given shaped triangulations. 
We can choose the dihedral angles such that they are all positive. And we are able to compose these cobordisms. 

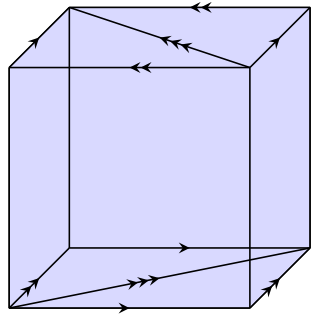
\begin{figure}[h]
\begin{center}
	\begin{tikzpicture}[scale=.8,baseline]
	\tkzDefPoint(-2,-2){A}
	\tkzDefPoint(2,-2){B}
	\tkzDefPoint(2,2){C}
	\tkzDefPoint(-2,2){D}
	\tkzDefPoint(-1,-1){E}
	\tkzDefPoint(3,-1){F}
	\tkzDefPoint(3,3){G}
	\tkzDefPoint(-1,3){H}
	\draw[fill=blue!15] (-2,-2)--(2,-2)--(3,-1)--(3,3)--(-1,3)--(-2,2)--cycle;
	\tkzDrawSegments[arrowMe=stealth](A,B E,F) 
	\tkzDrawSegments[arrowMe=>>s](A,E B,F)
	\tkzDrawSegments[arrowMe=stealth](D,H C,G)
	\tkzDrawSegments[arrowMe=>>s](C,D G,H)
	\tkzDrawSegment (B,C)
	\tkzDrawSegment (A,D)
	\tkzDrawSegment (F,G)
	\tkzDrawSegment (E,H)
	\tkzDrawSegment[arrowMe=>>>s](A,F)
	\tkzDrawSegment[arrowMe=>>>s](C,H)
	\end{tikzpicture}
	\caption{The cobordism for the operator $S$.}
	\label{Cobordism $X_S$}
	\end{center}
\end{figure}

\begin{figure}[h]
\begin{center}
	\begin{tikzpicture}[scale=.8,baseline]
	\tkzDefPoint(-2,-2){A}
	\tkzDefPoint(2,-2){B}
	\tkzDefPoint(2,2){C}
	\tkzDefPoint(-2,2){D}
	\tkzDefPoint(-1,-1){E}
	\tkzDefPoint(3,-1){F}
	\tkzDefPoint(4,3){G}
	\tkzDefPoint(0,3){H}
	\draw[fill=blue!15] (-2,-2)--(2,-2)--(3,-1)--(4,3)--(0,3)--(-2,2)--cycle;
	\tkzDrawSegments[arrowMe=stealth](A,B E,F) 
	\tkzDrawSegments[arrowMe=>>s](A,E B,F)
	\tkzDrawSegments[arrowMe=>>s](D,H C,G)
	\tkzDrawSegments[arrowMe=stealth](D,C H,G)
	\tkzDrawSegment (B,C)
	\tkzDrawSegment (A,D)
	\tkzDrawSegment (F,G)
	\tkzDrawSegment (E,H)
	\tkzDrawSegment[arrowMe=>>>s](A,F)
	\tkzDrawSegment[arrowMe=>>>s](D,G)
	\end{tikzpicture}
	\caption{The cobordism for the operator $T$.}
	\label{Cobordism $X_T$}
	\end{center}
\end{figure}
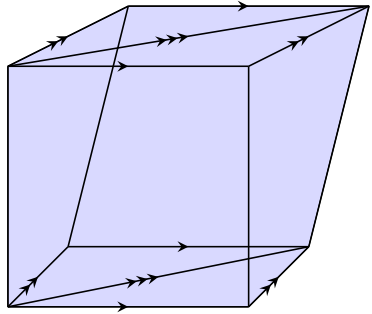

For each edge in these triangulations we assign a state variable. We abuse notation and label an edge and a state variable by the same letter. We assign a multiplier to each edge. As we will see below in Lemma \ref{proofoflemma} and Lemma \ref{proofoflemma2} it turns out, that all internal edges have trivial multiplier. Further we emphasise that there is a direction on each of the two boundary tori where the multiplier is trivial. 

The Teichm\"{u}ller TQFT gives an operator between the vector spaces associated to each of the boundary components. We will see that we get representations $$ \rho : \Gamma_{1,1} \to \mathcal B(C^{\infty}(\mathbb T^3,\mathcal L')), $$ of the mapping class group $\Gamma_{1,1}$ into bounded operators on the smooth sections $C^{\infty}(\mathbb T^3,\mathcal L')$. However, we will show below that we actually get representations into $\mathcal B(\mathcal S( \R))$, bounded operators on the Schwartz space $\mathcal S(\R) $.

\begin{thm} The Teichmüller TQFT provides us with representations (dependent on h)
$$\tilde{\rho} : \Gamma_{1,1} \to \mathcal B(\mathcal S(\R))$$ 
of the mapping class group $\Gamma_{1,1}$ into bounded operators on the Schwarz space $\mathcal S(\R)$.
	In particular we get operators $\tilde{\rho}(S), \tilde{\rho}(T) : \mathcal S(\R) \to \mathcal S(\R)$ according to the diagram \eqref{diag}.
	\begin{align}\label{diag}
	\xymatrix{
		\mathcal S (\R) \ar[r]^{\tilde{\rho}(S), \tilde{\rho}(T)} \ar[d]^{W} & \mathcal S (\R) \ar[d]^{W} \\ 
		C^{\infty}(\mathbb T^2, \mathcal L) \ar[d]^{\pi^*} \ar[r]^{} & C^{\infty}(\mathbb T^2, \mathcal L) \ar[d]^{\pi^*}  \\
		\pi^* \left(C^{\infty}(\mathbb T^2, \mathcal L)\right)  \ar[r]^{} & \pi^* \left(C^{\infty}(\mathbb T^2, \mathcal L)\right) \\
		C^{\infty}(\mathbb T^3, \mathcal L')\ar@{}[u]|-{\displaystyle{\cap}} \ar[r]^{\rho(S)}_{\rho(T)}&\ar@{}[u]|-{\displaystyle\cap} C^{\infty}(\mathbb T^3, \mathcal L')
	}
	\end{align}
	where $\mathcal L' = \pi^*\mathcal L.$ 
\end{thm}
\begin{proof}
	We know that the Weil--Gel'fand--Zak transformation gives an isomorphism from the Schwarz space to smooth sections of the complex line bundle $\mathcal L$ over the 2-torus. If a section of $C^{\infty}(\mathbb T^2,\mathcal L)$ is pulled back to $\pi^* \left(C^{\infty}(\mathbb T^2, \mathcal L)\right)$ we show in Lemma \ref{Sop} and Lemma \ref{Top} that the operators $\rho(S),\rho(T)$ acting on $C^{\infty}(\mathbb T^3,\mathcal L')$ take this pull back of a section in $\pi^* \left(C^{\infty}(\mathbb T^2, \mathcal L)\right)$. 
	In Lemma \ref{proofoflemma} and \ref{proofoflemma2} we prove that the multipliers on internal edges are trivial. Further we show that the multipliers on the two boundary tori are trivial in the direction $(1,1,1)$. We can therefore integrate over the fibre in this direction. We then use the inverse WGZ transformation. 
	In other words we have shown that the operators $\rho(S),\rho(T)$ induce operators $\tilde \rho(S),\tilde \rho(T):\mathcal S(\R)\to S(\R)$ given by 
	\begin{align*}
	&\tilde \rho(S) = W^{-1} \circ \int_{F_{z'}} \circ \: \rho(S) \circ \pi^* \circ W, \\
	&\tilde \rho(T) = W^{-1} \circ \int_{F_{z'}} \circ \:  \rho(S) \circ \pi^* \circ W.
	\end{align*}
\end{proof}

\begin{rem}
	Above we obtained a representation for the mapping class group $\Gamma_{1,1}$. We do not in a similar manner get a representation for the mapping class group $\Gamma_{1,0}$. The reason is that not all edges in the cobordisms can be balanced without turning to negative angles. 
\end{rem}

\section{Line bundles over the two boundary tori}
Let us here describe how the line bundles we pull back looks like.

Let $\pi : \R^3 \to \R^2$ be defined by $\pi(x_1,x_2,x_3)=(ax_1+bx_2+cx_3,\alpha x_1+\beta x_2+\gamma x_3)$. Recall that we have the relation on multipliers 
\begin{align}\label{pullbackmultiplier}
e^{\pi^*}_{\lambda}(x,y,z)= e_{\pi(\lambda)}(\pi(x,y,z)).
\end{align}

Note that the map $\pi$ sends $\lambda_{x_1}=(1,0,0),\; \lambda_{x_2}=(0,1,0),\; \lambda_{x_3}=(0,0,1)$ to the following elements of $\R^2$
$$ \pi(\lambda_{x_1})=(a,\alpha), \quad \pi(\lambda_{x_2})=(b,\beta), \quad \pi(\lambda_{x_3})=(c,\gamma).$$
The equation \eqref{pullbackmultiplier} gives the following relations:

\begin{align*}
\textbf{In the $\lambda_{x_1}$-direction}\\
e^{2\pi i (x_3-x_2)}=&e_{(1,0,0)(x_1,x_2,x_3)}=e_{(a,\alpha)}(ax_1+bx_2+cx_3,\alpha x_1+\beta x_2+ \gamma x_3) \\
=& e_{(a,0)}(ax_1+bx_2+cx_3,\alpha x_1+\beta x_2+ \gamma x_3) \\
&e_{(0,\alpha)}(ax_1+bx_2+cx_3,\alpha (x_1+1)+\beta x_2+ \gamma x_3) \\
=&e^{-\pi i a(\alpha x_1+\beta x_2+ \gamma x_3)}e^{\pi i (ax_1+bx_2+cx_3)}\\=&e^{\pi i ((\alpha b-a\beta)x_2+(\alpha c - a\gamma)x_3 )}, \\
\textbf{In the $\lambda_{x_2}$-direction}\\
e^{2\pi i (x_1-x_3)}=&e_{(0,1,0)(x_1,x_2,x_3)}=e^{\pi i ((\beta a- \alpha b)x_1+(\beta c-b\gamma )x_3 )}, \\
\textbf{In the $\lambda_{x_3}$-direction}\\
e^{2\pi i (x_2-x_1)}=&e_{(0,0,1)(x_1,x_2,x_3)}=e^{\pi i ((\gamma a- \alpha c)x_1+(\gamma b -c \beta )x_2 )}.
\end{align*}
In other words we only need to solve the three equations 
\begin{align}
\alpha b - a\beta = -2, \quad \alpha c- a\gamma =2, \quad \beta c - b \gamma =-2.
\end{align}
One particular solution is $a=-2,b=0,c=2, \alpha=0,\beta=-1,\gamma=1$ which gives the map 
$$ \pi(x_1,x_2,x_3)=(-2x_1+2x_3,-x_2+x_3).$$


\section{The operator $\rho(S)$}
The operator $\rho(S)$ can be viewed as the cobordism $X_S$ which is triangulated into 6 tetrahedra $T_1,\dots, T_6$ where $T_1,T_3,T_4,T_6$ have positive orientation and the tetrahedra $T_2,T_5$ have negative orientation. See the gluing pattern in figure \ref{gluings}.

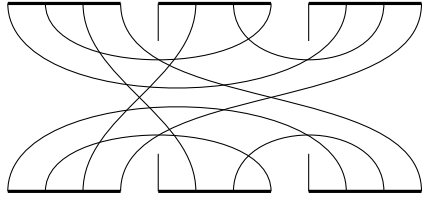
\begin{figure}[H]
\begin{center}
	\begin{tikzpicture}[scale=.5,baseline=5]
	\draw[very thick] (0,0)--(3,0);
	\draw[very thick] (4,0)--(7,0);
	\draw[very thick] (8,0)--(11,0);
	\draw[very thick] (0,5)--(3,5);
	\draw[very thick] (4,5)--(7,5);
	\draw[very thick] (8,5)--(11,5);
	\draw(0,0)..controls (0,3) and (9,3)..(9,0);
	\draw(1,0)..controls (1,2) and (7,2)..(7,0);
	\draw(2,0)..controls (2,2) and (5,3)..(5,5);
	\draw(3,0)..controls (3,3) and (11,2)..(11,5);
	\draw(5,0)..controls (5,2) and (2,3)..(2,5);
	\draw(6,0)..controls (6,2) and (10,2)..(10,0);
	\draw(11,0)..controls (11,3) and (3,2)..(3,5);
	\draw(0,5)..controls (0,2) and (9,2)..(9,5); 
	\draw(1,5)..controls (1,3) and (7,3)..(7,5);
	\draw(6,5)..controls (6,3) and (10,3)..(10,5);
	\draw (4,0)--(4,1);
	\draw (8,0)--(8,1);
	\draw (4,5)--(4,4);
	\draw (8,5)--(8,4);
	\end{tikzpicture}
	\caption{Gluing pattern for the operator $X_S$}
	\label{gluings}
	\end{center}
\end{figure}

In the triangulation we have ten edges $x_1,x_2\dots,x_7,x_1',x_2',x_3'$. To each of the edges on the boundary we associate the a weight function:
\begin{align*}
&\omega_{X_S}(x_1)=2\pi (a_1+a_5+c_3), \quad &&\omega_{X_S}(x_2)=2\pi (a_4+c_5+a_6), \quad &&&\omega_{X_S}(x_3)=2\pi (b_5+b_6), \\
&\omega_{X_S}(x_1')=2\pi (a_1+c_2+a_3), \quad &&\omega_{X_S}(x_2')=2\pi (a_2+c_3+a_4), \quad &&&\omega_{X_S}(x_3')=2\pi (b_2+b_3).
\end{align*}
and to the edges $x_4,x_5,x_6,x_7$ we associate the weight functions:
\begin{align*}
&\omega_{X_S}(x_4)=2\pi (a_1+c_2+b_4+c_5+c_6), \quad &&\omega_{X_S}(x_5)=2\pi (c_1+b_3+b_4+a_5+a_6), \\
&\omega_{X_S}(x_6)=2\pi (b_1+a_2+a_3+c_4+b_6), \quad &&\omega_{X_S}(x_7)=2\pi (b_1+c_2+c_3+c_4+b_5).
\end{align*}
\subsection{Boltzmann weights}
The Bolzman weights assigned to the tetrahedra are
\begin{align*}
&B\left(T_1,x_{\vert_{\Delta_{1}(T_1)}}\right)=g_{a_1,c_1}(x_7+x_6-x_4-x_5,x_7+x_6-x_1'-x_1), \\
&B\left(T_2,x_{\vert_{\Delta_{1}(T_2)}}\right)=\overline{g_{a_2,c_2}(x_3'+x_4-x_1'-x_7,x_3'+x_4-x_2'-x_6)}, \\
&B\left(T_3,x_{\vert_{\Delta_{1}(T_3)}}\right)=g_{a_3,c_3}(x_3'+x_5-x_2'-x_7,x_3'+x_5-x_1'-x_6) \\
&B\left(T_4,x_{\vert_{\Delta_{1}(T_4)}}\right)=g_{a_4,c_4}(x_5+x_4-x_7-x_6,x_5+x_4-x_2'-x_2) \\
&B\left(T_5,x_{\vert_{\Delta_{1}(T_5)}}\right)=\overline{g_{a_5,c_5}(x_7+x_3-x_4-x_2,x_5+x_4-x_5-x_1) }\\
&B\left(T_6,x_{\vert_{\Delta_{1}(T_6)}}\right)=g_{a_6,c_6}(x_6+x_3-x_4-x_1,x_6+x_3-x_5-x_2).
\end{align*}



\begin{lemma}\label{proofoflemma}
	The multipliers corresponding to the edges are calculated to be 1 for the internal edges $x_4,x_5,x_6,x_7$. And the multipliers for the remaining 6 edges are calculated to be
	\begin{align*}
	e_{\lambda_{x_1}}(\textbf{x})&=e^{2\pi i (x_3-x_2)}, \quad e_{\lambda_{x_2}}(\textbf{x})=e^{2\pi i (x_1-x_3)},\quad e_{\lambda_{x_3}}(\textbf{x})=e^{2\pi i (x_2-x_1)}, \\
	e_{\lambda_{x_1'}}(\textbf{x})&=e^{2\pi i (x_2'-x_3')}, \quad e_{\lambda_{x_2'}}(\textbf{x})=e^{2\pi i (x_3'-x_1')},\quad e_{\lambda_{x_3'}}(\textbf{x})=e^{2\pi i (x_1'-x_2')},
	\end{align*}
	where \textbf{x} denotes the tuple $\textbf{x}=(x_1,x_2,x_3,x_1',x_2',x_3').$
\end{lemma}
\begin{proof}
	Let us here just calculate the multiplier for the direction $x_4$. The rest follows by analogous calculations.
	The edge $x_4$ is an edge in the tetrahedra $T_1,T_2,T_4,T_5,T_6$ each contributing to the multiplier. 
	The contribution from $T_1$ corresponds to the multiplier
	\begin{align*}
	e_{\lambda_{x_4}}(x_1,x_2,\dots,x_7,x_1',x_2',x_3')&=e_{-(1,0)}(x_5+x_4-x_7-x_6,x_5+x_4-x_2'-x_2)\\&=e^{\pi i (x_7+x_6-x_1'-x_1)}.
	\end{align*}
	The contribution from $T_2$ is
	\begin{align*}
	e_{\lambda_{x_4}}(x_1,x_2,\dots,x_7,x_1',x_2',x_3')&=\overline{e_{(1,1)}(x_3'+x_4-x_1'-x_7,x_3'+x_4-x_2'-x_6)}\\&=-e^{-\pi i (x_2'+x_6-x_1'-x_7)}.
	\end{align*}
	The contribution from $T_4$  is
	\begin{align*}
	e_{\lambda_{x_4}}(x_1,x_2,\dots,x_7,x_1',x_2',x_3')&={e_{(1,1)}(x_5+x_4-x_7-x_6,x_5+x_4-x_2'-x_2)}\\&=-e^{\pi i (x_2'+x_2-x_6-x_7)}.
	\end{align*}
	The contribution from $T_5$ is
	\begin{align*}
	e_{\lambda_{x_4}}(x_1,x_2,\dots,x_7,x_1',x_2',x_3')=-e^{-\pi i (x_7+x_3-x_5-x_1)}.
	\end{align*}
	The contribution from $T_6$ is
	\begin{align*}
	e_{\lambda_{x_4}}(x_1,x_2,\dots,x_7,x_1',x_2',x_3')=-e^{\pi i (x_6+x_3-x_5-x_2)}.
	\end{align*}
	Multiplying these contributions gives $e^0=1.$
\end{proof}
We remark that the multiplier on each boundary component in direction $(1,1,1)$ is trivial. 

We are interested in how the operator $\rho(S)$ acts. We express the operator $\rho(S)$ in terms of the integral kernel $K_S$. The operator $\rho(S)$ acts on sections in the following manner
\begin{equation}
\rho(S)(s)(x'_1,x'_2,x'_3) = \int_{[0,1]^3}K_S(x'_1,x_2',x_3',x_1,x_2,x_3)s(x_1,x_2,x_3)  \:dx_1dx_2dx_3.
\end{equation} 

We want to show that the operator $S$ takes the pull back of a section to the pull back of a section. Using integration by parts it is enough to check that the sum of partial derivatives disappear.
\begin{lemma}\label{Sop} The sum of the partial derivatives of $K_S$ disappears. I.e.
	\begin{align*}
	\frac{\partial K_S}{\partial x_1'}+\frac{\partial K_S}{\partial x_2'}+\frac{\partial K_S}{\partial x_3'}+\frac{\partial K_S}{\partial x_1}+\frac{\partial K_S}{\partial x_2}+\frac{\partial K_S}{\partial x_3}=0.
	\end{align*}
\end{lemma}

\begin{proof}
	Let 
	\begin{align*}
	I_3^{n,m,k,j}(x_1,x_2,x_3,x_2',x_3'):=\int_{[0,1]^2}&\tpsi{a_1}{c_1}{x_7+k} \tpsi{a_4}{c_4}{-x_7+n} \\
	& e^{2\pi i x_7(x_2-x_3+x_7+x_5+2n-m+k-j)} \\&
	e^{2\pi i (x_3'-x_2'-x_3+x_1+k-j)} dx_5dx_7
	\end{align*}
	The partial derivatives of $I_3$ with respect to $x_1,x_2,x_3,x_2',x_3'$ are easily calculated to be
	\begin{align*}
	\frac{\partial}{\partial x_1} I_{3}^{n,m,k,j}(x_1,x_2,x_3,x_2',x_3')&=2\pi i x_5 I_3(x_1,x_2,x_3,x_2',x_3')=:I_3'(x_1,x_2,x_3,x_2',x_3'),\\
	\frac{\partial}{\partial x_2} I_{3}^{n,m,k,j}(x_1,x_2,x_3,x_2',x_3')&=2\pi i x_7 I_3(x_1,x_2,x_3,x_2',x_3')=:I_3''(x_1,x_2,x_3,x_2',x_3'),\\
	\frac{\partial}{\partial x_3} I_{3}^{n,m,k,j}(x_1,x_2,x_3,x_2',x_3')&=-I_3'(x_1,x_2,x_3,x_2',x_3')-I_3''(x_1,x_2,x_3,x_2',x_3'),\\
	\frac{\partial}{\partial x_2'}I_3^{n,m,k,j}(x_1,x_2,x_3,x_2',x_3') &=-I_3'(x_1,x_2,x_3,x_2',x_3'),\\
	\frac{\partial}{\partial x_3'}I_3^{n,m,k,j}(x_1,x_2,x_3,x_2',x_3') &=I_3'(x_1,x_2,x_3,x_2',x_3').
	\end{align*} 
	
	The partial derivatives of $I_2$ with respect to the variables $x_2,x_3,x_1',x_3'$ are 
	\begin{align*}
	\frac{\partial}{\partial x_2}I_2^{k,l,n,p}(x_2,x_3,x_1',x_3') = &\frac{e^{2\pi i (x_1'-x_3'-x_3+x_2+2(k,l,n,p))}(x_1'-x_3'-x_3+x_2+2(k,l,n,p))}{(x_1'-x_3'-x_3+x_2+2(k,l,n,p))^2} \\
	&- \frac{( e^{2\pi i (x_1'-x_3'-x_3+x_2+2(k,l,n,p))}-1)}{2\pi i (x_1'-x_3'-x_3+x_2+2(k,l,n,p))^2} \\
	=&: I_2'(x_2,x_3,x_1',x_3'), \\
	\frac{\partial}{\partial x_3}I_2^{k,l,n,p}(x_2,x_3,x_1',x_3') =& -I_2'(x_2,x_3,x_1',x_3'),\\
	\frac{\partial}{\partial x_1'}I_2^{k,l,n,p}(x_2,x_3,x_1',x_3') =& I_2'(x_2,x_3,x_1',x_3'),\\
	\frac{\partial}{\partial x_3'}I_2^{k,l,n,p}(x_2,x_3,x_1',x_3') = &-I_2'(x_2,x_3,x_1',x_3').
	\end{align*}
	The partial derivatives of $I_1$ with respect to the variables $x_2,x_3,x_1',x_3'$ are
	\begin{align*}
	\frac{\partial}{\partial x_2}I_1^{l,p}(x_2,x_3,x_1',x_3') = &-\frac{e^{2\pi i (x_3'-x_1'-x_2+x_3+2(m+q))}(x_3'-x_1'-x_2+x_3+2(m+q))}{(x_3'-x_1'-x_2+x_3+2(m+q))^2} \\
	&+ \frac{( e^{2\pi i (x_3'-x_1'-x_2+x_3+2(m+q))}-1)}{2\pi i (x_3'-x_1'-x_2+x_3+2(m+q))^2} \\
	= &I_1'(x_2,x_3,x_1',x_3'),\\
	\frac{\partial}{\partial x_3}I_1^{l,p}(x_2,x_3,x_1',x_3') = &-I_1'(x_2,x_3,x_1',x_3'),\\
	\frac{\partial}{\partial x_1'}I_1^{l,p}(x_2,x_3,x_1',x_3') = &I_1'(x_2,x_3,x_1',x_3'),\\
	\frac{\partial}{\partial x_3'}I_1^{l,p}(x_2,x_3,x_1',x_3') = &-I_1'(x_2,x_3,x_1',x_3').
	\end{align*}
	The rest of the terms in $K_S$ all depends on pairs of the variables $x_1,x_2,x_3,x_1',x_2',x_3'$ with opposite sign, summing all contributions together therefore shows that the sum of the partial derivatives disappears.
\end{proof}

\section{The operator $\rho(T)$}
The operator $\rho(T)$ is the TQFT operator associated to the cobordism $X_T$ which is triangulated into 6 tetrahedra $T_1,\dots, T_6$ where $T_1,T_4,T_5$ have negative orientation and the tetrahedra $T_2,T_3,T_6$ have positive orientation. See Figure \ref{gluingt}.

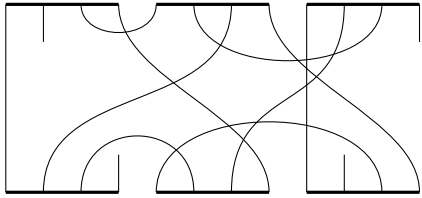
\begin{figure}[H]
\begin{center}
	\begin{tikzpicture}[scale=.5,baseline=5]
	\draw[very thick] (0,0)--(3,0);
	\draw[very thick] (4,0)--(7,0);
	\draw[very thick] (8,0)--(11,0);
	\draw[very thick] (0,5)--(3,5);
	\draw[very thick] (4,5)--(7,5);
	\draw[very thick] (8,5)--(11,5);
	\draw (0,0)--(0,5);
	\draw (8,0)--(8,5);
	\draw (9,0)--(9,1);
	\draw (3,0)--(3,1);
	\draw (1,4)--(1,5);
	\draw (11,4)--(11,5);
	\draw(1,0)..controls (1,3) and (6,2)..(6,5);
	\draw(2,0)..controls (2,2) and (5,2)..(5,0);
	\draw(4,0)..controls (4,2.5) and (10,2.5)..(10,0);
	\draw(6,0)..controls (6,3) and (9,2)..(9,5);
	\draw(7,0)..controls (7,2) and (3,3)..(3,5);
	\draw(11,0)..controls (11,2) and (7,3)..(7,5);
	\draw(2,5)..controls (2,4) and (4,4)..(4,5);
	\draw(5,5)..controls (5,3) and (10,3)..(10,5);
	\end{tikzpicture}
	\caption{Gluing pattern for the operator $X_T$.}
	\label{gluingt}
	\end{center}
\end{figure}

%

In the triangulation we have ten edges $x_1,x_2\dots,x_7,x_1',x_2',x_3'$.
The weight functions corresponding to this triangulation for the edges $x_1,x_2,x_3,x_1',x_2',x_3'$ are 
\begin{align*}
&\omega_{Y_T}(x_1)=2\pi (c_3+a_6), \quad &&\omega_{Y_T}(x_2)=2\pi (b_2+a_3+b_6), \quad &&&\omega_{Y_T}(x_3)=2\pi (b_3+b_5+c_6), \\
&\omega_{Y_T}(x_1')=2\pi (a_1+c_4), \quad &&\omega_{Y_T}(x_2')=2\pi (b_1+a_4+b_5), \quad &&&\omega_{Y_T}(x_3')=2\pi (c_1+b_2+b_4).
\end{align*}
and to the edges $x_4,x_5,x_6,x_7$ we associate the weight functions
\begin{align*}
&\omega_{Y_T}(x_4)=2\pi (a_1+c_2+c_5+a_6), \quad &&\omega_{Y_T}(x_5)=2\pi (b_1+a_2+b_3+b_4+a_5+b_6), \\
&\omega_{Y_T}(x_6)=2\pi (c_1+a_2+a_3+a_4+a_5+c_6), \quad &&\omega_{Y_T}(x_7)=2\pi (b_2+c_3+c_4+c_5).
\end{align*}


The Bolzman weights assigned to the tetrahedra are
\begin{align*}
&B\left(T_1,x_{\vert_{\Delta_{1}(T_1)}}\right)=\overline{g_{a_1,c_1}(x_5+x_2'-x_3'-x_6,x_5+x_2'-x_1'-x_4)}, \\
&B\left(T_2,x_{\vert_{\Delta_{1}(T_2)}}\right)={g_{a_2,c_2}(x_3'+x_2-x_7-x_4,x_3'+x_2-x_5-x_6)}, \\
&B\left(T_3,x_{\vert_{\Delta_{1}(T_3)}}\right)=g_{a_3,c_3}(x_5+x_3-x_7-x_1,x_5+x_3-x_6-x_2) \\
&B\left(T_4,x_{\vert_{\Delta_{1}(T_4)}}\right)=\overline{g_{a_4,c_4}(x_3'+x_5-x_7-x_1',x_3+x_5-x_2'-x_6)} \\
&B\left(T_5,x_{\vert_{\Delta_{1}(T_5)}}\right)=\overline{g_{a_5,c_5}(x_2'+x_3-x_7-x_4,x_2'+x_3-x_6-x_5) }\\
&B\left(T_6,x_{\vert_{\Delta_{1}(T_6)}}\right)=g_{a_6,c_6}(x_2+x_5-x_3-x_6,x_2+x_5-x_4-x_1).
\end{align*}

\begin{lemma}\label{proofoflemma2}
	The multipliers corresponding to the edges are calculated to be 1 for the internal edges $x_4,x_5,x_6,x_7$. And the multipliers for the remaining 6 edges are calculated to be
	\begin{align*}
	e_{\lambda_{x_1}}(\textbf{x})&=e^{2\pi i (x_3-x_2)}, \quad e_{\lambda_{x_2}}(\textbf{x})=e^{2\pi i (x_1-x_3)},\quad e_{\lambda_{x_3}}(\textbf{x})=e^{2\pi i (x_2-x_1)}, \\
	e_{\lambda_{x_1'}}(\textbf{x})&=e^{2\pi i (x_2'-x_3')}, \quad e_{\lambda_{x_2'}}(\textbf{x})=e^{2\pi i (x_3'-x_1')},\quad e_{\lambda_{x_3'}}(\textbf{x})=e^{2\pi i (x_1'-x_2')},
	\end{align*}
	where \textbf{x} denotes the tuple $\textbf{x}=(x_1,x_2,x_3,x_1',x_2',x_3').$
\end{lemma}
\begin{proof}
	The proof is straight forward verification. The computations are analogue to the calculations in the proof of Lemma \ref{proofoflemma}.
\end{proof}

Again, in order to check that the operator $\rho(T)$ takes the pull back of a section to a pull back of a section we show the following Lemma.
\begin{lemma}\label{Top} The sum of the partial derivatives of $K_T$ disappears. I.e.
	\begin{align*}
	\frac{\partial K_T}{\partial x_1'}+\frac{\partial K_T}{\partial x_2'}+\frac{\partial K_T}{\partial x_3'}+\frac{\partial K_T}{\partial x_1}+\frac{\partial K_T}{\partial x_2}+\frac{\partial K_T}{\partial x_3}=0
	\end{align*}
\end{lemma}
\begin{proof}
	In each term of the expression for $K_T$ there is an equal number of variables one half having positive coefficient and the other half having negative coefficient. Therefore the sum of the partial differentials must equal zero.
\end{proof}

\begin{appendices}
	\section{Faddeev's quantum dilogarithm}
	The quantum dilogarithm function $\Li_2(x;q)$, studied by Fadeev and Kashaev \cite{MR1264393} and other authors, is the function of two variables defined by the series 
	\begin{equation}
	\Li_2(x;q)=\sum_{n=1}^{\infty}\frac{x^n}{n(1-q^n)},
	\end{equation}
	where $x,q \in \C$, with $\abs{x},\abs{q}<1$.
	It is connected to the classical Euler dilogarithm $\Li_2$ given by $\Li_2(x)=\sum_{n=1}^{\infty}\frac{x^n}{n^2}$ in the sense that it is a $q$-deformation of the classical one in the following manner 
	\begin{equation}
	\lim_{\epsilon \to 0}\left(\epsilon \Li_{2}(x,e^{-\epsilon})\right)=\Li_2(x), \quad \abs{x}<1.
	\end{equation}
	Indeed using the expansion $\frac{1}{1-e^{-t}}=\frac{1}{t}+\frac{1}{2}+\frac{t}{12}-\frac{t^3}{720}+\dots$ we obtain a complete asymptotic expansion 
	\begin{equation}
	\Li_{2}(x,e^{-\epsilon})=\Li_2(x)\epsilon^{-1} + \frac{1}{2}\log\left(\frac{1}{1-x}\right)+\frac{x}{1-x}\frac{\epsilon}{12}-\frac{x+x^2}{(1-x)^3}\frac{\epsilon^3}{720}+\dots
	\end{equation}
	as $\epsilon \to 0$ with fixed $x\in \C$, $\abs{x}<1.$
	
	The second quantum dilogarithm $(x;q)_{\infty}$ defined for $\abs{q}<1$ and all $x \in \C$ is given as the function
	\begin{equation}
	(x;q)_{\infty}=\prod_{i=0}^\infty (1-xq^i).
	\end{equation}
	This second quantum dilogarithm is related to the first by the formula 
	\begin{equation}
	(x;q)_{\infty}=\exp(-\Li_2(x;q)).
	\end{equation}
	This is easily proven by a direct calculation
	\begin{equation}
	-\log{(x;q)_{\infty}}=\sum_{i=0}^{\infty}\log(1-xq^i)=\sum_{i=0}^{\infty}\sum_{n=1}^{\infty}\frac{1}{n}x^nq^{in}=\sum_{n=1}^\infty \frac{x^n}{n(1-q^n)}=\Li_2(x;q).
	\end{equation}
	
	\begin{prop}
		The function $(x;q)_{\infty}$ and its reciprocal have the Taylor expansions 
		\begin{equation}
		(x;q)_{\infty} = \sum_{n=0}^{\infty}\frac{(-1)^n}{(q)_n}q^{\frac{n(n-1)}{2}}x^n, \quad \frac{1}{(x;q)_{\infty}} \sum_{n=0}^{\infty} \frac{1}{(q)_n}x^n,
		\end{equation}
		around $x=0$, where $$ (q)_n=\frac{(q;q)_{\infty}}{(q^{n+1};q)_{\infty}}=(1-q)(1-q^2)\cdot(1-q^n).$$ 
	\end{prop}
	The proofs of these formulas follows easily from the recursion formula $(x;q)_{\infty}=(1-x)(qx;x)_{\infty}$, which together with the initial value $(0;q)_{\infty}=1$ determines the power series for $(x;q)_{\infty}$ uniquely.

	Yet another famous result for the function $(x;q)_{\infty}$, which can be proven by use of the Taylor expansion and the identity $\sum_{m-n=k}\frac{q^{mn}}{(q)_m(q)_n}=\frac{1}{(q)_{\infty}}$, is the Jacobi triple product formula 
	
	\begin{equation}
	(q;q)_{\infty}(x;q)_{\infty}(qx^{-1};q)_{\infty} = \sum_{k\in \Z} (-1)^k q^{\frac{k(k-1)}{2}}x^k,
	\end{equation} 
	which relates the quantum dilogarithm function to the classical Jacobi theta-function.
	
	The quantum dilogarithm functions introduced are related to yet another quantum dilogarithm function named after Faddeev. 
	\subsubsection*{Faddeev's quantum dilogarithm}
	\begin{defn}
		Faddeev's quantum dilogarithm function is the function in two complex arguments $z$ and $\bq$ defined by the formula
		\begin{equation}\label{qdl}
		\Phi_{\bq}(z)=\exp\pa{\int_{C}\frac{e^{-2izw}dw}{4\sinh(w\bq)\sinh(w/\bq)w}},
		\end{equation}
		where the contour $C$ runs along the real axis, deviating into the upper half plane in the vicinity of the origin.
	\end{defn}
	\begin{prop}
		Faddeev's quantum dilogarithm function $\Phi_{\bq}(z)$  is related to the function $(x;q)_{\infty}$, where $\abs{q}<1,$ in the following sense. When $\Im(\bq^2)>0$, the integral can be calculated explicitly 
		\begin{equation}\label{frac}
		\Phi_{\bq}(z)=\frac{\pa{e^{2\pi(z+c_{\bq}){\bq}};q^2}_\infty}{{\pa{e^{2\pi(z-c_{\bq}){\bq}};\tilde{q}^2}_\infty}}
		\end{equation} 
		where $q \equiv e^{i \pi {\bq}^2} \quad \text{and} \quad \tilde{q} \equiv e^{-\pi i {\bq}^{-2}}.$ 
	\end{prop}
	\begin{proof}
		We consider the integrand of the integral $I(z,\bq)=\frac{1}{4}\int_C \frac{e^{-2izw}}{\sinh(w\bq)\sinh(w/\bq)w}dw$. The integrand has poles at $w=\pi i n \bq$ and $w=\pi i n \bq^{-1}.$
		The residue at $c$ of a fraction i.e. $f(x)=\frac{g(x)}{h(x)}$ can be calculated as $\Res f(c)=\frac{g(c)}{h'(c)}$ when $c$ is a simple pole. Therefore we get by the residue theorem 
		\begin{align*}
		I(z,\bq)&=\frac{\pi i}{2} \sum_{n=1}^{\infty} \frac{e^{2\pi z \bq n}}{\pi i n \bq (-1)^n \sinh(\pi i n \bq^2)}+\frac{e^{2\pi z \bq^{-1} n}}{\pi i n \bq (-1)^n \sinh(\pi i n \bq^{-2})}\\
		&=\sum_{n=1}^{\infty} \frac{e^{\pi i n}e^{2\pi z \bq n}}{n(e^{\pi i n \bq^2}-e^{-\pi i n \bq^2})} + \frac{e^{\pi i n}e^{2\pi z \bq^{-1}n}}{n(e^{\pi i n \bq^{-2}}-e^{-\pi i n \bq^{-2}})} \\
		&=\sum_{n=1}^{\infty} -\frac{\left(e^{2\pi z \bq + \pi i +\pi i \bq^2}\right)^n}{n(1-e^{2\pi i \bq^2 n})} + \frac{\left(e^{2\pi z \bq^{-1} - \pi i -\pi i \bq^{-2}}\right)^n}{n(1-e^{-2\pi i \bq^{-2} n})}\\
		&=\sum_{n=1}^{\infty} -\frac{e^{2\pi(z+c_{\bq})\bq n}}{n(1-e^{2\pi i \bq^2 n})} + \frac{e^{2\pi(z-c_{\bq})\bq^{-1}n}}{n(1-e^{-2\pi i \bq^{-2} n})}\\
		&=\log\pa{e^{2\pi(z+c_{\bq})\bq};q^2}_{\infty}-\log\pa{e^{2\pi(z-c_{\bq})\bq};\tilde{q}^2}_{\infty}.
		\end{align*}
	\end{proof}
	\subsubsection*{Functional equations}
	\begin{prop}\label{funcfad}
		Faddeev's quantum dilogarithm function satisfies the two functional equations
		\begin{equation}\label{feq}
		\frac{1}{\Phi_{\bq}(z+i{\bq}/2)}=\frac{1}{\Phi_{\bq}(z-i{\bq}/2)}\pa{1+e^{2\pi {\bq}z}},
		\end{equation}
		\begin{equation}\label{inv}
		\Phi_{\bq}(z)\Phi_{\bq}(-z)=\zeta_{inv}^{-1}e^{i\pi z^2},
		\end{equation}
		where $\zeta_{inv} = e^{i \pi (1+2c_{\bq}^2)/6}$.
	\end{prop}
	\begin{proof} Let us first prove \eqref{feq}. We have 
		\begin{align*}
		\frac{\Phi_{\bq}(z-i{\bq}/2)}{\Phi_{\bq}(z+i{\bq}/2)}&=\exp{\int_{C}\frac{e^{-2i(z-i{\bq}/2)w}-e^{-2i(z+i{\bq}/2)w}}{4\sinh(w{\bq})\sinh(w/{\bq})w}}\;dw \\
		&= \exp{\int_{C}\frac{e^{-2izw}\pa{e^{-{\bq}w}-e^{{\bq}w}}}{4\sinh(w{\bq})\sinh(w/{\bq})w}}\;dw \\
		&= \exp{\pa{-\frac{1}{2}\int_{C}\frac{e^{-2izw}}{\sinh(w/{\bq})w}}\;dw}.
		\end{align*}
		Let $a>0$. Let $\varepsilon=1$ if $\Im(-2iz)\geq 0$ and $\varepsilon=-1$ otherwise.
		Put $\delta_a^-=[-a,i\varepsilon a]$ and $\delta_a^-=[i\varepsilon a,a]$. The integrals $\int_{\delta_{a^{\pm}}}\frac{e^{-2izw}}{2\sinh(w/{\bq})w}\;dw$ converge to zero as $a\to \infty$. Therefore 
		\begin{equation*}
		\int_{C}\frac{e^{-2izw}}{\sinh(w/{\bq})w}\;dw = \epsilon 2\pi i \pa{c_{\epsilon} + \sum_{n=1}^{\infty}\Res_{w=\epsilon i \pi {\bq} n} \left\{\frac{e^{-2izw}}{\sinh(w/{\bq})w}\right\}},
		\end{equation*}
		where $c_1=0$ and $c_{-1}=\Res_{w=0}\tpa{\frac{e^{2 i z w}}{\sinh(w/{\bq})w}}=-2iz{\bq}$. For $n\in \Z \backslash \tpa{0}$ we have
		\begin{equation*}
		\Res_{w=\pi i n {\bq} \epsilon}\tpa{\frac{e^{-2izw}}{\sinh(w/{\bq})w}}=\frac{(-1)^n e^{2z\pi {\bq}\epsilon n}}{\pi i n}
		\end{equation*}
		so
		\begin{equation*}
		\int_{C}\frac{e^{-2izw}}{\sinh(w/{\bq})w}\;dw = (\epsilon-1)2\pi z {\bq} -2\log(1+e^{2z\pi {\bq} \epsilon}),
		\end{equation*}
		giving the first result.
		
		To prove equation \eqref{inv} let us choose the path $C=(-\infty,-\epsilon]\cup \epsilon \exp([\pi i,0])\cup [\epsilon,\infty)$ and let $\epsilon \to 0$. The rest is just calculations
		\begin{align*}
		\log \Phi_{\bq}(z)\Phi_{\bq}(-z)=\frac{1}{2}\int_{C}\frac{\cos(2wz)}{\sinh(w\bq)\sinh(w/\bq)w}dw
		\end{align*}
		Note that 
		\begin{equation*}
		\frac{1}{2}\int_{(-\infty,-\epsilon]}\frac{\cos(2wz)}{\sinh(w\bq)\sinh(w/\bq)w}dw=-\frac{1}{2}\int_{[\epsilon,\infty)}\frac{\cos(2wz)}{\sinh(w\bq)\sinh(w/\bq)w}dw.
		\end{equation*}
		i.e. it is enough to collect the half residue around $w=0$ of the remaining intergral
		\begin{align*}
		\frac{1}{2}\int_{\epsilon([\pi i,0])}\frac{\cos(2wz)}{\sinh(w\bq)\sinh(w/\bq)w} dw &= \frac{\pi i}{2}\Res_{w=0}\frac{\cos(2wz)}{\sinh(w\bq)\sinh(w/\bq)w}\\
		&= \frac{\pi i}{2}\left( \frac{b^2+b^{-2}}{6}+2z^2 \right)\\
		&=e^{-\pi i(1+2c_{\bq}^2)/6}e^{\pi i z^2}.
		\end{align*}
	\end{proof}
	
	\subsubsection*{Zeros and poles}
	The functional equation \eqref{feq} shows that $\Phi_{\bq}(z),$ which in its initial domain of definition has no zeroes and poles, extends (for fixed $\bq$ with $\Im  \bq^2 > 0$) to a meromorphic function in the variable $z$ to the entire complex plane with essential singularity at infinity and with characteristic properties:
	\begin{equation}
	\left(\Phi_{\bq}(z)\right)^{\pm 1}=0 \iff z=\mp (c_{\bq} + m i \bq + n i \bq). 
	\end{equation} 
	The behaviour at infinity depends on the direction along which the limit is taken 
	\begin{align}
	\Phi_{\bq}(z)\big\vert_{\abs{z}\to \infty} \approx
	\begin{cases}		
	1 & \abs{\arg{z}}>\frac{\pi}{2}+\arg{\bq},\\
	\zeta_{inv}^{-1}e^{\pi i z^2} & \abs{\arg{z}}<\frac{\pi}{2}-\arg{\bq} \\
	\frac{(\tilde{q}^2,\tilde{q}^2)_{\infty}}{\Theta(i \bq^{-1}z;-\bq^{-2})} & \abs{\arg{z-\frac{\pi}{2}}}<\arg{\bq} \\
	\frac{\Theta(i \bq z;\bq^{2})}{(q^2;q^2)_{\infty}} & \abs{\arg{z+\frac{\pi}{2}}}<\arg{\bq}
	\end{cases}
	\end{align}
	where 
	\begin{align}
	\Theta(z;\tau)\equiv \sum_{n\in \Z} e^{\pi i \tau n^2 + 2\pi i zn}, \quad \Im \:\tau > 0.
	\end{align}
	\subsubsection*{Unitarity} When $\bq$ is real or on the unit circle 
	\begin{align}\label{eq:unit}
	(1-\abs{\bq}) \Im \bq = 0\quad  \Rightarrow \quad  \overline{\Phi_{\bq}(z)} = \frac{1}{\Phi_{\bq}(\overline{z})}.
	\end{align}
	\subsubsection*{Quantum Pentagon Identity}
	In terms of specifically normalised sefladjoint Heisenberg momentum and position operators acting as unbounded operators on $L^2(\R)$ by the formulae
	\begin{align*}
	\pos f(x) = xf(x), \quad \mom f(x) = \frac{1}{2\pi i}f(x),
	\end{align*}
	the following pentagon identity for unitary operators is satisfied \cite{MR1264393}
	\begin{equation}\label{eq:fiveterm}
	\Phi_{\bq}(\mom)\Phi_{\bq}(\pos)=\Phi_{\bq}(\pos)\Phi_{\bq}(\mom+\pos)\Phi_{\bq}(\mom).
	\end{equation} 
	
	\subsubsection*{Fourier transformation formulae for Faddeev's quantum dilogarithm}
	The quantum pentagon identity \eqref{eq:fiveterm} is equivalent to the integral identity
	\begin{align}
	\int_{\R+i\varepsilon}\frac{\Phi_{\bq}(x+u)}{\Phi_{\bq}(x-c_{\bq})}e^{-2\pi i wx}\: dx = \frac{\Phi_{\bq}(u)\Phi_{\bq}(c_{\bq}-w)}{\Phi_{\bq}(u-w)}e^{\frac{\pi i }{	12}(1-4c_{\bq}^2)},
	\end{align}
	where $\Im \bq^2 > 0.$
	From here we get the Fourier transformation formula for the quantum dilogarithm formally sending $u\to -\infty$ by the use of \eqref{frac} and \eqref{eq:unit}
	\begin{equation}\label{eq:Fourier}
	\int_{\R+i\varepsilon} {\Phi_{\bq}(x+c_{\bq})} {e^{2\pi i wx}} = \frac{1}{\Phi_{\bq}(-w-c_{\bq})}e^{-\frac{\pi i }{12}(1-4c_{\bq}^2)}.
	\end{equation}
	
	\subsubsection*{Quasi-classical limit of Faddeev's quantum dilogarithm}
	\begin{prop}
		For fixed $x$ and $\bq \to 0$ we have the following asymptotic expansion
		\begin{equation}\label{expan}
		\log\Phi_{\bq}\pa{\frac{x}{2\pi \bq}} = \sum_{n=0}^{\infty}(2\pi i \bq)^{2n-1}\frac{B_{2n}(1/2)}{(2n)!}\frac{\partial^{2n}\Li_2(-e^x)}{\partial x^{2n}},
		\end{equation}
		where $B_{2n}(1/2)$ are the Bernoulli polynomials $B_{2n}$ evaluated at $1/2$.
	\end{prop}
	\begin{proof}
		From \eqref{feq} we have that
		\begin{align*}
		\log\pa{\frac{\Phi_{\bq}\pa{\frac{x-i\pi\bq^2}{2\pi \bq}}}{\Phi_{\bq}\pa{\frac{x+i\pi\bq^2}{2\pi \bq}}}}=\log(1+e^x).
		\end{align*}
		The left hand side yields 
		\begin{equation*}
		\log\Phi_{\bq}\pa{\frac{x-i\pi \bq^2}{2\pi \bq}}-\log\Phi_{\bq}\pa{\frac{x+i\pi \bq^2}{2\pi \bq}} =-2\sinh(i\pi b^2 \partial/\partial x) \log\Phi_{\bq}\pa{\frac{x}{2\pi \bq}},
		\end{equation*}
		where we have used the fact that 
		\begin{equation*}
		f(x+y)=e^{y\frac{\partial}{\partial x}}\pa{f}(x),
		\end{equation*}
		which is just the Taylor expansion of $f$ around $x$. 
		While the right hand side can be written in the following manner
		\begin{equation*}
		\log(1+e^x)=\frac{\partial}{\partial x}\int_{-\infty}^x \log(1+e^z)\;dz = - \frac{\partial}{\partial x} \Li_{2}(-e^x).
		\end{equation*}
		Using the expansion 
		\begin{equation*}
		\frac{z}{\sinh(z)}=\sum_{n=0}^{\infty}B_{2n}(1/2)\frac{(2z)^{2n}}{(2n)!} 
		\end{equation*}
		gives exactly \eqref{expan}.
	\end{proof}
	\begin{cor}\label{eq:asym}
		For fixed $x$ and $\bq \to 0$ one has 
		\begin{equation}
		\Phi_{\bq}\pa{\frac{x}{2\pi \bq}}=\exp\pa{\frac{1}{2\pi i b^2} \Li_{2} (-e^x)}\pa{1+O(b^2)}.
		\end{equation}
	\end{cor}
	\section{The Tetrahedral Operator}\label{TOP}
	In order to prove Proposition \ref{prop:Kashaev} we make use of the following formulae
	\begin{lemma}\label{Lemmatop}
		Suppose $x$ and $y$ are operators in an algebra such that $$z=[x,y], [x,z]=0.$$ Then
		\begin{align*}
		f(x)y &=yf(x)+zf'(x) \\
		e^x f(x) &= f(y+z)e^x,
		\end{align*}
		for every power series such that $f(x),f'(x)$ and $f(y+z)$ can be defined in the same operator algebra.
	\end{lemma}
	\begin{proof}
		Let $f(x)=\sum_{j=0}^{\infty} a_jx^j.$ Then,
		\begin{align*}
		[f(x),y]=\sum_{j=0}^{\infty} a_j [x^j,y]=\sum_{j=0}^{\infty} a_j \sum_{k=0}^{j-1}x^k[x,y]x^{j-k-1}=\sum_{j=0}^{\infty} a_j j zx^{j-1} = zf'(x).
		\end{align*}
		which shows the first equation. The second equation follows from this when we set $f(x)=e^xy^{l-1}$
		\begin{align*}
		e^{x}y^l=ye^x y^{l-1} = (y-z)e^{x}y^{l-l}=\dots = (y+z)^l e^x,
		\end{align*}
		and from here we get that
		\begin{equation*}
		e^xf(y)=e^x \sum_{j=0}^{\infty} a_j y^j = \sum_{j=0}^{\infty}a_j (y+z)^je^x = f(y+z)e^x.
		\end{equation*}
	\end{proof}
	\begin{proof}[Proof of Proposition \ref{prop:Kashaev}]
		The equations in \eqref{tetraop} follows from the system of equations 
		\begin{align*}
		&\pto \pos_1 = (\pos_1+\pos_2)\pto, \\
		&\pto (\mom_1+\mom_2) =(\mom_1+\pos_2)\pto, \\
		&\pto (\mom_1+\pos_2) = (\mom_1+\pos_2)\pto, \\
		&\pto e^{2\pi \bq \mom_1} = (e^{2\pi \bq\mom_1}+e^{2\pi \bq(\pos_1+\mom_2)})\pto,
		\end{align*}
		where $\pto=e^{2\pi i \mom_1\pos_2}\psi(\pos_1-\pos_2+\mom_2).$ We prove them one by one below using Lemma \ref{Lemmatop}.
		\begin{align*}
		\pto \pos_1 &= e^{2\pi i \mom_1\pos_2}\psi(\pos_1+\mom_2-\pos_2)\pos_1 \\
		&= e^{2\pi i \mom_1\pos_2} \pos_1 \psi(\pos_1+\mom_2-\pos_2) \\
		&= (\pos_1 e^{2\pi i \mom_1\pos_2}+\pos_2e^{2\pi i \mom_1\pos_2})\psi(\pos_1+\mom_2-\pos_2) \\
		&=(\pos_1+\pos_2)\pto.
		\end{align*}
		\begin{align*}
		\pto (\mom_1+\mom_2)&= e^{2\pi i \mom_1\pos_2}\psi(\pos_1+\mom_2-\pos_2)(\mom_1+\mom_2) \\
		&=e^{2\pi i \mom_1+\pos_2}(\mom_1+\mom_2)\psi(\pos_1+\mom_2-\pos_2) \\
		&=\left\{\mom_1e^{2\pi i \mom_1\pos_2} + \mom_2 e^{2\pi i \mom_1\pos_2} -  \mom_1e^{2\pi i \mom_1\pos_2} \right\}\psi(\pos_1+\mom_2-\pos_2) \\
		&= \mom_2\pto,
		\end{align*}
		where the second equality is true since $[\pos_1+\mom_2-\pos_2 , \mom_1+\mom_2]=0.$
		\begin{align*}
		\pto (\mom_1+\pos_2) &=e^{2\pi i \mom_1\pos_2}\psi(\pos_1+\mom_2-\pos_2)(\mom_1+\pos_2) \\
		&= e^{2\pi i \mom_1\pos_2}(\mom_1+\pos_2)\psi(\pos_1+\mom_2-\pos_2) \\
		&=(\mom_1+\pos_2)\pto,
		\end{align*}
		where second equality is true since $[q_1+p_2-q_2,p_1+q_2]=0.$
		\begin{align*}
		\pto e^{2\pi \bq \mom_1} &= e^{2\pi i \mom_1\pos_2}\psi(\pos_1-\pos_2+\mom_2) e^{2\pi \bq \mom_1} \\
		&= \psi(\pos_1-\mom_1+\mom_2)e^{2\pi i \mom_1\pos_2} e^{2\pi \bq \mom_1} \\
		&= \psi(\pos_1-\mom_1+\mom_2)e^{2\pi \bq \mom_1}e^{2\pi i \mom_1\pos_2} \\
		&= e^{2\pi \bq \mom_1}\psi(\pos_1-\mom_1+\mom_2+i\bq)e^{2\pi i \mom_1\pos_2} \\
		&= e^{2\pi \bq \mom_1}\left(1+e^{2\pi \bq (\pos_1-\mom_1+\mom_2+\frac{i\bq}{2})}\right)\psi(\pos_1-\mom_1+\mom_2)e^{2\pi i \mom_1\pos_2} \\
		&= \left(e^{2\pi \bq \mom_1}+ e^{2\pi i \bq (\pos_1 +\mom_2)}\right)\pto,
		\end{align*}
		where in the last equality we use the \textit{Baker--Campbell--Hausdorff formula}.
	\end{proof}
\end{appendices}


\noindent 
J{\o}rgen Ellegaard Andersen and Jens-Jakob Kratmann Nissen \\
Center for Quantum Geometry of Moduli Spaces\\
Department of Mathematics\\
University of Aarhus\\
DK-8000, Denmark


\label{lastpage}


\begin{thebibliography}{11}
	
\bibitem[A1]{A1} J.E. Andersen, Deformation Quantization
and Geometric Quantization of Abelian Moduli Spaces.,
{\em Comm. of Math. Phys.} {\bf 255}:727--745, 2005.

\bibitem[A2]{A2} J.~E. Andersen. 
\newblock  Asymptotic faithfulness of the quantum $SU(n)$
representations of the mapping class groups.  
\newblock {\em Annals of Mathematics.} {\bf
	163}:347--368, 2006.
	
\bibitem[AH]{AH1} J.E. Andersen \& S.K. Hansen.
\newblock Asymptotics of the quantum
invariants for surgeries on the figure 8 knot.
\newblock {\em Journal of Knot
	theory and its Ramifications,} {\bf 15}:479--548, 2006.
	
\bibitem[AGr1]{AGr1} J.E. Andersen \& J. Grove.
\newblock  Automorphism Fixed Points in the Moduli Space of
Semi-Stable Bundles.
\newblock  {\em The Quarterly Journal of
	Mathematics,} {\bf 57}:1--35, 2006.

\bibitem[AMU]{AMU} J.E. Andersen, G. Masbaum \& K. Ueno. 
\newblock Topological Quantum Field Theory
and the Nielsen-Thurston classification of $M(0,4)$. 
\newblock {\em Math. Proc.
	Cambridge Philos. Soc.} {\bf 141}:477--488, 2006.



\bibitem[AU1]{AU1} J.~E. Andersen \& K. Ueno. 
\newblock Abelian Conformal Field theories and
Determinant Bundles.
\newblock {\em International Journal of Mathematics.}  {\bf 18}:919--993, 2007.



\bibitem[AU2]{AU2} J.~E. Andersen \& K. Ueno, 
\newblock Constructing modular functors
from conformal field theories.  
\newblock{\em  Journal of Knot theory and its
Ramifications.} {\bf 16}(2):127--202, 2007.

\bibitem[A3]{A3} J.~E. Andersen.
\newblock The Nielsen-Thurston classification of mapping
classes is determined by TQFT. math.QA/0605036.  
 \newblock{\em J. Math. Kyoto Univ.} {\bf 48}(2):323--338,  2008.

\bibitem[A4]{A4} J.E. Andersen.
\newblock Asymptotics of the Hilbert-Schmidt Norm of Curve Operators in
TQFT.
\newblock {\em Letters in Mathematical Physics} {\bf 91}:205--214, 2010.


\bibitem[A5]{A5} J.E. Andersen.
\newblock Toeplitz operators and Hitchin's projectively flat
connection.
\newblock in {\it The many facets of geometry: A tribute to
	Nigel Hitchin}, Edited by O. Garc\'{i}a-Prada, Jean Pierre Bourguignon, Simon Salamon, 177--209, Oxford Univ. Press, Oxford, 2010.
	
\bibitem[AB1]{AB1} J.E. Andersen \& J. L. Blaavand, Asymptotics of Toeplitz operators and applications in TQFT, {\em Traveaux Math\'{e}matiques}, {\bf 19}:167--201, 2011.

\bibitem[AGa1]{AGa1} J.E. Andersen \& N.L. Gammelgaard.
\newblock  Hitchin's Projectively Flat Connection, Toeplitz Operators and the Asymptotic Expansion of TQFT Curve Operators.  
\newblock {\em Grassmannians, Moduli Spaces and Vector Bundles}, 1--24, {\em Clay Math. Proc.,} 14, Amer. Math. Soc., Providence, RI, 2011.


\bibitem[AU3]{AU3} J.~E. Andersen \& K. Ueno.  
\newblock Modular functors are determined by
their genus zero data.
\newblock{\em  Quantum Topology.} {\bf 3}:255--291, 2012.

\bibitem[A6]{A6}
J.~E. Andersen. 
\newblock Hitchin's connection, {T}oeplitz operators, and symmetry invariant
deformation quantization.
\newblock {\em Quantum Topol.} {\bf 3}(3-4):293--325, 2012.



\bibitem[AGL]{AGL} J.E. Andersen, N.L. Gammelgaard \& M.R. Lauridsen,
Hitchin's Connection in Metaplectic Quantization, {\em Quantum Topology} {\bf 3}:327--357, 2012.

\bibitem[AHi]{AHi} J.~E. Andersen \& B. Himpel. 
\newblock The Witten-Reshetikhin-Turaev invariants of finite order mapping tori II
\newblock {\em Quantum Topology.} {\bf 3}:377--421, 2012.




\bibitem[A7]{A7} J.~E. Andersen.  
\newblock The Witten-Reshetikhin-Turaev invariants of finite order mapping tori I. 
\newblock {\em Journal f\"{u}r Reine und Angewandte Mathematik.}  {\bf 681}:1--38, 2013. 



\bibitem[AG]{AGa2}
J.~E. Andersen and N.~L. Gammelgaard.
\newblock {The Hitchin-Witten Connection and Complex Quantum Chern-Simons
	Theory}.
\newblock {\em arXiv:1409.1035}, 2014.




\bibitem[AK1]{AK1}
J.~E. Andersen and R. Kashaev. 
\newblock A {TQFT} from {Q}uantum {T}eichm{\"u}ller {T}heory.
\newblock {\em Comm. Math. Phys.} {\bf 330}(3):887--934, 2014.





\bibitem[AK1a]{AK1a} J.E. Andersen \& R.M.
\newblock Kashaev, Quantum Teichm\"{u}ller theory and TQFT.
\newblock {\em XVIIth International Congress on Mathematical Physics,} World Sci. Publ., Hackensack, NJ, 684--692, 2014. 


\bibitem[AK1b]{AK1b} J.E. Andersen \& R.M. Kashaev.
\newblock Faddeev's quantum dilogarithm and state-integrals on shaped triangulations.
\newblock  In {\it Mathematical Aspects of Quantum Field Theories}, Editors D. Calaque  and Thomas Strobl,{\em  Mathematical Physics Studies.} {\bf XXVIII}:133--152, 2015.

\bibitem[AHJMMc]{AHJMMc} J.~E. Andersen, B.~Himpel, S.~F. J{\o}rgensen, J.~Martens and B.~McLellan,
\newblock "The {W}itten-{R}eshetikhin-{T}uraev invariant for links in finite
  order mapping tori I", {\em Advances in Mathematics}, {\bf 304}:131--178, 2017.

\bibitem[AK2]{AK2}
J.~E. Andersen and R. Kashaev.
\newblock A new formulation of the {T}eichm{\"u}ller TQFT.
\newblock {\em arXiv:1305.4291}, 2013.

\bibitem[AK3]{AK3}
J.~E. Andersen and R. Kashaev.
\newblock {Complex Quantum Chern-Simons}.
\newblock {\em arXiv:1409.1208}, 2014.

\bibitem[AM]{AM} J. E. Andersen and S. Marzioni, Level $N$ Teichm\"{u}ller TQFT and Complex Chern--Simons Theory, Preprint, 2016. 
	


\bibitem[AU4]{AU4} J.~E. Andersen \& K. Ueno.  
\newblock Construction of the Witten-Reshetikhin-Turaev TQFT from
conformal field theory.
\newblock{\em Invent. Math.} {\bf 201}(2):519--559, 2015.




\bibitem[AE]{AE} J.E. Andersen \& J.K. Egsgaard, The equivalence of the Hitchin connection and the {Knizhnik--Zamolodchikov} connection,  In preparation.


\bibitem[At]{MR1001453}
M.~Atiyah.
\newblock Topological quantum field theories.
\newblock {\em Inst. Hautes \'Etudes Sci. Publ. Math.}, {\bf 68}:175--186,
1988.

\bibitem[BB]{BB}
S. Baseilhac and R. Benedetti. 
\newblock Quantum hyperbolic geometry.
\newblock {\em Algebr. Geom. Topol.} {\bf 7}:845--917, 2007.


\bibitem[BHMV1]{BHMV1} C. Blanchet, N. Habegger, G. Masbaum \&
P. Vogel.  
\newblock Three-manifold invariants derived from the Kauffman Bracket.
\newblock {\em Topology.} {\bf 31}:685--699, 1992.


\bibitem[BHMV2]{BHMV2} C. Blanchet, N. Habegger, G. Masbaum \&
P. Vogel.  
\newblock Topological Quantum Field Theories derived from the
Kauffman bracket. 
\newblock {\em Topology.} {\bf 34}:883--927, 1995.


\bibitem[DFM]{DFM}
R.~Dijkgraaf, H.~Fuji, and M.~Manabe.
\newblock The volume conjecture, perturbative knot invariants, and recursion
relations for topological strings.
\newblock {\em Nuclear Phys. B}, {\bf 849}(1):166--211, 2011.

\bibitem[Di]{Di} T. Dimofte, 
\newblock Complex Chern-Simons theory at level k via the 3d-3d correspondence. 
\newblock {\em Comm. Math. Phys.} {\bf 339}(2):619--662, 2015.

\bibitem[DGLZ]{DGLZ}
T.~Dimofte, S.~Gukov, J.~Lenells, and D.~Zagier.
\newblock Exact results for perturbative {C}hern-{S}imons theory with complex
gauge group.
\newblock {\em Commun. Number Theory Phys.}, {\bf 3}(2):363--443, 2009.

\bibitem[FK]{MR1264393}
L.~D. Faddeev and R.~M. Kashaev.  
\newblock Quantum dilogarithm.
\newblock {\em Modern Phys. Lett. A.} {\bf 9}(5):427--434, 1994.

\bibitem[F]{F}
L.~D. Faddeev.
\newblock Discrete {H}eisenberg-{W}eyl group and modular group.
\newblock {\em Lett. Math. Phys.}  {\bf 34}(3):249--254, 1995.

\bibitem[FC]{CF}
V.~V. Fock and L.~O. Chekhov.
\newblock Quantum {T}eichm\"uller spaces.
\newblock {\em Teoret. Mat. Fiz.} {\bf 120}(3):511--528, 1999.

\bibitem[GKT]{GKT}
N.~Geer, R.~Kashaev, and V.~Turaev.
\newblock Tetrahedral forms in monoidal categories and 3-manifold invariants.
\newblock {\em J. Reine Angew. Math.} {\bf 673}:69--123, 2012.

\bibitem[H1]{HI1}
K. Hikami. 
\newblock Hyperbolicity of partition function and quantum gravity.
\newblock {\em Nuclear Phys. B.} {\bf 616}(3):537--548, 2001.

\bibitem[H2]{HI2}
K. Hikami. 
\newblock Generalized volume conjecture and the {$A$}-polynomials: the
{N}eumann-{Z}agier potential function as a classical limit of the partition
function.
\newblock {\em J. Geom. Phys.} {\bf 57}(9):1895--1940, 2007.


\bibitem[K1]{K6j}
R.~M. Kashaev.
\newblock Quantum dilogarithm as a {$6j$}-symbol.
\newblock {\em Modern Phys. Lett. A}, {\bf 9}(40):3757--3768, 1994.

\bibitem[K2]{K1}
R.~M. Kashaev.
\newblock Quantization of {T}eichm\"uller spaces and the quantum dilogarithm.
\newblock {\em Lett. Math. Phys.}, {\bf 4}3(2):105--115, 1998.

\bibitem[K3]{K2}
R.~M. Kashaev.
\newblock The {L}iouville central charge in quantum {T}eichm\"uller theory.
\newblock {\em Tr. Mat. Inst. Steklova}, {\bf 226}(Mat. Fiz. Probl. Kvantovoi Teor.
Polya):72--81, 1999.

\bibitem[K4]{K3}
R.~M. Kashaev.
\newblock On the spectrum of {D}ehn twists in quantum {T}eichm\"uller theory.
\newblock In {\em Physics and combinatorics, 2000 ({N}agoya)}, pages 63--81.
World Sci. Publ., River Edge, NJ, 2001.

\bibitem[KLV]{arXiv:1210.8393}
R.~M. Kashaev, F.~Luo, and G.~Vartanov.
\newblock {A TQFT of Turaev-Viro type on shaped triangulations}, 2012.

\bibitem[L]{La} Y. Laszlo.  
\newblock Hitchin's and WZW connections are the
same.
\newblock {\em J. Diff. Geom.} {\bf 49}(3):547--576, 1998.

\bibitem[N]{JJKN}
J.-J.~K. Nissen.
\newblock {The {A}ndersen--{K}ashaev {TQFT}}.
\newblock PhD thesis, Aarhus University, 2014.

\bibitem[P]{P1}
R.~C. Penner.
\newblock The decorated {T}eichm\"uller space of punctured surfaces.
\newblock {\em Comm. Math. Phys.}, {\bf 113}(2):299--339, 1987.

\bibitem[RT1]{RT1} N. Reshetikhin \& V. Turaev.
\newblock Ribbon graphs and
their invariants derived from quantum groups
\newblock {\em Comm. Math. Phys.}
{\bf 127}:1--26, 1990.

\bibitem[RT2]{RT2} N. Reshetikhin \& V. Turaev.
\newblock Invariants of
$3$-manifolds via link polynomials and quantum groups
 \newblock {\em Invent. Math.} {\bf 103}:547--597,  1991.

\bibitem[S]{MR981378}
G.~B. Segal.
\newblock The definition of conformal field theory.
\newblock In {\em Differential geometrical methods in theoretical physics
	({C}omo, 1987)}, volume 250 of {\em NATO Adv. Sci. Inst. Ser. C Math. Phys.
	Sci.}, pages 165--171. Kluwer Acad. Publ., Dordrecht, 1988.

\bibitem[T]{Turaev}
V.~G. Turaev.
\newblock {\em Quantum invariants of knots and 3-manifolds}, volume~18 of {\em
	de Gruyter Studies in Mathematics}.
\newblock Walter de Gruyter \& Co., Berlin, 1994.

\bibitem[TV]{TV}
V.~G. Turaev and O.~Y. Viro.
\newblock State sum invariants of {$3$}-manifolds and quantum {$6j$}-symbols.
\newblock {\em Topology}, {\bf 31}(4):865--902, 1992.

\bibitem[W]{W}
E.~Witten.
\newblock Topological quantum field theory.
\newblock {\em Comm. Math. Phys.}, {\bf 117}(3):353--386, 1988.
	
	
\end{thebibliography}
\end{document}